\documentclass[a4paper,dvipsnames]{article}
\usepackage[utf8]{inputenc}

\usepackage{amssymb}
\usepackage{amscd}
\usepackage{amsmath}
\usepackage{enumerate}
\usepackage[shortlabels]{enumitem}
\usepackage{hyperref}
\usepackage{tikz-cd}
\usetikzlibrary{cd}
\usepackage{amsthm}
\theoremstyle{plain}

\newtheorem{thm}{Theorem}[section]
\newtheorem{prop}[thm]{Proposition}
\newtheorem{lem}[thm]{Lemma}
\newtheorem{cor}[thm]{Corollary}
\usepackage{geometry}
\theoremstyle{definition}
\newtheorem{df}[thm]{Definition}
\newtheorem{notn}[thm]{Notation}
\newtheorem*{notn*}{Notation}

\newtheorem{remark}[thm]{Remark}

\newtheorem{example}[thm]{Example}
\newtheorem{observation}[thm]{Observations}
\newtheorem{alg}{Algorithm}

\newtheorem{innercustomthm}{Theorem}
\newenvironment{customthm}[1]
  {\renewcommand\theinnercustomthm{#1}\innercustomthm}
  {\endinnercustomthm}

\usepackage{xcolor}
\usepackage{color}
\usepackage{marvosym}

\newcommand{\Z}{\mathbb{Z}}
\newcommand{\Q}{\mathbb{Q}}
\newcommand{\F}{\mathbb{F}}

\newcommand{\End}{\mathrm{End}}
\newcommand{\Hom}{\mathrm{Hom}}
\newcommand{\Cl}{\mathrm{Cl}}

\newcommand{\OO}{\mathcal{O}}

\newcommand{\val}{\mathrm{val}}

\newcommand{\frp}{\mathfrak{p}}
\newcommand{\frP}{\mathfrak{P}}
\newcommand{\frf}{\mathfrak{f}}
\newcommand{\frl}{\mathfrak{l}}
\newcommand{\cW}{\mathcal{W}} 

\usepackage[hyphenbreaks]{}  

\title{Abelian varieties over finite fields with commutative endomorphism algebra: theory and algorithms} 
\author{Jonas Bergstr\"om, Valentijn Karemaker, and Stefano Marseglia}

\date{}

\AtEndDocument{\bigskip{\footnotesize\noindent
    \textsc{
    Matematiska institutionen, Stockholms Universitet, SE-106 91 Stockholm, Sweden
    }\\
    \emph{Email address}: 
    \text{jonasb@math.su.se}\\
    \\
    \textsc{
    Mathematisch Instituut, Universiteit Utrecht, Postbus 80010, 3508 TA Utrecht,
    The Netherlands
    }\\
    \emph{Email address}: 
    \text{v.z.karemaker@uu.nl}\\
    \\
    \textsc{
    Mathematisch Instituut, Universiteit Utrecht, Postbus 80010, 3508 TA Utrecht,
    The Netherlands
    }
    \\
    \textsc{
    Laboratoire GAATI, Université de la Polynésie française, BP 6570 – 98702 Faaa, Polynésie
française
    }
    \\
    \textsc{
    Laboratoire Jean Alexandre Dieudonné, Université Côte Azur, 06108 Nice Cedex 2, France
    }\\
    \emph{Email address}: 
    \text{stefano.marseglia@univ-cotedazur.fr}
}}

\begin{document}
 
\maketitle
  
\begin{abstract}
We give a categorical description of all abelian varieties with commutative endomorphism ring over a finite field with $q=p^a$ elements in a fixed isogeny class in terms of pairs consisting of a fractional $\mathbb Z[\pi,q/\pi]$-ideal and a fractional $W\otimes_{\mathbb Z_p} \mathbb Z_p[\pi,q/\pi]$-ideal, with $\pi$ the Frobenius endomorphism and $W$ the ring of integers in an unramified extension of $\mathbb Q_p$ of degree $a$. 
The latter ideal should be compatible at $p$ with the former and stable under the action of a semilinear Frobenius (and Verschiebung) operator; it will be the Dieudonn\'e module of the corresponding abelian variety. 
Using this categorical description we create effective algorithms to compute isomorphism classes of these objects and we produce many new examples exhibiting exotic patterns. 
\end{abstract}

\section{Introduction}
In order to classify abelian varieties over a finite field up to isomorphism one needs a way to represent them. This is a notoriously hard problem. 
Indeed, representing abelian varieties over any field with equations becomes impractical already in dimension two.  
As a first step, one can consider the classification and representation problem up to isogeny. Over finite fields, this problem has been completely solved by Honda-Tate theory in terms of the characteristic polynomial of the Frobenius endomorphism of the abelian variety, see~\cite{Tate66,Honda68}.
The isogeny classes can also be computed effectively, see e.g.~\cite{DupuyKedlayaRoeVincent2020}. 
In this article we consider the representation and classification problem up to isomorphism for any fixed isogeny class of abelian varieties over finite fields with commutative endomorphism algebra.

To state our results we now introduce some notation. Fix a prime power $q=p^a$ and let~$\pi$ be the Frobenius endomorphism of an abelian variety $X$ with commutative endomorphism algebra, defined over the finite field $\F_q$.  
Note that the condition of the endomorphism algebra being commutative is equivalent to the condition that the characteristic polynomial of the Frobenius endomorphism is square-free. 
Let $\mathcal{A}_{\pi}$ be the category of abelian varieties isogenous to $X$. Morphisms in this category will be homomorphisms defined over $\F_q$.

\subsubsection*{Main contribution}
Put $R=\Z[\pi,q/\pi] \subseteq E=\Q[\pi]$. 
For any rational prime $\ell$, including $\ell=p$, let $E_\ell = E\otimes_\Q \Q_\ell$ and  $R_\ell = R\otimes_\Z \Z_\ell$, and let $i_{\ell}$ be the injection $E \to E_{\ell}$.
Let $L$ be an unramified extension of $\Q_p$ of degree $a$, with $W$ its maximal $\Z_p$-order, and put $W_R=W \otimes_{\Z_p} R_p \subseteq A=L\otimes_{\Q_p}E_p$. 
The algebra $A$ comes with an action of $\sigma$, the Frobenius map of~$L$ over $\Q_p$, and an embedding $\Delta\colon E_p \to A$. For more details on these definitions see Section~\ref{ssec:setupXp}. 
Fix an additive map $F\colon A\to A$ such that $F \lambda=\lambda^{\sigma}F$ for all $\lambda \in L$ and such that $F^a$ coincides with the multiplication by $\Delta(\pi)$.
Put $V=pF^{-1}$. 
With a $W_R\{F,V\}$-ideal we will mean a fractional $W_R$-ideal in $A$ which is stable under the action of $F$ and $V$. 

Now let $\mathcal C_{\pi}$ be the category whose objects are pairs $(I,M)$, where $I$ is a fractional $R$-ideal in $E$ and $M$ is a $W_R\{F,V\}$-ideal such that $\Delta^{-1}(M)=i_p(I)R_p$. 
The homomorphisms between objects $(I,M)$ and $(J,N)$ in $\mathcal C_{\pi}$ are the elements $\alpha \in E$ such that $\alpha I \subseteq J$ and $\Delta(i_p(\alpha)) M \subseteq N$. 

Our first main result can now be stated as follows.

\begin{customthm}{5.2} There is an equivalence of categories $\Psi\colon\mathcal C_{\pi} \to \mathcal A_{\pi}$.
\end{customthm}

This result is closely related to \cite[Theorem 2.1]{xue2017counting}.
The idea goes back (at least) to Waterhouse \cite[Theorem 5.1]{Wat69} and is based on Tate's theorem (see \cite[Main theorem]{Tate66}). 
More precisely, the functor $\Psi$ sends an object $(I,M) \in \mathcal{C}_{\pi}$ to an abelian variety $X_{(I,M)} \in \mathcal{A}_{\pi}$, which admits an isogeny $\varphi: X_{(I,M)} \to X$ for a fixed $X \in \mathcal{A}_{\pi}$. Any isogeny is determined by its (finite) kernel, and the kernel of~$\varphi$ is the subgroup whose $\ell$-primary part for $\ell \neq p$ is (a suitable multiple of) $i_{\ell}(I) \subseteq T_{\ell}(X)$, where $T_{\ell}(X)$ denotes the $\ell$-adic Tate module of $X$, and whose $p$-primary part is (the same multiple of)~$M \subseteq M(X)$, where $M(X)$ denotes the Dieudonn{\'e} module of $X$; existence of this~$\varphi$ is guaranteed by Tate's theorem. 
In other words, in a pair $(I,M)$, the ideal $M$ determines the Dieudonn\'e module (which is equivalent to the $p$-divisible group) of the abelian variety. The ideal $I$ encodes the Tate modules of the abelian variety for all $\ell \neq p$, partial information about the Dieudonn\'e module, as well as ``global information'' which is determined by an element of the class group of the endomorphism ring of the abelian variety. Indeed, the endomorphism ring of the corresponding abelian variety is $\Delta^{-1}(\{x\in A : xM\subseteq M\})$ locally at $p$ and $\{x\in E : xI\subseteq I\}$ locally at every other prime.

We believe that a statement analogous to Theorem 5.2 should also hold for isogeny classes with non-commutative endomorphism algebra. But, we expect that in this more general context it would be harder to effectively compute representations of the $p$-Frobenius and Verschiebung. 
Effectiveness of the results being one of the main goals of this work, we decided to restrict our investigation to the commutative case, and leave the general one for future work.

\subsubsection*{Effectiveness of our result} 
The equivalence provided in Theorem 5.2 holds for all square-free isogeny classes $\mathcal{A}_{\pi}$. For each $\mathcal{A}_{\pi}$, the target category $\mathcal{C}_{\pi}$ consists of modules over commutative rings, which are amenable to concrete computations.
So using the equivalence, the second main result of this article is that we have created effective algorithms to compute isomorphism classes of abelian varieties over a finite field with commutative endomorphism algebra. 

Our most important contribution here is Algorithm~\ref{alg:FVclasses}: it computes isomorphism classes of Dieudonn\'e modules even when the endomorphism ring of the local-local part is not maximal. 
To our knowledge, this is the first algorithm that can compute isomorphism classes of Dieudonn\'e modules with non-maximal endomorphism ring.
To do so, we leverage the fact that the local-local parts of Dieudonn\'e modules with maximal endomorphism ring have been classified by Waterhouse, see \cite[Theorem 5.1]{Wat69}. 
Interestingly, in all isogeny classes of abelian varieties whose isomorphism classes we have computed, which are not defined over a prime field, nor ordinary, nor almost ordinary, it turns out that there are examples of isomorphism classes of abelian varieties with Dieudonn\'e modules whose local-local part has a non-maximal endomorphism ring, see Observation~\ref{obs}. 
The implementation of the algorithms in Magma \cite{Magma} is available at \cite{IsomClAbVarFqCommEndAlg}.

The examples we compute in Section~\ref{sec:examples} show that the endomorphism rings that appear in a given isogeny class of abelian varieties (which are not defined over a prime field, nor ordinary, nor almost ordinary) behave quite wildly. 
To make this concrete, let $\mathcal{S}$ be the set of overorders of $R = \mathbb{Z}[\pi,q/\pi]$ and let $\mathcal{E}$ be the subset of $\mathcal{S}$ consisting of orders $T$ such that there exists an abelian variety $X \in \mathcal{A}_\pi$ with $\End(X)=T$, and let $\OO_E$ be the maximal order of $E=\Q[\pi]$. 
Consider the following three statements which are true in the ordinary and almost-ordinary cases (meaning of $p$-rank $g$ and $g-1$, respectively), as well as over a prime field (meaning of cardinality $q=p$):
\begin{enumerate}
    \item For every $S\in \mathcal{E}$ and $T\in \mathcal{S}$, if $S\subseteq T$ then $T \in \mathcal{E}$.
    \item The order $S = \cap_{T \in \mathcal{E}} T$ is in $\mathcal{E}$.
    \item For every $S$ in $\mathcal{E}$, $n(\OO_E)$ divides $n(S)$, where $n(S)$ (resp.~$n(\OO_E)$) is the number of isomorphism classes of abelian varieties in $\mathcal{A}_\pi$ with endomorphism ring $S$ (resp.~$\OO_E$).
\end{enumerate}
In Section~\ref{sec:examples}, among other things, we give examples that violate all three of these statements.
Our explicit knowledge of the Dieudonn{\'e} modules also allows us to compute the $a$-numbers of the abelian varieties and study their distribution in the isogeny class.

\subsubsection*{Comparison with previous results}
There are several equivalence of categories in the literature that are similar to ours, some of which consider very large subcategories of the category of abelian varieties over a finite field $\F_q$.

Recall that each Dieudonn\'e module of an abelian variety over a finite field splits into three parts: the \'etale part, its dual the multiplicative part, and the local-local part.
The difficulty in realizing categorical equivalences can be measured by the exponent $a$ of $q=p^a$ and the complexity of the local-local part.
The $p$-rank of an abelian variety measures how large each of these parts is.
For example, the $p$-rank is maximal (that is, it equals the dimension $g$ of the variety) if the local-local part is trivial.

In \cite{Del69}, Deligne uses canonical liftings to represent ordinary abelian varieties over a finite field in terms of finitely generated modules with a Frobenius-like endomorphism.
So in this case we have a close relation between abelian varieties over finite fields and CM-abelian varieties over the complex numbers (the canonical liftings). 

This is extended by Centeleghe and Stix in \cite{CentelegheStix15}, in line with the techniques of Waterhouse rather than using canonical liftings, to represent all abelian varieties over a finite prime field 
, again, in terms of finitely generated modules with a Frobenius-like endomorphism.
The change of technique is connected to the fact that canonical liftings do not exist in general if the abelian variety is not ordinary nor almost ordinary.

In our notation, the target of the Deligne and Centeleghe-Stix equivalences restricted to an isogeny class with commutative endomorphism ring determined by $\pi$ is the category of fractional $R$-ideals. 
Indeed, in the case of ordinary abelian varieties as well as the case of abelian varieties over prime fields, the Dieudonn\'e module plays no special role and behaves in the same way as the Tate modules for $\ell \neq p$.  In Proposition~\ref{prop-ordinary-ideals} we show analogous statements for the \'etale-local and local-\'etale parts of the Dieudonn\'e module, but the local-local part will behave quite differently as soon as we leave these realms.

An almost-ordinary abelian variety has a Dieudonn\'e module with nonzero local-local part, whose endomorphism ring is maximal, see \cite[Proposition 2.1]{OswalShankar19}. If the place of $E_p$ of slope $1/2$ (called the supersingular part) in an isogeny class of almost-ordinary abelian varieties is unramified, then there are two isomorphism classes of Dieudonn\'e modules; in the ramified case there is one. In the ramified case, Oswal and Shankar in \cite{OswalShankar19} use a canonical lifting to give a categorical equivalence between almost-ordinary abelian varieties and fractional $S$-ideals, where $S$ is the minimal overorder of $R$ that has maximal endomorphism ring at the supersingular part. In the unramified case, they use canonical liftings of different CM-types to distinguish between the two isomorphism classes of Dieudonn\'e modules; this gives rise to two disjoint equivalences of such abelian varieties with fractional $S$-ideals.

In \cite{BhatnagarFu}, Bhatnagar and Fu use the techniques of \cite{OswalShankar19} to give a similar description for abelian varieties over finite fields with maximal real multiplication and such that $p$ is split in the ring of integers of the maximal real extension.
As in the almost-ordinary case, the endomorphism ring of the local-local part of the Dieudonn\'e module of these abelian varieties is necessarily maximal, and \cite[Theorem 1.3]{BhatnagarFu} can be compared with \cite[Theorem 5.3]{Wat69}.  

Theorem~\ref{thm-equivalence} subsumes all of the above equivalences for the isogeny classes with commutative endomorphism algebra. 
In the almost-ordinary case, Theorem~\ref{thm-equivalence} also gives a description of morphisms between abelian varieties with non-isomorphic supersingular parts, which is not present in \cite{OswalShankar19}. Moreover, Theorem~\ref{thm-equivalence} also works for almost-ordinary abelian varieties in characteristic~$2$, while the results in \cite{OswalShankar19} do not. 

In \cite{CentelegheStix23}, Centeleghe and Stix give a categorical description of all abelian varieties over a fixed finite field $\F_q$.
This is achieved by gluing together infinitely many functors parametrized by sets $w$ of $q$-Weil numbers whose respective target is the category of $S_w$-modules that are free and finite rank over $\Z$, where $S_w$ is the endomorphism ring of a carefully chosen (``$w$-balanced'') abelian variety~$X_w$, cf.~\cite[Theorem 1.1]{CentelegheStix23}.
This description, while completely general, is less amenable to be turned into an effective algorithm -- which is one of the main aims of this article -- even when we restrict to an isogeny class with commutative endomorphism algebra.
One immediate issue is that, as soon as $X_w$ is not ordinary nor defined over a prime field, see \cite[Proposition~8.7]{CentelegheStix23}, the non-commutative ring $S_w$ is given as a fiber product of certain matrix algebras, see \cite[Theorem~4.9]{CentelegheStix23}, on which explicit representations of the Frobenius map are hard to compute.

Finally, let us mention that there are other categorical descriptions of abelian varieties in the literature, such as \cite{JKPRSBT18}, \cite{Kani11}, \cite{KNRR} and \cite{chiafu2012}. 

\subsubsection*{Comparison with previous algorithms}
Building on the previously cited literature, there are several algorithms to compute isomorphism classes of abelian varieties in a given isogeny class.
We focus on the case of isogeny classes of abelian varieties of dimension $>1$.
In \cite{MarAbVar18}, the third author designs an algorithm that uses Deligne's and Centeleghe-Stix' results to compute the isomorphism classes in an isogeny class, when this is ordinary or defined over a prime field, and with commutative endomorphism algebra.
As written above, this boils down to computing all isomorphism classes of fractional $R$-ideals.
In \cite{MarICM18}, the third author designs an algorithm to compute these isomorphism classes, which in fact form a commutative monoid under the operation induced by ideal multiplication, called the ideal class monoid of $R$.
Each ideal class, say with multiplicator ring $S$, is uniquely determined by a unique element of the class group $\Cl(S)$ of $S$ and by a unique local isomorphism class, also known as a weak equivalence class.
The computation of $\Cl(S)$ is efficiently reduced to the computation of $\Cl(\OO_K)$ which is classical, see \cite{klupau05}.
The third author has designed algorithms to compute the local isomorphism classes, starting in \cite{MarICM18}, with improvements in \cite{MarsegliaType22}, and finally culminating with the algorithm \texttt{ComputeW} contained in \cite{Mar23_Loc}, which is the current state-of-the-art.
The algorithm \texttt{ComputeW} plays an important role in this paper, as we will describe in more detail in the outline of the paper.

Algorithms to compute isomorphism classes of abelian varieties that do not have commutative endomorphism algebras, again for isogeny classes which are either ordinary or over a prime field, are also due to the third author and can be found in \cite{MarBassPow} and \cite{MarModules25}. 
Finally, there are algorithms to compute isomorphism classes of polarized abelian varieties in the ordinary case, see \cite{MarAbVar18,KNRR}, and over a prime field, see \cite{BKM}.

\subsubsection*{Outline of the paper}
In Section~\ref{sec-abvars} we describe the isomorphism classes of abelian varieties over finite fields in a fixed isogeny class in terms of isomorphism classes of Tate and Dieudonn\'e modules together with the class groups of the endomorphism ring they determine. 

From here on, we only consider abelian varieties with a commutative endomorphism ring.
The Tate modules of our abelian varieties can then be viewed as fractional $R_{\ell}$-ideals for rational primes $\ell \neq p$. 
In Section~\ref{sec:tate}, we show how these local objects can be described in terms of fractional ideals of some overorders of $R$, which are global objects, see Theorem~\ref{thm:W_R_X_ell}.

Indeed, we want to highlight that our algorithms are all designed to operate with global objects, that is, with $\Z$-lattices, as opposed to $\Z_\ell$-lattices, where $\ell$ is a rational prime.
This feature, which comes at the cost of some technical lemmas and heavier notation, allows us to obtain \emph{exact} algorithms, that is, algorithms which are not affected by the precision issues that arise from working with local lattices.
There are well-established algorithms to compute isomorphism classes of these types of global ideals, as the aforementioned~\cite[\texttt{ComputeW}]{Mar23_Loc}.

Inspired by Waterhouse \cite[Section 5]{Wat69}, who was working under the additional assumption that $R$ is the maximal order, we show in Subsections~\ref{ssec:setupXp} and~\ref{sec:Xpip} how to realize the Dieudonn\'e modules of our abelian varieties as fractional $W_R \{F,V\}$-ideals. 
In Subsection~\ref{sec:isoWFV} we then show a series of results about the structure of $W_R \{F,V\}$-ideals corresponding to the connected-\'etale sequence of Dieudonn\'e modules. In particular, there is an \'etale part $R^{\{0\}}_p$ (resp.~multiplicative part $R^{\{1\}}_p$) of $R$, and we show in Proposition~\ref{prop-ordinary-ideals} 
that the \'etale (resp.~multiplicative) part of $W_R \{F,V\}$-ideals can be described as fractional $R^{\{0\}}_p$-ideals (respectively $R^{\{1\}}_p$-ideals). 

The descriptions of Tate modules and Dieudonn\'e modules in Sections~\ref{sec-abvars}--\ref{sec:dieudonne} are put together in Section~\ref{sec:equiv} to give a categorical equivalence between abelian varieties in our fixed isogeny class and pairs of a fractional $R$-ideal and a fractional $W_R \{F,V\}$-ideal, see Theorem~\ref{thm-equivalence} (which was also stated above). In Proposition~\ref{prop:extension_are_end} and Corollary~\ref{cor:min_end_ell_01}, we show some consequences of this equivalence for the possible endomorphism rings of our abelian varieties. 

In Section~\ref{sec:aftergreenredblue} we show how to create algorithms to compute isomorphism classes of $W_R \{F,V\}$-ideals.
Since Subsection~\ref{seq:etalemult} deals with the {\'e}tale and multiplicative parts of the Dieudonn{\'e} modules, we focus our attention on the local-local part.
Our approach can be summarized in four steps, of which especially the third and fourth steps contain entirely new ideas which, to our knowledge, have never been implemented before.
To avoid precision issues, in the first step we globalize our objects, considering fractional ideals in \'etale $\Q$-algebras rather than in \'etale $\Q_p$-algebras. 
In the second step, we again use the  
procedure described in Section~\ref{sec:tate}, to compute global representatives of (the local-local parts of) of fractional $W_R$-ideals modulo $W_R$-linear isomorphisms, essentially relying on \cite[\texttt{ComputeW}]{Mar23_Loc}.
The starting point for the third and fourth steps is the combinatorial classification of the isomorphism classes of $W_R \{F,V\}$-ideals with maximal endomorphism ring $\OO_A$ given by Waterhouse in \cite[Theorem~5.1]{Wat69}, which can be readily adapted to compute only the local-local parts.
The third step then consists in studying the fibers of the extension map from (the local-local part of) $W_R$-ideals to ideals of the maximal order, to understand how each $W_R$-linear isomorphism class that we computed in the second step is partitioned into classes under the action of (the local-local parts of) $\Delta(E_p^\times)$, giving us what we call the $\Delta'$-isomorphism classes.
This is achieved by computing a global representative of certain finite quotients of the unit group of (the local-local part of) $\OO_A$ together with the induced action of an approximation of the (local) Frobenius $\sigma\colon L \to L$.
Since these quotients are finite, after a careful analysis to determine the required precision in the approximation, we are guaranteed to obtain the correct output; see Algorithm~\ref{alg:Deltaclasses_OO}.
Finally, in the fourth step we determine which of the $\Delta'$-isomorphism classes that are stable under $F$ and $V$. To do this we need to compute suitable approximations of $F$ and $V$, that is, representatives on finite quotients, since they do not admit global representatives.
The fourth step results in Algorithm~\ref{alg:FVclasses}, which is our most important algorithmic contribution.

To upshot of this four-step approach is that the representative we compute for each isomorphism class contains complete information about the Tate modules of the corresponding abelian variety for each $\ell \neq p$. 
Moreover, we get an approximation of $F$ and $V$ acting on the Dieudonn\'e module, modulo some power of $p$, which is sufficiently large to detect stability. This approximation can in principle be made arbitrarily accurate.

Finally, in Section~\ref{sec:examples}, we give examples of computations of isomorphism classes of abelian varieties using our algorithms. 
In particular, we study the behavior of the endomorphism rings that appear for $g$-dimensional abelian varieties of $p$-rank $<g-1$, which shows patterns very different from abelian varieties with higher $p$-rank (that is, ordinary and almost-ordinary abelian varieties, as referred to above), see Examples~\ref{ex:simple1}, \ref{ex:nonsimple1}, \ref{ex:simple2}, \ref{ex:simple3}, \ref{ex:nonsimple2} and \ref{ex:end_max_p}. 

\tableofcontents

\section{Tate and Dieudonn\'e modules}\label{sec-abvars}
Let $X$ be an abelian variety over a finite field $\mathbb{F}_q$ of characteristic~$p$. 
In this section we discuss a general description of the isogeny class of~$X$ over~$\mathbb{F}_q$, as also contained in~\cite{xue2017counting} (cf.~\cite{chiafu2010}). 
Since we will focus on abelian varieties with commutative endomorphism algebra in this article, we will restrict to this case at the end of this section.

By Poincar{\'e} reducibility, the abelian variety $X$ admits an isogeny decomposition $X \sim X_1^{k_1} \times \ldots \times X_r^{k_r}$ over $\F_q$ into $\F_q$-simple abelian varieties $X_i$ with respective multiplicities $k_i$. 
By Honda-Tate theory, the isogeny class of $X$ is determined by the characteristic polynomial~$h(x)$ of its Frobenius endomorphism. The isogeny decomposition of $X$ implies that $h(x) = h_1(x)^{k_1} \cdot \ldots \cdot h_r(x)^{k_r} \in \mathbb{Z}[x]$, where each $h_i(x)$ is the characteristic polynomial of the Frobenius endomorphism $\pi_i$ of $X_i$. 
Write $m(x) = m_1(x) \cdot \ldots \cdot m_r(x)$ for the minimal polynomial of $\pi$ and let $R = \mathbb{Z}[\pi,q/\pi]$ denote the $\mathbb{Z}$-order in the {\'e}tale algebra $E=\mathbb{Q}[x]/(m(x))$ with $\pi= x \bmod m(x)$. 

Let $\underline{k}=(k_1,\ldots,k_r)$.
We denote by $\mathcal{A}_{\pi,\underline{k}}$ the category whose objects are the abelian varieties $\F_q$-isogenous to $X$ together with $\F_q$-homomorphisms.
We write $\mathcal A_{\pi,\underline{k}}^{\mathrm{isom}}$ for the set of $\mathbb{F}_q$-isomorphism classes in $\mathcal{A}_{\pi,\underline{k}}$.
If $\underline{k}=(1,\ldots,1)$, then the subscript $\underline{k}$ will be omitted from the notation.

All homomorphisms will be assumed to be defined over $\mathbb{F}_q$.
Recall that the endomorphism ring $\End(X)$ is an order in the endomorphism algebra
\begin{equation}\label{eq:End0Asplit}
\mathrm{End}^0(X) = \mathrm{End}(X) \otimes_{\mathbb{Z}} \mathbb{Q} \simeq \mathrm{Mat}_{k_1}(\mathrm{End}^0(X_1)) \times \ldots \times \mathrm{Mat}_{k_r}(\mathrm{End}^0(X_r)).
\end{equation}
Consider the group scheme $G_X$ over $\mathbb{Z}$ such that 
\begin{equation}\label{eq:GA}
G_X(R_0) = (\mathrm{End}(X) \otimes_{\mathbb{Z}} R_0)^{\times}
\end{equation}
for any commutative ring $R_0$ with unit.

For a rational prime $\ell\neq p$, let $T_{\ell}(X)$ be the $\ell$-adic Tate module of $X$ and let $V_\ell(X) = T_\ell(X)\otimes_{\Z_\ell} \Q_\ell$.
Let $M(X)$ be the covariant Dieudonn\'e module of $X$ (which is categorically equivalent to its $p$-divisible group). 
Put $R_{\ell} = R \otimes_{\mathbb{Z}} \mathbb{Z}_{\ell}$ for any prime $\ell$, including $\ell = p$.
\begin{df}\label{def:Xpitilde}
 Out of the Tate and Dieudonn{\'e} modules of $X$, we build 
\begin{align*}
    \widetilde{\mathfrak{X}}_{\pi,\underline{k},\ell} & = \{ \text{left $R_\ell$-lattices in }V_\ell(X) \text{ of full rank} \} \text{ for all $\ell \neq p$};\\
    \sideset{}{'}\prod_{\ell \neq p} \widetilde{\mathfrak{X}}_{\pi,\underline{k},\ell} &= \{ (T_{\ell})_{\ell \neq p} \in \prod_{\ell \neq p} \widetilde{\mathfrak{X}}_{\pi,\underline{k},\ell} \text{ : } T_{\ell} = T_{\ell}(X) \text{ for all but finitely many } \ell \neq p \};\\
    \widetilde{\mathfrak{X}}_{\pi,\underline{k},p} & = \{ \text{Dieudonn\'e submodules in $M(X)\otimes_{\Z_p} \Q_p$ of full rank} \};\\
    \widetilde{\mathfrak{X}}_{\pi,\underline{k}} & = \widetilde{\mathfrak{X}}_{\pi,\underline{k},p} \times \sideset{}{'}\prod_{\ell\neq p}\widetilde{\mathfrak{X}}_{\pi,\underline{k},\ell}.
\end{align*}
As before, if $\underline{k}=(1,\ldots,1)$ then it will be omitted from the notation.
\end{df}

By Tate's theorem, $\widetilde{\mathfrak{X}}_{\pi,\underline{k}}$ is in bijection with 
\begin{equation}\label{eq:Atilde}
\widetilde{\mathcal{A}}_{\pi,\underline{k}} = \{ \text{ quasi-isogenies } \varphi\colon X' \to X, \text{ up to isomorphism } \}
\end{equation}
by mapping $\left( \varphi\colon X' \to X \right) \mapsto \left(\varphi_*(M(X')), (\varphi_*(T_{\ell}(X')))_{\ell \neq p}\right)$; a quasi-isogeny $\varphi\colon X' \to X$ is an element of $\mathrm{Hom}(X',X) \otimes_{\mathbb{Z}} \mathbb{Q}$ such that there exists an integer $N$ for which $N\varphi$ is an isogeny. 
This map $\widetilde{\mathcal{A}}_{\pi,\underline{k}} \to \widetilde{\mathfrak{X}}_{\pi,\underline{k}}$ is $G_X(\mathbb{A}_f)$-equivariant by construction, where $\mathbb{A}_f$ denotes the finite adele ring of~$\mathbb{Q}$.

Further, we may project $\widetilde{\mathcal{A}}_{\pi,\underline{k}} \twoheadrightarrow \mathcal{A}_{\pi,\underline{k}}^{\mathrm{isom}}$ by mapping $\left( \varphi\colon X' \to X \right)$ to the isomorphism class of $X'$.
Since two quasi-isogenies $\varphi_1\colon X' \to X$ and $\varphi_2\colon X'' \to X$ are considered isomorphic if there is an isomorphism $\alpha\colon X' \to X''$ such that $\varphi_2 \circ \alpha = \varphi_1$, this surjection can be viewed as taking the $G_X(\mathbb{Q})$-orbits of $\widetilde{\mathcal{A}}_{\pi,\underline{k}}$. 

\begin{df} \label{def:Xpi} Starting with the objects in~Definition~\ref{def:Xpitilde}, we let 
\begin{align*}
    \mathfrak{X}_{\pi,\underline{k},\ell} & = \widetilde{\mathfrak{X}}_{\pi,\underline{k}, \ell} \text{ up to isomorphism of $R_{\ell}$-modules};\\
    \mathfrak{X}_{\pi,\underline{k},p} & = \widetilde{\mathfrak{X}}_{\pi,\underline{k},p} \text{ up to isomorphism of Dieudonn{\'e} modules};\\
    \mathfrak{X}_{\pi,\underline{k}} & = \mathfrak{X}_{\pi,\underline{k},p} \times \prod_{\ell\neq p}\mathfrak{X}_{\pi,\underline{k},\ell}.
\end{align*}
\end{df}
Thus we obtain a projection map $\widetilde{\mathfrak{X}}_{\pi,\underline{k}} \to \mathfrak{X}_{\pi,\underline{k}}$.
Consider the fiber $\widetilde{\mathcal{A}}_{\pi,\underline{k},Y}$ in $\widetilde{\mathcal{A}}_{\pi,\underline{k}} \simeq \widetilde{\mathfrak{X}}_{\pi,\underline{k}}$ above an element $Y = (M/\simeq, (T_{\ell})_{\ell \neq p}/\simeq)$ in $\mathfrak{X}_{\pi,\underline{k}}$. 
By definition, it consists of all isomorphism classes of quasi-isogenies $\varphi\colon X' \to X$ such that the Dieudonn{\'e} and Tate modules of the source abelian varieties $X'$ satisfy $M(X')\simeq M$ and $T_\ell(X')\simeq T_\ell$ for each $\ell\neq p$.

The group $G_X(\mathbb{A}_f)$ acts transitively on the set of quasi-isogenies on each fiber $\widetilde{\mathcal{A}}_{\pi,\underline{k},Y}$.
The stabilizer of any given $\varphi\colon X' \to X$ in $\widetilde{\mathcal{A}}_{\pi,\underline{k},Y}$
under this action is the open compact subgroup
\begin{equation}\label{eq:UX}
U_Y = \mathrm{Aut}(M) \times \prod_{\ell \neq p} \mathrm{Aut}_{R_\ell}(T_{\ell}) 
\end{equation}
of $G_X(\mathbb{A}_f)$. Hence, the fiber can be described as
\begin{equation}\label{eq:AAXtilde}
\widetilde{\mathcal{A}}_{\pi,\underline{k},Y} \simeq G_X(\mathbb{A}_f)/U_Y.
\end{equation}

For all $X' \in \mathcal{A}_{\pi,\underline{k}}$, we have that $\End(X)$ and $\End(X')$ are locally equal at all but finitely many rational primes.
It follows that $G_{X'}(\Q) \simeq G_X(\Q)$ and $G_{X'}(\mathbb{A}_f) \simeq G_X(\mathbb{A}_f)$ for all $X'  \in \mathcal{A}_{\pi,\underline{k}}$, where $G_{X'}(R_0) = (\mathrm{End}(X') \otimes_{\Z} R_0)^{\times}$, cf.~\eqref{eq:GA}.
In particular, we have
\[ G_X(\Q) \simeq G_{X_1^{k_1}}(\Q) \times \ldots \times G_{X_r^{k_r}}(\Q) \]
and
\[ G_X(\mathbb{A}_f) \simeq G_{X_1^{k_1}}(\mathbb{A}_f) \times \ldots \times G_{X_r^{k_r}}(\mathbb{A}_f). \]
By contrast, in general, the same splitting does not apply to $U_Y$, since it is an integral object.
 
Finally, and similarly to $\widetilde{\mathcal{A}}_{\pi,\underline{k}} \to \widetilde{\mathfrak{X}}_{\pi,\underline{k}}$, we obtain a surjective map $\Phi: \mathcal{A}_{\pi,\underline{k}}^{\mathrm{isom}} \twoheadrightarrow \mathfrak{X}_{\pi,\underline{k}}$ by associating to an abelian variety the isomorphism classes of its Dieudonn{\'e} and Tate modules. 
Recall that $\widetilde{\mathcal{A}}_{\pi,\underline{k}} \twoheadrightarrow \mathcal{A}_{\pi,\underline{k}}^{\mathrm{isom}}$ arises from taking $G_X(\mathbb{Q})$-orbits. It then follows from~\eqref{eq:AAXtilde} and the commutative diagram
\[
\begin{tikzcd}
\widetilde{\mathcal{A}}_{\pi,\underline{k}} \arrow[r, "\sim"] \arrow[d]
& \widetilde{\mathfrak{X}}_{\pi,\underline{k}} \arrow[d] \\
\mathcal{A}_{\pi,\underline{k}}^{\mathrm{isom}} \arrow[r, "\Phi"]
& \mathfrak{X}_{\pi,\underline{k}}
\end{tikzcd}
\]
that the fiber $\mathcal{A}_{\pi,\underline{k},Y}^{\mathrm{isom}}$ above $Y \in \mathfrak{X}_{\pi}$ in $\mathcal{A}_{\pi,\underline{k}}^{\mathrm{isom}}$ satisfies
\begin{equation}\label{eq:AAX}
\mathcal{A}_{\pi,\underline{k},Y}^{\mathrm{isom}} \simeq G_X(\mathbb{Q})\backslash G_X(\mathbb{A}_f) / U_Y.
\end{equation}
By construction, this fiber consists exactly of all isomorphism classes of abelian varieties in $\mathcal{A}_{\pi,\underline{k}}^{\mathrm{isom}}$ that are everywhere locally isomorphic, with Dieudonn{\'e} and Tate modules given by the vector~$Y$. 
Ranging over all local isomorphism types $Y$, we obtain the following result. 

\begin{prop}\label{prop:Api}\textup{(cf.~\cite[Theorem~2.1]{xue2017counting})}
The set of isomorphism classes 
$\mathcal{A}_{\pi,\underline{k}}^{\mathrm{isom}}$ of $X$ can be described algebraically as
\begin{equation}\label{eq:Acalpi}
\mathcal{A}_{\pi,\underline{k}}^{\mathrm{isom}} \simeq \bigsqcup_{Y \in \mathfrak{X}_{\pi,\underline{k}}} G_X(\mathbb{Q}) \backslash G_X(\mathbb{A}_f) / U_Y;
\end{equation}
here, the isomorphism class of $X$ is sent to the neutral element.
\end{prop}

\begin{cor}
The number of isomorphism classes of abelian varieties in $\mathcal{A}_{\pi,\underline{k}}$ equals
\[
\vert \mathcal{A}_{\pi,\underline{k}}^{\mathrm{isom}} \vert =  \sum_{M \in \mathfrak{X}_{\pi,\underline{k},p}} \, \sum_{N \in \prod_{\ell\neq p}\mathfrak{X}_{\pi,\underline{k},\ell}} 
\vert G_X(\mathbb{Q}) \backslash G_X(\mathbb{A}_f) /U_{M \times N} \vert , 
\]    
where the sums run over the local isomorphism types $Y = M \times N \in \mathfrak{X}_{\pi,\underline{k},p} \times \prod_{\ell\neq p}\mathfrak{X}_{\pi,\underline{k},\ell} = \mathfrak{X}_{\pi,\underline{k}}$.
\end{cor}

As mentioned above, in this article we will restrict our attention to abelian varieties with commutative endomorphism algebra. 
By \cite[Theorem 2.(c)]{Tate66}, this is equivalent to requiring~$h$ to be \emph{square-free}, i.e. $k_i~=~1$ for all $i$.
From now on, we will therefore omit the multiplicities from the notation and simply write $\mathcal{A}_\pi$, $\mathcal{A}^{\mathrm{isom}}_\pi$, $\mathfrak{X}_\pi$, $\mathfrak{X}_{\pi,p}$ and $\mathfrak{X}_{\pi,\ell}$.
In particular, Equation~\eqref{eq:End0Asplit} then reads $\mathrm{End}^0(X) \simeq \mathrm{End}^0(X_1) \times \ldots \times \mathrm{End}^0(X_r)$. 
Furthermore, the fiber above 
$Y = M \times N$,
given as in~\eqref{eq:AAX}, then equals the (usual) class group $\mathrm{Cl}(\mathcal{O}_Y)$, where 
\begin{equation} \label{eq-OX}
\mathcal{O}_Y = \mathrm{End}^0(X) \cap \left( \mathrm{End}(M) \times \prod_{\ell \neq p} \mathrm{End}_{R_\ell}(T_{\ell}) \right);
\end{equation}
this intersection takes place in $\mathrm{End}^0(X) \otimes_{\mathbb{Q}} \mathbb{A}_f$.
This also implies that its cardinality is the (usual) class number $h(\mathcal{O}_Y)$. We summarize this discussion in the following corollary.

\begin{cor}\label{cor:Api} 
The number of isomorphism classes of abelian varieties with commutative endomorphism algebra in $\mathcal{A}_\pi$ equals 
\[
\vert \mathcal{A}_{\pi}^{\mathrm{isom}} \vert =  \sum_{M \in \mathfrak{X}_{\pi,p}} \, \sum_{N \in \prod_{\ell\neq p}\mathfrak{X}_{\pi,\ell}} h(\mathcal{O}_{M \times N}).
\]
\end{cor}

\section{Global representatives of Tate modules} \label{sec:tate}
In Section~\ref{sec-abvars} we realized the $\ell$-adic Tate modules of abelian varieties in $\mathcal{A}_\pi$ as $R_{\ell}$-lattices, for every rational prime $\ell\neq p$.
In this section, we provide a series of reduction steps that allow us to compute the isomorphism classes of these $R_{\ell}$-lattices using existing algorithms for local isomorphism classes -- also known as weak equivalence classes -- of fractional ideals, see Theorem~\ref{thm:W_R_X_ell}.
The current state-of-the-art of these algorithms is \cite[\texttt{ComputeW}]{Mar23_Loc}.
This algorithm has the desirable feature of being exact, that is, not affected by the precision issues that come from working with local objects.
The content of this section will be used also later in the text in the computation of the Dieudonn\'e modules, that is, the part at $p$.
More precisely, each Dieudonn\'e module is the direct sum of an \'etale, a multiplicative and local-local part, see Subsection~\ref{sec:connetale}. 
The computation of the first two parts can be completely reduced to a finite number of calls of \cite[\texttt{ComputeW}]{Mar23_Loc}, see Proposition~\ref{prop-01}.
The first two steps out of four of computation of the local-local part also rely on the content of this section, see Subsection~\ref{sec:Step1}.

\subsection{Isomorphism classes of fractional ideals} \label{sec:isofrac}
Let $Z$ be either $\Z$ or $\Z_p$, and let $Q$ be the field of fractions of~$Z$.
Let $S$ be a $Z$-order in an \'etale $Q$-algebra~$E$ with maximal order $\OO$.
For a fractional $S$-ideal $I$, the multiplicator ring $(I:I)$ is an order.
For a maximal ideal $\ell$ in $Z$, let $E_\ell=E\otimes_Q Q_\ell$, $S_\ell =S \otimes_Z Z_\ell$, and $I_\ell = I \otimes_Z Z_\ell$.
Let $\frl$ be a maximal ideal of $S$.
Denote by $S_\frl$ the completion of $S$ at $\frl$.
Let $I_\frl=I\otimes_S S_\frl$.

\begin{df}
\label{def:Ws}
Let $\mathcal{S}$ be a (possibly infinite) set of maximal ideals $\ell$ of $Z$,
and $\mathcal{T}$ be a (possibly infinite) set of maximal ideals $\mathfrak{l}$ of $S$.
We define the following objects:
\[
\begin{split}
    \cW_\ell(S) & =\frac{\{ \text{fractional $S$-ideals} \}}{ \{ I_\ell \simeq J_\ell \text{ as $S_\ell$-modules} \}}, \\
    \cW_\frl(S) &=\frac{\{ \text{fractional $S$-ideals} \}}{ \{ I_\frl \simeq J_\frl \text{ as $S_\frl$-modules} \}}, \\
    \cW(S) &=\frac{\{ \text{fractional $S$-ideals} \}}{ \{ I_\ell \simeq J_\ell \text{ as $S_\ell$-modules, for every $\ell$} \}}, \\
    \cW_{\mathcal{S}}(S) & =\frac{\{ \text{fractional $S$-ideals} \}}{ \{ I_\ell \simeq J_\ell \text{ as $S_\ell$-modules, for every $\ell \in \mathcal{S}$} \}}, \\
    \cW_{\mathcal{T}}(S) & =\frac{\{ \text{fractional $S$-ideals} \}}{ \{ I_\frl \simeq J_\frl \text{ as $S_\frl$-modules, for every $\frl \in \mathcal{T}$} \}}.
    \end{split}
\]
We denote the class of a fractional $S$-ideal $I$ in $\cW_\ell(S)$ (resp.~$\cW_\frl(S)$, $\cW(S)$, $\cW_{\mathcal{S}}(S)$,  $\cW_{\mathcal{T}}(S)$)
by $[I]_\ell$ (resp.~$[I]_{\frl}$, $[I]$, $[I]_\mathcal{S}$, $[I]_\mathcal{T}$).
\end{df}

This subsection will be devoted to studying the relations between the objects defined in Definition~\ref{def:Ws}, which are local in nature, and to giving a concrete description of them by means of global representatives.

\begin{remark}
    \label{rmk:W_R_semilocal}
    If $Z=\Z_p$ then $\cW(S)$ coincides with the set of $S$-linear isomorphism classes of fractional $S$-ideals.
\end{remark}

\begin{remark}
    \label{rmk:monoid}
    Ideal multiplication endows each set defined in Definition \ref{def:Ws} with a commutative monoid structure, whose unit is given by the class of the order $S$. This structure will not be used in this article.
    The class $[I]$ of $I$ in $\cW(S)$ is often called the genus of $I$. 
    There is a vast literature studying genera of fractional ideals and more generally of finitely generated modules. 
    See for example \cite{Roggenkamp70I}, \cite{Roggenkamp70II}, \cite{Guralnick84}, \cite{Guralnick87} and \cite{Reiner03}.
    Moreover, two fractional $S$-ideals $I$ and $J$ are in the same genus if and only if they are weakly equivalent, that is, their localizations at every maximal ideal of $S$ are isomorphic, see \cite[Section~5]{LevyWiegand85}.
    The definition of weak equivalence was originally given in \cite{dadetz62}.
    Results to compute and classify weak equivalence classes are given in \cite{MarICM18}, \cite{MarsegliaType22} and \cite{Mar23_Loc}.
\end{remark}

\begin{lem}\label{lem:ell_vs_frl}
    Let $I$ and $J$ be two fractional $S$-ideals and fix a maximal ideal $\ell$ of $Z$. Then the following statements are equivalent:
    \begin{enumerate}[(i)]
        \item \label{lem:ell_vs_frl:ell} $[I]_\ell=[J]_\ell$.
        \item \label{lem:ell_vs_frl:frl} $I_\frl \simeq J_\frl$ as $S_\frl$-modules for every maximal ideal $\frl$ of $S$ above $\ell$.
    \end{enumerate}
\end{lem}
\begin{proof}
    Denote by $\frl_1,\ldots, \frl_n$ the maximal ideals of $S$ above $\ell$.
    Since $S_\ell$ is a complete semilocal ring, we have a canonical isomorphism
    \[ S_\ell \simeq  S_{\frl_1} \times \ldots \times  S_{\frl_n}. \]
    Tensoring with $I$ and $J$, we obtain the equivalence of \ref{lem:ell_vs_frl:ell} and \ref{lem:ell_vs_frl:frl}.
\end{proof}

\begin{lem}\label{rmk:WRlocalize}
    Let $\frl$ be a maximal ideal of $S$ above the maximal ideal $\ell$ of $Z$.
    \begin{enumerate}[(i)]
        \item \label{rmk:WRlocalize:1}
        For every fractional $S_\ell$-ideal $I$ there exists a fractional $S$-ideal $\widetilde I$ such that $\widetilde I \otimes_Z Z_\ell = I$. The class $[\widetilde I]_\ell$ in $\cW_\ell(S)$ is uniquely determined by the class $[I]_\ell$ in  $\cW(S_\ell)$.
        Hence, we have a canonical bijection $\cW_\ell(S) \longleftrightarrow \cW(S_\ell)$.
        \item \label{rmk:WRlocalize:2}
        Similarly, we have a canonical bijection $\cW_\frl(S) \longleftrightarrow \cW(S_\frl)$.
    \end{enumerate}
\end{lem}
\begin{proof}
    Identify $E$ with its image in $E_\ell$.
    Then $\widetilde I = I \cap E$ satisfies $\widetilde I \otimes_Z Z_\ell = I$.
    Morever, we see that if $J$ is a second fractional $S_\ell$-ideal, then $[I] = [J]$ if and only if $[\widetilde I]_\ell = [\widetilde J]_\ell$.
    This completes the proof of \ref{rmk:WRlocalize:1}.
    For \ref{rmk:WRlocalize:2}, assume that $\frl=\frl_1,\ldots,\frl_n$ are the maximal ideals of $S$ above $\ell$ and that we are given a fractional $S_\frl$-ideal $I$.
    Consider the fractional $S_\ell$-ideal $I'=I \times S_{\frl_2} \times \ldots S_{\frl_n}$ and the fractional $S$-ideal~$\widetilde I$ such that $\widetilde I \otimes_Z Z_\ell = I'$.
    We get that $(\widetilde I)_\frl = I$ and that the class of $\widetilde I$ in $\cW_\frl(S)$ is uniquely determined by the class of $I$ in $\cW(S_\frl)$.
\end{proof}

Given a set $\mathcal{S}$ as above, denote by $\mathcal{S}_0$ the finite subset of $\mathcal{S}$ consisting of maximal ideals of $Z$ dividing the index $[\OO:S]$.
Define $\cW_{\mathcal{S}_0}(S)$ and $[I]_{\mathcal{S}_0}$ analogously to $\cW_{\mathcal{S}}(S)$ and $[I]_{\mathcal{S}}$.
Similarly, given a set $\mathcal{T}$ as above, denote by $\mathcal{T}_0$ the finite subset of $\mathcal{T}$ consisting of maximal ideals of $S$ containing the conductor $\frf=(S:\OO)$ of $S$.
Define $\cW_{\mathcal{T}_0}(S)$ and $[I]_{\mathcal{T}_0}$ analogously.

\begin{prop}\label{prop:S_to_S0}
    Consider the natural surjections 
    \[ \cW(S) \overset{i_1}{\longrightarrow} \cW_{\mathcal{S}}(S) \overset{i_2}{\longrightarrow}  \cW_{\mathcal{S}_0}(S) \]
    and
    \[ \cW(S) \overset{j_1}{\longrightarrow} \cW_{\mathcal{T}}(S) \overset{j_2}{\longrightarrow}  \cW_{\mathcal{T}_0}(S). \]
    Then $i_2$ and $j_2$ are bijections.
    Moreover, $i_1$ is a bijection if and only if $\mathcal{S}$ contains all maximal ideals above the index $[\OO:S]$, and 
    $j_1$ is a bijection if and only if $\mathcal{T}$ contains all maximal ideals above the conductor $\frf$ of $S$.
\end{prop}
\begin{proof}
    A maximal ideal $\ell$ does not divide the index $[\OO:S]$ if and only if $S_\ell = \OO_\ell$.
    For such a prime $\ell$, every fractional $S$-ideal $I$ satisfies $I_\ell \simeq \OO_\ell$, since $\OO_\ell$ is a principal ideal ring.
    This immediately implies the statements about $i_1$ and $i_2$.
    The statements about $j_1$ and $j_2$ follow analogously from the observation that a maximal ideal $\frl$ of $S$ does not divide the conductor $\frf$ of~$S$ if and only if $S_\frl = \OO_\frl$.
\end{proof}

\begin{lem}\label{lem:glue_local_parts}
    Let $\mathfrak{l}_1,\ldots,\mathfrak{l}_n$ be maximal ideals of $S$.
    Consider a vector of fractional $S$-ideals $(I_1,\ldots,I_n)$ such that each~$I_i$ is contained in~$S$.
    For each~$i$, let~$k_i$ be a nonnegative integer such that~$\mathfrak{l}_i^{k_i}S_{\mathfrak{l}_i} \subseteq (I_i)_{\mathfrak{l}_i}$.
    Set
    \[ J = \sum_{i=1}^n\left( (I_i+\mathfrak{l}_i^{k_i})\prod_{j\neq i} \mathfrak{l}_j^{k_j}  \right). \]
    Then~$J_{\mathfrak{l}_i}=(I_i)_{\mathfrak{l}_i}$ for each~$i$, and~$J_\mathfrak{l}=S_\mathfrak{l}$ for every other maximal ideal~$\mathfrak{l}$ of $S$.
\end{lem}
\begin{proof}
    See the proof of \cite[Theorem~4.4]{Mar23_Loc}.
\end{proof}

\begin{prop}\label{prop:S0_to_prod_ell}
    The natural maps
    \[ \cW_{\mathcal{S}_0}(S) \longrightarrow  \prod_{\ell\in \mathcal{S}_0}\cW_\ell(S) \]
    and
    \[ \cW_{\mathcal{T}_0}(S) \longrightarrow  \prod_{\frl\in \mathcal{T}_0}\cW_\frl(S), \]
    are bijections.
\end{prop}
\begin{proof}
    The maps are injective by construction.
    Surjectivity follows from Lemma~\ref{lem:glue_local_parts}.
\end{proof}

\begin{prop}\label{prop:ell_to_prod_frl}
    Let $\ell$ be a maximal ideal of $Z$ and let $\frl_1,\ldots,\frl_n$ be the maximal ideals of~$S$ above $\ell$ and containing the conductor $\frf=(S:\OO)$. 
    We have a natural bijection
    \[ \varphi: \cW_\ell(S) \longrightarrow \prod_{i=1}^n \cW_{\frl_i}(S). \]
\end{prop}
\begin{proof}
    As pointed out above, if $\frl$ is a maximal ideal of $S$ which does not divide the conductor, then $S_\frl = \OO_\frl$, which is a principal ideal ring.
    Hence, for every fractional $S$-ideal $I$, we have $I_\frl \simeq \OO_\frl$, or equivalently, $\cW_\frl(S)$ is trivial.
    Therefore $\varphi$ is injective by Lemma \ref{lem:ell_vs_frl}.
    Surjectivity follows from Lemma~\ref{lem:glue_local_parts}.
\end{proof}

The next proposition is \cite[Proposition~4.3]{Mar23_Loc}.
\begin{prop}\label{prop:global_to_frl}
    Let $\frl$ be a maximal ideal of $S$. 
    Let $k$ be a nonnegative integer such that $(\frl^k\OO)_\frl \subseteq S_\frl$.
    Then the natural map
    \[ \cW(S+\frl^k\OO) \longrightarrow \cW_\frl(S) \]
    is a bijection.
\end{prop}

\begin{remark}\label{rmk:push_primes_O_in_R}
    In the statement of Proposition~\ref{prop:global_to_frl}, 
    we can take $k = \val_\ell( [\OO:S] )$ to achieve $(\frl^{k}\OO)_{\frl} \subseteq S_{\frl}$, where $\val_\ell$
    denotes the $\ell$-adic valuation. 
    Note also that if $(\frl^{k}\OO)_{\frl} \subseteq S_{\frl}$ and $k'>k$ then $S+\frl^k\OO = S+\frl^{k'}\OO$.
\end{remark}

\subsection{Determining \texorpdfstring{$\prod_{\ell\neq p}\mathfrak{X}_{\pi,\ell}$}{Prl} } 
\label{sec:deterl}
As in~Section~\ref{sec-abvars}, we consider an isogeny class $\mathcal{A}_{\pi}$ 
of abelian varieties over~$\mathbb{F}_q$, where $q=p^a$ for some prime~$p$, with commutative endomorphism algebra $E = \mathbb{Q}[\pi]$. We now use the results of the previous subsection to give a concrete description of the set $\prod_{\ell\neq p}\mathfrak{X}_{\pi,\ell}$. 

\begin{thm}\label{thm:W_R_X_ell}
    Let $R=\Z[\pi,q/\pi]$.  
    \begin{enumerate}[(i)] 
        \item \label{thm:W_R_X_ell:one_ell} For each $\ell \neq p$, localization at~$\ell$ induces a bijection
        \[ \cW_\ell(R)\longrightarrow \mathfrak{X}_{\pi,\ell}.\]
        \item \label{thm:W_R_X_ell:all_ell} Denote by $\frl_1,\ldots,\frl_n$ the maximal ideals of $R$ which divide the conductor $(R:\OO)$ and do not contain~$p$.
        For each $i$, let $k_i$ be a nonnegative integer such that $(\frl_i^{k_i}\OO)_{\frl_i} \subseteq R_{\frl_i}$. 
        Then we have a bijection
        \[ \prod_{i=1}^n \cW(R+\frl_i^{k_i}\OO) \longrightarrow \prod_{\ell \neq p} \mathfrak{X}_{\pi,\ell}. \]
    \end{enumerate}    
\end{thm}
\begin{proof}
    Fix a prime $\ell\neq p$ and an abelian variety $X \in \mathcal{A}_{\pi}$.
    After identifying $V_\ell(X) = T_{\ell}(X) \otimes_{\mathbb{Z}_{\ell}} \mathbb{Q}_{\ell}$ with $E_\ell = \mathbb{Q}[\pi]_{\ell}$ we see that 
    \[ \mathfrak{X}_{\pi,\ell}=
    \{\text{fractional $R_\ell$-ideals in $E_\ell$ up to isomorphism as $R_{\ell}$-modules} \}.
    \]
    Hence, localization at~$\ell$ induces a natural injective map
    \[ \cW_\ell(R)\longrightarrow \mathfrak{X}_{\pi,\ell}. \]
    Surjectivity follows from the fact that every fractional $R_\ell$-ideal~$J$ is of the form $J=I\otimes_R R_\ell$ for some fractional $R$-ideal $I$, cf.~Lemma~\ref{rmk:WRlocalize}.(i). This completes the proof of \ref{thm:W_R_X_ell:one_ell}.

    Combining Propositions \ref{prop:S_to_S0} (applied with $\mathcal{S}$ consisting of all rational primes $\ell\neq p$), \ref{prop:ell_to_prod_frl}, and \ref{prop:global_to_frl}, we see that there is a  natural bijection from $\prod_{i=1}^n \cW(R+\frl_i^{k_i}\OO)$ to $\prod_{\ell \neq p} \cW_\ell(R)$.
    By applying \ref{thm:W_R_X_ell:one_ell} for each $\ell\neq p$, we obtain \ref{thm:W_R_X_ell:all_ell}.
\end{proof}

\section{Dieudonn\'e modules and where to find them} \label{sec:dieudonne}
In Subsections~\ref{ssec:setupXp} and \ref{sec:Xpip} we will find representatives of Dieudonn\'e modules as certain fractional $W_R$-ideals in $A=L \otimes_{\Q_p} E_p $, for $W_R = W \otimes_{\Z_p} R_p$ where $W$ is the maximal order in an unramified extension $L$ of~$\Q_p$ of degree $a$.
These fractional ideals should be stable under operators $F\colon A \to A$ and $V=pF^{-1}$, where~$F$ has the Frobenius property, see Definition~\ref{def:Frobprop}; they are called $W_R\{F,V\}$-ideals. 
Maps between $W_R\{F,V\}$-ideals will be given by elements from $E_p$ via the diagonal embedding $\Delta\colon E_p \to A$, see Lemma~\ref{lem-homs}. 
Theorem~\ref{thm:X_pi_p_WR_F_V} gives a bijection between isomorphism classes of Dieudonn\'e modules and $\Delta$-isomorphism classes of $W_R\{F,V\}$-ideals.

In Subsection~\ref{sec:connetale}, we study an analogue for $W_R\{F,V\}$-ideals of the connected-\'etale sequence, which gives a decomposition of any $W_R\{F,V\}$-ideal into an \'etale part, a multiplicative part and a local-local part. 
The \'etale and multiplicative parts are studied in Subsection~\ref{seq:etalemult}. 
A classification due to Waterhouse of all possible local-local parts with maximal endomorphism ring is given in Subsection~\ref{sec:locallocal}.

Later, in Section~\ref{sec:aftergreenredblue} we will create an effective algorithm to find global representatives of $\Delta$-isomorphism classes of $W_R\{F,V\}$-ideals.

\subsection{The Frobenius property}\label{ssec:setupXp}
Firstly, we introduce some notation, closely following that of Waterhouse~\cite[Section 5]{Wat69}. 
Let $q=p^a$ again be a power of a prime~$p$.
We consider an isogeny class $\mathcal{A}_\pi$ of abelian varieties over~$\F_q$ with commutative endomorphism algebra $E=\Q[\pi]$, or equivalently, of abelian varieties such that the characteristic polynomial of their Frobenius endomorphism $\pi$ is square-free. 
Let $R$ be the order of $E$ generated by the Frobenius and Verschiebung endomorphisms, that is, $R=\Z[\pi,q/\pi]$.

We are interested in the completions $E_p=E\otimes_\Q \Q_p$ of $E$, and $R_p=R\otimes_\Z \Z_p\subseteq E_p$ of $R$, at the rational prime $p$.
For each place $\nu$ of $E_p$, let $e_{\nu}$ denote the ramification index and $f_{\nu}$ the inertia degree of $E_{\nu}$ over $\mathbb{Q}_p$. Also let $n_{\nu}=e_{\nu}f_{\nu}=[E_{\nu}:\mathbb{Q}_p]$ and denote by $\pi_\nu$ the image of $\pi$ in $E_\nu$. Note that $E_p=\prod_{\nu \vert p} E_\nu$.

Let $L$ be the totally unramified extension of $\Q_p$ of degree $a$ and let $W=\mathcal{O}_L$ be the maximal $\Z_p$-order of $L$. 
We know that $L=\Q_p(\zeta_{q-1})$ and $W=\Z_p[\zeta_{q-1}]$, where $\zeta_{q-1}$ is a primitive $(q-1)$-st root of unity.

Observe that $L\otimes_{\Q_p}E_p$ is an etal\'e $\Q_p$-algebra, that is, it is isomorphic to a direct sum of finitely many finite extensions of $\Q_p$.
Indeed,
\begin{equation}\label{eq:splittingA}
     L\otimes_{\Q_p}E_p \simeq \prod_{\nu \vert p} \left( \underbrace{LE_\nu \times \ldots \times LE_\nu}_{g_\nu\text{ copies}} \right) =: A,
\end{equation}
where $g_\nu=\gcd(a,f_\nu)$.
Denote the $\nu$-component of $A$ by $A_\nu$. 
For an element $b = (b_{\nu})_{\nu | p}$ 
of $A = \prod_{\nu} A_{\nu}$ we denote by $b_{\nu,i}$ the $i$-th component of $b_{\nu}$ in $A_\nu=\prod_{j=1}^{g_{\nu}} LE_{\nu}$.
We see $E_\nu$ embeds diagonally into~$A_\nu$ for each~$\nu$. Hence we have an induced embedding
\[ \Delta: E_p \hookrightarrow A, \]
which endows $A$ with an $E_p$-algebra structure.

The isomorphism in Equation \eqref{eq:splittingA} restricted to a component $A_{\nu}$ is given on simple tensors $\omega\otimes\beta \in L \otimes_{\mathbb{Q}_p} E_p$ by
\begin{equation}\label{eq:wat_simple_tensors}
    \omega\otimes \beta \mapsto ( \omega\beta_\nu, \omega^{\sigma}\beta_\nu, \ldots, \omega^{\sigma^{g_\nu-1}}\beta_\nu ),
\end{equation}
where $\sigma$ is the Frobenius of $L$ over $\Q_p$.
This means that $A$ has an $L$-algebra structure, 
and the image of $\lambda\in L$ in each $A_\nu$ equals
\[ 
(\lambda,\lambda^{\sigma},\ldots,\lambda^{\sigma^{g_\nu-1}} ). \]
By slight abuse of notation, we will denote the maps induced by $\sigma$ on $A_\nu$ and $A$ also by $\sigma$. On $A_{\nu}$ the action is as follows:
\begin{equation}\label{eq:def_sigma}
\begin{split}
    ( \omega\beta_{\nu}, \omega^{\sigma}\beta_{\nu}, \ldots, \omega^{\sigma^{g_\nu-1}}\beta_{\nu} )
    & \mapsto
    ( \omega^{\sigma}\beta_{\nu}, \omega^{\sigma^2}\beta_{\nu}, \ldots, \omega^{\sigma^{g_\nu}}\beta_{\nu}).
\end{split}
\end{equation}
Since $\omega^{\sigma^{g_\nu}}\beta_{\nu}=(\omega\beta_{\nu})^{\tau_{\nu}}$, where $\tau_{\nu}$ is the Frobenius of $LE_\nu$ over $E_\nu$, we see that~$\sigma$ acts on each $A_\nu$ as a cyclic permutation followed by $\tau_{\nu}$ in the last component. 

\begin{df}\label{def:Frobprop} 
A map $F_\nu\colon A_\nu \to A_\nu$ will be said to have the \emph{Frobenius property} if it is additive and satisfies $F_\nu\lambda = \lambda^\sigma F_\nu$ for all $\lambda\in L$ and $F_\nu^a = \Delta|_{E_\nu}(\pi_\nu) \in A_\nu $.

Given a set $\mathcal{S}$ of places $\nu$ of $E_p$, a map $F\colon \prod_{\nu \in \mathcal{S}} A_{\nu}\to A_{\nu}$ will be said to have the Frobenius property if $F|_{A_\nu}$ has the Frobenius property for all places $\nu$ of $S$.
\end{df}
    
\begin{lem}\label{lem:norm_sigma}
    For every $\alpha_{\nu}=(\alpha_{\nu,1},\ldots,\alpha_{\nu,g_{\nu}}) \in A_\nu$ we have
    \[ \alpha_{\nu}\cdot\alpha_{\nu}^\sigma\cdots\alpha_{\nu}^{\sigma^{a-1}} = (\beta_{\nu},\ldots,\beta_{\nu})\in A_\nu, \]
    where
    \[ \beta_{\nu} = N_{LE_\nu/E_\nu}(\alpha_{\nu,1}\cdots \alpha_{\nu,g_\nu}) \in E_\nu. \]
\end{lem}
\begin{proof}
    For ease of notation, let $g=g_\nu$, $\alpha = \alpha_{\nu}$ and $N=N_{LE_\nu/E_\nu}$, and write $\tau = \tau_{\nu}$ for the Frobenius of $LE_\nu$ over $E_\nu$. As above, we see $\tau$ has order $a/g$ and satisfies $\sigma^g(\alpha)=(\tau(\alpha_1),\ldots,\tau(\alpha_g))$.
    For each $0\leq k \leq a/g -1$ and each $0\leq i \leq g-1$, we have
    \[ \sigma^{kg+i} (\alpha) = (\tau^k(\alpha_{i+1}),\ldots,\tau^k(\alpha_g),\tau^{k+1}(\alpha_1),\ldots,\tau^{k+1}(\alpha_i)). \]
    Hence
    \begin{align*}
        \alpha\cdot\sigma(\alpha)\cdots\sigma^{a-1}(\alpha) 
        & = \prod_{k=0}^{\frac{a}{g}-1}\prod_{i=0}^{g-1} \sigma^{kg+i} (\alpha) \\
        & = \prod_{i=0}^{g-1} \left( \prod_{k=0}^{\frac{a}{g}-1} \tau^k(\alpha_{i+1}),\ldots,\prod_{k=0}^{\frac{a}{g}-1} \tau^k(\alpha_{g}),\prod_{k=0}^{\frac{a}{g}-1} \tau^{k+1}(\alpha_{1}),\ldots,\prod_{k=0}^{\frac{a}{g}-1} \tau^{k+1}(\alpha_{i})  \right)\\
        & = \prod_{i=0}^{g-1} \left( N(\alpha_{i+1}),\ldots,N(\alpha_g),N(\alpha_1),\ldots,N(\alpha_i)  \right)\\
        & = \left(N(\alpha_1\cdots \alpha_g),\ldots,N(\alpha_1\cdots \alpha_g) \right),
    \end{align*}
    as required. 
\end{proof}

\begin{lem} \label{lem-F} 
    Let $\alpha_\nu \in A_\nu$ be such that
    \[ \alpha_\nu\cdot\alpha_\nu^\sigma\cdots\alpha_\nu^{\sigma^{a-1}} = (\pi_\nu,\ldots,\pi_\nu)\in A_\nu. \quad \quad (*)\]
    For $z\in A_\nu$, define $F_\nu(z) = \alpha_\nu\cdot z^\sigma$, 
    i.e.~$F_\nu=\alpha_\nu \circ \sigma$, and extend this component-wise to obtain a map $F$ on $A$. 
    Then $F$ has the Frobenius property.
\end{lem}
\begin{proof}
    It is enough to check the desired properties on each component $A_\nu$.
    Rather than doing so using the isomorphism given by Equation \eqref{eq:wat_simple_tensors}, we use a different presentation of $A_\nu$.
    Namely, let $m_{\nu}(x)$
    be the minimal polynomial of $\pi_\nu$ over $\Q_p$ and write 
    \[ m_\nu(x) = f_1(x)\cdots f_r(x) \]
    for its factorization into irreducible factors over $L[x]$.
    Then we have an $L$-algebra isomorphism 
    \[ \varphi: A_\nu \overset{\sim}{\longrightarrow} \prod_{i=1}^r \frac{L[x]}{f_i(x)}. \]
    The action of $\sigma$ on $A_\nu$ in this presentation is 
    induced by the automorphism of $L[x]$ sending $s(x) = \sum_{k} a_k x^k$ to $s(x)^\sigma = \sum_k a_i^\sigma x^k$.
    
    For $z\in A_\nu$, we have $\varphi(z)=([s_{z,1}(x)],\ldots,[s_{z,r}(x)])$ for some $s_{z,i}(x) \in L[x]$, where $[s_{z,i}(x)]$ denotes the class of $s_{z,i}(x)$ in the quotient $L[x]/f_i(x)$.  
    For any $\lambda\in L$, 
    \[
    \varphi(F_\nu(\lambda z))=([s_{\alpha_\nu,i}(x) s_{(\lambda z)^\sigma,i}(x)],\ldots)=([\lambda^\sigma s_{\alpha_\nu,i}(x) s_{z,i}(x)^\sigma],\ldots)=\varphi(\lambda^\sigma F_\nu(z)), 
    \]
    which shows that $F_\nu\lambda = \lambda^\sigma F_\nu$, as required.
    Finally, we have by construction that 
    \[ s_{\alpha_\nu,i}(x)\cdot s_{\alpha_\nu,i}(x)^\sigma\cdots s_{\alpha_\nu,i}(x)^{\sigma^{a-1}} s_{z,i}(x) - x \in (f_i(x)), \]
    for $i=1,\ldots,r$. This shows that $\varphi(F_\nu^a) = \varphi(\pi_\nu)$, completing the proof. 
\end{proof}

In \cite[p.~544]{Wat69} there is the following construction of an element $\alpha_{\nu} \in A_{\nu}$ enjoying property~$(*)$ of Lemma~\ref{lem-F}: Since the degree $a/g_{\nu}$ of the unramified extension $LE_\nu/E_\nu$ divides the $\nu$-adic valuation $\val_{\nu}(\pi_{\nu})$, it follows that $\pi_{\nu}$ is the norm of an element $u_\nu\in LE_\nu$.
We can then put $\alpha_\nu=(1,1,\ldots,1,u_\nu)\in A_\nu$.  
The following definition will reappear later in Proposition~\ref{prop:FVstable_OO} and in  Algorithm~\ref{alg:FVclasses}.

\begin{df}\label{def:F-Waterhouse}
If $F_\nu\colon A_\nu \to A_\nu$ is a map with the Frobenius property defined by $F_\nu(z) = \alpha_\nu\cdot z^\sigma$ for an element $\alpha_\nu=(1,1,\ldots,1,u_\nu)\in A_\nu$ then we will say that $F_\nu$ is of \emph{$W$-type}.   

Given a set $\mathcal{S}$ of places $\nu$ of $E_p$, a map $F\colon \prod_{\nu \in \mathcal{S}} A_{\nu}\to A_{\nu}$ will be said to be of $W$-type if $F|_{A_\nu}$ is of $W$-type for all places $\nu$ of $S$.
\end{df}

\subsection{Determining \texorpdfstring{$\mathfrak{X}_{\pi,p}$}{Prp}} \label{sec:Xpip}
In this section we describe $\mathfrak{X}_{\pi,p}$ in terms of isomorphism classes of $W_R\{F,V\}$-ideals, see Theorem~\ref{thm:X_pi_p_WR_F_V}.

\begin{df}
Put $W_R=W\otimes_{\Z_p}R_p$ which is a $\Z_p$-order inside $A$.
Let $\alpha \in A$ be an element such that $F=\alpha\circ\sigma$ has the Frobenius property; this implies that~$F$ is bijective on~$A$. Put $V=p F^{-1}$.
By $W_R\{F,V\}$-ideals, we will mean fractional $W_R$-ideals inside $A$ which are stable under the action of $F$ and $V$. 
Given two $W_R\{F,V\}$-ideals
$J_1,J_2$, let:
\begin{itemize}
    \item $\Hom_{W_R\{F,V\}}(J_1,J_2)$ consist of the $W_R$-linear morphisms from $J_1$ to $J_2$ that commute with the action of $F$ and $V$.
    \item $\Hom_{\Delta}(J_1,J_2) = \{ x \in E_p : \Delta(x)J_1\subseteq J_2\}$.
\end{itemize}
\end{df}

Note that $W_R$ does not in general respect the splitting of $A$ given by Equation \eqref{eq:splittingA}, that is, it cannot be written as a direct sum of components, each lying inside one~$A_\nu$.
On the other hand, the maximal order~$\OO_A$ of $A$, which is the image of $W\otimes \OO_{E_p}$ under~\eqref{eq:splittingA}, does respect the splitting.\\

The following lemma can be seen as a version of Tate’s isogeny theorem (see \cite[Theorem A.1.1.1]{chaiconradoort14}) in this setting. 

\begin{lem} \label{lem-homs}
    Let $J_1$ and $J_2$ be any two fractional $W_R$-ideals. 
    Then $\Delta$ induces a bijection
    \[
    \Hom_{W_R\{F,V\}}(J_1,J_2) 
    \longleftrightarrow \Hom_{\Delta}(J_1,J_2),
    \]
    which is in fact an isomorphism of $R_p$-modules.
\end{lem}
\begin{proof}
    The map $\Delta$ gives an injection of $\Hom_{\Delta}(J_1,J_2)$ into $\Hom_{W_R\{F,V\}}(J_1,J_2)$.
    We now show that this map is also surjective.
    
    Since every $W_R$-linear morphism~$\varphi\colon J_1 \to J_2$ determines a unique $A$-linear endomorphism of $A$, also denoted by $\varphi$, we get that $\varphi$ is actually multiplication by some element $y = (y_\nu)_\nu \in A$, that is, $\varphi(x)=yx$ for every $x \in A$. 
    
    We restrict ourselves to the component $A_\nu$ of $A$. We want to show that since~$y_{\nu} \in A_{\nu}$ commutes with~$F_\nu$, it belongs to $\Delta(E_\nu)$.  
    Put $J_{1,\nu}=J_1 \cap A_\nu$.
    Commuting with~$F_\nu$ means that for every $z \in J_{1,\nu}$ we have
    \[ F_\nu(y_{\nu}z) = y_{\nu}F_\nu(z). \]
    Since $J_{1,\nu}$ and $\OO_{A_\nu}$ are $\Z_p$-lattices of the same rank, the quotient $(J_{1,\nu} + \OO_{A_\nu})/J_{1,\nu}$ is a finite abelian group, say of exponent $n$. 
    Then $\Delta_{|_{E_\nu}}(n)(J_{1,\nu} + \OO_{A_\nu}) \subseteq J_{1,\nu}$, which implies that $\Delta_{|_{E_\nu}}(n)\in J_{1,\nu}$ since $1\in J_{1,\nu} + \OO_{A_\nu}$.
    We have 
    \[
    F(y_{\nu} \Delta_{|_{E_\nu}}(n))=\alpha \, \bigl(y_{\nu,2},y_{\nu,3},\ldots,y_{\nu,g_{\nu}},\tau(y_{\nu,1}) \bigr) \Delta_{|_{E_\nu}}(n)
    \]
    and
    \[
    y_{\nu} F(\Delta_{|_{E_\nu}}(n))=\alpha \,\bigl(y_{\nu,1},y_{\nu,2},\ldots,y_{\nu,g_{\nu}-1}, y_{\nu,g_{\nu}} \bigr) \Delta_{|_{E_\nu}}(n)
    \]
    with $\alpha \neq 0$, which shows that $y_{\nu,1}=\ldots=y_{\nu,g_{\nu}}$ and $y_{\nu,g_{\nu}}=\tau(y_{\nu,1})$ so $y_{\nu} \in \Delta_{|_{E_\nu}}(E_{\nu})$.
\end{proof}

\begin{df}
    Two fractional $W_R$-ideals will be called \emph{$\Delta$-isomorphic} if they are isomorphic as $R_p$-modules via multiplication by an invertible element of $\Delta(E_p)$. 
\end{df} 

\begin{cor}\label{cor:FV_on_classes}
    Two fractional $W_R\{F,V\}$-ideals are isomorphic precisely if they are $\Delta$-isomorphic. 
    Moreover, being stable under $F$ and $V$ is a well-defined notion on $\Delta$-isomorphism classes.
\end{cor}
\begin{proof}
    The first statement is a direct application of Lemma \ref{lem-homs}.
    The second statement follows from the fact that $F$ and $V$ act trivially on the elements of $\Delta(E_p)$.
\end{proof}

From now on, we will use Lemma \ref{lem-homs} to identify $\Hom_{W_R\{F,V\}}(J_1,J_2)$ with the corresponding subset of $E_p$.

\begin{thm}\label{thm:X_pi_p_WR_F_V}
    There is a bijection between $\mathfrak{X}_{\pi,p}$ and the $\Delta$-isomorphism classes of $W_R\{F,V\}$-ideals. 
\end{thm}
\begin{proof}
A Dieudonn\'e module is a $W_R\{F,V\}$-module, with $F$ having the Frobenius property (cf.~Definition~\ref{def:Frobprop}) and $V=pF^{-1}$, that is free over $W$ of rank $2g$, cf.~\cite[Chapter 1]{Wat69}. Morphisms between them are $W_R$-linear morphisms that commute with the action of $F$ and~$V$. Every $W_R$-module that is free over $W$ of rank $2g$ is $W_R$-linearly isomorphic to a fractional ideal in $A$ (which are all free over $W$ of rank $2g$). The result then follows from Lemma~\ref{lem-homs}. 
\end{proof}

\subsection{Isomorphism classes of \texorpdfstring{$W_R\{F,V\}$-ideals}{WRFV} } \label{sec:isoWFV}
\subsubsection{The connected-\'etale sequence} \label{sec:connetale}
For a place $\nu$ of $E_p$ put $s(\nu)=\mathrm{val}_{\nu}(\pi)/(ae_{\nu})$. 
This will be rational number between $0$ and $1$ (which is equal to $i_{\nu}/n_{\nu}$ in the notation of \cite[p. 527]{Wat69}) called the \emph{slope} of $\nu$. 
Let $\frp_{E_\nu}$ denote the maximal ideal of~$\mathcal{O}_{E_p}$ corresponding to $\nu$.

The splitting $\mathcal O_{E_p}=\prod_{\nu|p} \mathcal O_{E_{\nu}}$ does not necessarily descend to $R_p$. 
However, $R_p$ is complete and semilocal, so it can be identified with the direct sum of the completions at its maximal ideals. 
The maximal ideals of $R_p$ are exactly equal to $\frp_{E_\nu} \cap R_p=\{x \in R_p: \mathrm{val}_{\nu}(x)>0 \}$, which might coincide for different $\nu$.
We denote the set of maximal ideals of $R_p$ by $\mathcal{P}_{R_p}$. 

\begin{lem} \label{lem-decomp}
    Let $*$ denote one of $\{0\}$, $\{1\}$ or the open interval $(0,1)$.
    Let $\nu$ and $\mu$ be two places of $E_p$.
    Assume that $s(\nu) \in *$.
    If $\frp_{E_\nu} \cap R_p = \frp_{E_\mu} \cap R_p$ then $s(\mu) \in *$, as well.
\end{lem}
\begin{proof}
We have $R_p=\Z_p[\pi,q/\pi]$. Say that $s(\nu)=0$ and $s(\mu)>0$, then $\pi \in \frp_{E_{\mu}} \cap R_p$ but $\pi \notin \frp_{E_\nu} \cap R_p$, contradiction. Say that $s(\nu)=1$ and $s(\mu)<1$, then $q/\pi \in \frp_{E_\mu} \cap R_p$ but $q/\pi \notin \frp_{E_\nu} \cap R_p$, contradiction. 
\end{proof}

\begin{notn}
    \label{not:splitR}
    We partition
    $\mathcal{P}_{R_p} = \mathcal{P}_{R_p}^{\{0\}} \sqcup \mathcal{P}_{R_p}^{(0,1)} \sqcup \mathcal{P}_{R_p}^{\{1\}}$ according to Lemma~\ref{lem-decomp}.
    That is, for each $* \in \{ \{0\}, \{1\}, (0,1)\}$, let $\mathcal{P}_{R_p}^{*}$ denote the maximal ideals $\frp_{E_\nu} \cap R_p$ of $R_p$ for which $s(\nu) \in *$.
    We decompose
    \[  
    R_p = R_p^{\{0\}} \times R_p^{(0,1)} \times R_p^{\{1\}},
    \]
    where
    \[ R_p^* = \prod_{\frp \in \mathcal{P}_{R_p}^*} R_\frp. \]
    In this product we put $R_p^* = 0$ if $\mathcal{P}^*_{R_p} = \emptyset$.
    For ease of notation, we will write $R'_p = R^{(0,1)}_p$ from now on.
    The decomposition of $R_p$ induces a corresponding decomposition of $E_p$, of~$A$, of~$W_R$, of any fractional $W_R\{F,V\}$-ideal, and of their endomorphism rings. 
    All these decompositions will be denoted analogously to that of $R_p$.
    In particular, we will write $E'$, $W'_R$, $A'$, $\Delta'$, $F'$ and $V'$.
    In the same vein, by a \emph{$\Delta^*$-morphism} we will mean a morphism given by multiplication by an invertible element of $\Delta(E_p^*)$ and~$F^*$ will denote the restriction of $F$ to $A^{*}$.  
    \end{notn}

The decomposition of the endomorphism rings can be deduced from the connected-\'etale sequence, see \cite[Example 3.1.6]{chaiconradoort14}. The three parts corresponding to $\{0\}$, $(0,1)$ and $\{1\}$ will be called the \emph{\'etale part}, the \emph{local-local part}, and the \emph{multiplicative part}, respectively. 

\begin{prop} \label{prop-0and1}
    There is a bijection between $\mathcal{P}_{R_p}^{\{0\}}$ (resp.~$\mathcal{P}_{R_p}^{\{1\}}$) and the maximal ideals of $\Z_p[\pi]$ that do not contain $\pi$ (resp.~of $\Z_p[q/\pi]$ that do not contain $q/\pi$). 
\end{prop}
\begin{proof}
    Note that $\mathcal{P}_{R_p}^{\{0\}}$ consists of the maximal ideals of $R_p$ not containing $\pi$. 
    Let $\frp$ be a maximal ideal of $\Z_p[\pi]$ that does not contain $\pi$.
    Then $q/\pi \in ( \Z_p[\pi] )_\frp$ and hence $R_\frp = ( \Z_p[\pi])_\frp$.
    Since $R_p$ is isomorphic to the direct product of the completions of $R$ at the maximal ideals above $p$, and $(\Z_p[\pi])_\frp$ is local, we get that there is only one maximal ideal of $R$ above~$\frp$, 
    proving bijectivity.
    The proof for $\mathcal{P}_{R_p}^{\{1\}}$ is analogous. 
\end{proof}

\begin{cor}
    There is a bijection between the maximal ideals not containing $\pi$ (resp.~$q/\pi$) of $R_p$ and the irreducible factors $f(x) \neq x$ modulo $p$ of the minimal polynomial of $\pi$.
\end{cor}
\begin{proof}
    There is a bijection between the maximal ideals of $\Z_p[\pi]$ and the irreducible factors  modulo $p$ of the minimal polynomial of $\pi$, see \cite[Theorem 8.2]{psh08}. An irreducible factor $f(x)$ corresponds to the maximal ideal $p\Z_p[\pi]+f(\pi)\Z_p[\pi]$. The maximal ideals that do not contain $\pi$ therefore correspond to the irreducible factors different from $x$. The result now follows from Proposition~\ref{prop-0and1}.
\end{proof}

\begin{prop}\label{prop-01}
    The set $\mathcal{P}_{R_p}'$ is non-empty if and only if the isogeny class $\mathcal{A}_\pi$ is non-ordinary.
    If so, then $\mathcal{P}_{R_p}'$ consists of one element, the maximal ideal $\frp=(p,\pi,q/\pi)$ of $R_p$, which has residue field $\F_p$.
    If $q=p^a$ is not prime, that is, $a>1$, then $\frp$ is singular and $R_p$ is not the maximal order of $E_p$.
\end{prop}
\begin{proof}
    The first statement is clear.
    Assume for the rest of the proof that the isogeny class is non-ordinary.
    If we let $\frp=(p,\pi,q/\pi)$ then
    \begin{equation}\label{eq:primefield}
        \frac{R_p}{\frp} = \frac{\Z_p[\pi,q/\pi]}{(p,\pi,q/\pi)} \simeq \F_p.
    \end{equation} 
    Hence $\frp$ is a maximal ideal of $R_p$ with residue field $\F_p$.
    Let $\frp_{E_\nu}$ be any maximal ideal of $\OO_{E_p}$ associated to a place of slope in the interval $(0,1)$.
    Since both $\pi$ and $q/\pi$ belong to $\frp_{E_\nu}$, we have that $\frp_{E_\nu}\cap R_p = \frp$.
    So~$\frp$ is the unique element of $\mathcal{P}_{R_p}'$.

    Assume that $a>1$. We now show the last statement by contradiction: assume that $\frp$ is regular, that is, $R_{\frp}$ is a DVR, with valuation $\val_\frp$.
    Let $s = \val_\frp(\pi)/ae_\frp$ be the slope of the place associated with~$\frp$.
    If $0<s<1/2$ then there must a place of slope $1-s$, see Subsection~\ref{sec:duality}, contradicting the fact that $\mathcal{P}_{R_p}'$ consists of one element.
    Hence $s = 1/2$.
    By Equation~\eqref{eq:primefield}, we have  
    \[ e_\frp = e_\frp f_\frp = [E_\frp:\Q_p] >1, \]
    where the inequality follows from the fact that the endomorphism ring is commutative and so  
    \[ f_\frp \val_\frp(\pi)/a=s [E_\frp:\Q_p] = \frac{1}{2}[E_\frp:\Q_p] \in \Z, \]
    see \cite[p.~527]{Wat69}.
    Hence $e_\frp\geq 2$ which implies that $p\in \frp^2$.
    Since $\frp$ is fixed by the automorphism of $E_p$ sending $\pi$ to $q/\pi$, we get that
    \[ \val_\frp(\pi) = \val_\frp(q/\pi) = s a e_\frp = \frac{1}{2}ae_\frp \geq a >1. \]
    This implies that $\pi,q/\pi \in \frp^2$ as well.
    It follows that $\frp = \frp^2$ which is impossible.
    Therefore, $\frp$ is a singular maximal ideal and, hence, $R_p$ is not the maximal order of $E_p$.
\end{proof}
Conversely, it is not true that if $\frp$ is singular, then $q$ is not a prime; see Example~\ref{ex:end_max_p}.

\subsubsection{Duality} \label{sec:duality}
The CM involution $\pi\mapsto \bar\pi=q/\pi$ induces an involution $\phi\colon E_p \to E_p$ by $\phi(\sum_{i,j} a_{i,j}\pi^i \bar{\pi}^j)=\sum_{i,j} a_{i,j}\pi^j \bar{\pi}^i$ with $a_{i,j} \in \Q_p$.
If $\nu$ is a place of $E_p$ then $\bar \nu=\nu \circ \phi$ is also a place of $E_p$ with $s(\bar \nu)=1-s(\nu)$. We also define $\phi\colon A \to A$ by $\phi(\sum_{i,j} a_{i,j,k} \zeta^k_{q-1} \pi^i \bar{\pi}^j)=\sum_{i,j} a_{i,j,k} \zeta^k_{q-1}\pi^j \bar{\pi}^i$. 
The proof of the following lemma is straightforward and therefore omitted. 

\begin{lem} \label{lem:duality} Choose any $F_\nu=\alpha \circ \sigma$ with the Frobenius property. Then $F_{\bar \nu}=\phi(\alpha) \circ \sigma$ will have the Frobenius property, and $\phi$ induces a bijection, taking  $W_{R_\nu}\{F_\nu,V_\nu\}$-ideals $I$ to $W_{R_{\bar \nu}}\{F_{\bar \nu},V_{\bar \nu}\}$-ideals $\phi(I)$. 
\end{lem}

\subsubsection{The \'etale and multiplicative parts} \label{seq:etalemult}
By Theorem~\ref{thm:X_pi_p_WR_F_V}, Dieudonn{\'e} modules in $\mathfrak{X}_{\pi,p}$ are in bijection with $W_R\{F,V\}$-ideals. In this section, we show that we can describe the \'etale and multiplicative part of the Dieudonn\'e module, that is when $* = \{0\}$ or $* = \{1\}$, in terms of $R_p^*$-ideals in $E_p^*$ rather than $W_R^*\{F^*,V^*\}$-ideals in $A^*$.

\begin{prop} \label{prop-ordinary-ideals} 
If $*$ is either $\{0\}$ or $\{1\}$ then one can choose $F^*$ with the Frobenius property such that the following hold.
\begin{enumerate}[(i)]
    \item \label{prop-ordinary-ideals:ext}If $I$ is an $R^*_p$-ideal then $W \otimes_{\mathbb{Z}_p} I$ is a $W_R^* \{F^*,V^*\}$-ideal.
    \item \label{prop-ordinary-ideals:meet} If $J$ is a $W_R^* \{F^*,V^*\}$-ideal then $J = W \otimes_{\mathbb{Z}_p} (\Delta^*)^{-1}(J \cap \Delta^*(E_p^*))$.
    \item \label{prop-ordinary-ideals:homs} For any $W_R^* \{F^*,V^*\}$-ideals $J_1,J_2$ there is a bijection 
    \[
\Hom_{\Delta^*}(J_1 \cap \Delta^*(E_p^*),J_2 \cap \Delta^*(E_p^*)) \longleftrightarrow 
\Hom_{W_R^*\{F^*,V^*\}}(J_1,J_2). 
\]
\end{enumerate}
\end{prop}

\begin{remark}
 Note that $F^*$ chosen in Proposition~\ref{prop-ordinary-ideals}, and in Corollary~\ref{cor-01} below, is not necessarily of $W$-type. 
\end{remark}

\begin{proof}\
\begin{enumerate}[(i)]
    \item We first prove the case $* = \{0\}$. Arguing as in the proof of \cite[Theorem 7.4]{Wat69}, we can find an invertible element $\alpha$ of $W_R^{\{0\}}$ such that $F^{\{0\}}=\alpha \circ \sigma$ has the Frobenius property. 
    Since $\sigma(W \otimes I) \subseteq W \otimes I$ it follows that $F^{\{0\}}(W \otimes I) \subseteq W \otimes I$. Furthermore, $p\alpha^{-1}$ is in $W_R^{\{0\}}$ since $\alpha$ is a unit, and hence $V^{\{0\}}(W \otimes I) \subseteq W \otimes I$, as well. 
    
    Now consider $* = \{1\}$. Since $\pi/q$ is a unit in $R_p^{\{1\}}$, \cite[Theorem 7.4]{Wat69} can be used in the same way to find an invertible element $u$ of $W_R^{\{1\}}$ such that $F^{\{1\}}=(pu) \circ \sigma$ has  the Frobenius property. 
    It then follows in the same way as above that $W \otimes I$ will be stable under $F^{\{1\}}$ and $V^{\{1\}}$.
    \end{enumerate}
    
     In the rest of this proof we use terminology and results of \cite[Chapter 0]{Greither} to carry out a descent argument. 
     \begin{enumerate}[(i)]
     \setcounter{enumi}{1}
     \item  
        Let $G$ be the Galois group of $L=\Q_p(\zeta_{q-1})$ over $\Q_p$. We have that $G=\{\sigma_j \}_{j=1, \ldots, a}$ with $\sigma_j(\zeta_{q-1})=\zeta_{q-1}^{p^j}$. The group $G$ acts on $W=\Z_p[\zeta_{q-1}]$ and therefore has an induced action on $W \otimes_{\Z_p} R_p^*$. Moreover, $R_p^*=(W \otimes_{\Z_p} R_p^*)^G$ because $W^G=\Z_p$. 
        
        For any $i=1,\ldots,p^a-1$, let $x_i=(p^a-1)^{-1} \zeta_{q-1}^{-i} \in W$ and $y_i=\zeta_{q-1}^i \in W$. Any (not necessarily primitive) $(p^a-1)$-st root of unity except $1$ is a root of $x^{p^a-2}+x^{p^a-3}+\ldots +1$ and hence, 
        \[
        \sum_{i=1}^{p^a-1} x_i \sigma_j(y_i)=\frac{1}{p^a-1}\sum_{i=1}^{p^a-1} (\zeta_{q-1}^{p^j-1})^i=\begin{cases} 1 \text{ if } j=a, \\ 0 \text{ if } j = 1,\ldots, a-1.\end{cases}
        \]
        
        Let $J$ be any $W_R^*\{F^*,V^*\}$-ideal. It follows from \cite[Theorem 1.6.(ii')]{Greither} 
        that $W \otimes_{\Z_p} R_p^*$ over $R_p^*$ is a $G$-Galois-extension of commutative rings. 
        Moreover, $J$ is stable under~$G$, in fact $J^G=J \cap \Delta(E_p^*)$. Consider the descent datum on $J$ defined by $\Phi_{\sigma_j}=\sigma_j$ for all $j=1,\ldots, a$. From \cite[Theorem 7.1]{Greither} it follows that $J = W \otimes (\Delta^*)^{-1}(J \cap \Delta^*(E_p^*))$. 
        \item The set $\Hom_{\Delta^*}(J_1 \cap \Delta^*(E_p^*),J_2 \cap \Delta^*(E_p^*))$ can be identified with the set of $x \in E_p^*$ such that $\Delta^*(x) (J_1 \cap \Delta^*(E_p^*)) \subseteq J_2 \cap \Delta^*(E_p^*)$. By Lemma~\ref{lem-homs} we may identify $ \Hom_{W_R^*\{F^*,V^*\}}(J_1,J_2)$ with the set of $x \in E_p^*$ such that $\Delta^*(x) J_1 \subseteq J_2$. Since $\Phi_{\sigma_j}\Delta^*(x)=\Delta^*(x)\Phi_{\sigma_j}$ for any $j$, we conclude from \cite[Theorem 7.2]{Greither} that the map $x \mapsto \Delta^*(x)$ induces a bijection.
\end{enumerate} 
\end{proof}

\begin{cor} \label{cor-01} 
If $*$ is either $\{0\}$ or $\{1\}$ then one can choose $F^*$ with the Frobenius property such that the map $\psi: [I]_{R_p^*} \mapsto [I \otimes W]_{\Delta^*}$ is a bijection from isomorphism classes of $R^*_p$-ideals to $\Delta^*$-isomorphism classes of $W_R^* \{F^*,V^*\}$-ideals. 
\end{cor}
\begin{proof} It follows from Proposition~\ref{prop-ordinary-ideals}.\ref{prop-ordinary-ideals:ext} and~\ref{prop-ordinary-ideals:homs} that the map $\psi$ is well defined and injective, since the bijection in Proposition~\ref{prop-ordinary-ideals}.\ref{prop-ordinary-ideals:homs} restricts to a bijection on the subsets of $\mathrm{Hom}_{\Delta^*}$ resp.~$\mathrm{Hom}_{W_R^*}$ consisting of isomorphisms. 
Surjectivity follows from Proposition~\ref{prop-ordinary-ideals}.\ref{prop-ordinary-ideals:meet}.
\end{proof}

\subsubsection{The local-local part in the maximal endomorphism ring case} \label{sec:locallocal}
In \cite[proof of Theorem~5.1]{Wat69}, Waterhouse gave an explicit description of isomorphism classes of fractional $\OO_{A'}$-ideals which are stable under the action of $F'$ and $V'$. We paraphrase it in (i) of Proposition~\ref{prop:FVstable_OO} below, adding a formula for their number in (ii).

Each fractional $\OO_{A'}$-ideal is a direct sum of fractional $\OO_{A_\nu}$-ideals.
Moreover, the action of $F'$ and $V'$ preserves $\nu$-components.
Hence it is enough to focus our attention on a single $\nu$-component $A_\nu$ of $A'$.
Let $t_\nu$ be a uniformizer of $E_\nu$.
Then $t_\nu$ is also a uniformizer of $LE_\nu$ and every fractional $\OO_{A_\nu}$-ideal is uniquely determined by its generator, which has the form $(t_\nu^{\varepsilon_1},\ldots,t_\nu^{\varepsilon_{g_\nu}})$ for some nonnegative integers $\epsilon_1, \ldots, \epsilon_{g_{\nu}}$.

\begin{prop}\label{prop:FVstable_OO}
    For each place $\nu$ of $E$ with $s(\nu) \in (0,1)$ consider the set $\chi_\nu$ all ordered $g_\nu$-tuples
    \[ 0\leq n_1, \ldots, n_{g_\nu}\leq e_\nu \]
    such that
    \[ \sum_{i=1}^{g_\nu} n_i = \frac{g_{\nu} \val_\nu(\pi_\nu)}{a}. \]
    For each such tuple, set $\varepsilon_1=0$ and $\varepsilon_i=\varepsilon_{i-1}+n_{i-1}$, for $i=2,\ldots,g_\nu$.
    
    Fix any $F'$ of $W$-type (see Definition~\ref{def:F-Waterhouse}).
    Then:
    \begin{enumerate}[(i)]
        \item
            The map $\upsilon_\nu$ given by $(\varepsilon_1,\ldots,\varepsilon_{g_\nu}) \mapsto (t_\nu^{\varepsilon_1},\ldots,t_\nu^{\varepsilon_{g_\nu}})$ induces a bijection between $\chi'=\prod_{s(\nu) \in (0,1)} \chi_\nu$ and the set of $\Delta'$-isomorphism classes of fractional $W'_R\{F',V'\}$-ideals with multiplicator ring $\OO_{A'}$; 
        \item 
            We have 
            \begin{equation} \label{eq-numb}
            \vert \chi_{\nu}\vert =\sum_{i=0}^{d} (-1)^i \binom{g_{\nu}}{i}\binom{g_{\nu} \val_\nu(\pi_\nu)/a-i(e_{\nu}+1)+g_{\nu}-1}{g_{\nu}-1}, 
            \end{equation}
            with $d=\min(g_{\nu},\lfloor g_{\nu}\val_\nu(\pi_\nu)/(a(g_{\nu}+1))\rfloor )$.
    \end{enumerate}
\end{prop}
\begin{proof}
    The set $\chi_{\nu}$ is defined as in \cite[Theorem~5.1]{Wat69}.
    In the proof of \emph{loc.~cit.}, it is shown that $\upsilon_\nu$ can be found in the wanted form, and that an ideal determined by $g_\nu$-tuples of the form 
    $(t_\nu^{\varepsilon_1},\ldots,t_\nu^{\varepsilon_{g_\nu}})$
    is a $W'_R\{F',V'\}$-ideal if and only if, for all $\nu$, the number
    $\varepsilon_1$ is any integer and $\varepsilon_i$ for $i=2,\ldots,g_\nu$ are as in the statement.
    Since we are working up to $\Delta'$-isomorphism, we can scale and hence assume that $\varepsilon_1=0$.
    
    The cardinality $\vert \chi_{\nu} \vert$ is equal to the number of ways of writing $m=g_{\nu}\val_\nu(\pi_\nu)/a$ as a sum of positive integers (i.e., weak compositions) with each positive integer bounded by $e_{\nu}$. The formula~\eqref{eq-numb} can be found by noticing that it is also equal to the coefficient of $x^{m}$ in the polynomial
    \[ 
    (1+x+\ldots+x^{e_{\nu}})^{g_{\nu}}=\frac{(1-x^{e_{\nu}+1})^{g_{\nu}}}{(x-1)^{g_{\nu}}};
    \]
    the coefficient of $x^{i(e_{\nu}+1)}$ in the numerator of the rational expression equals $(-1)^i \binom{g_{\nu}}{i}$, and the coefficient $x^j$ for $j \geq 0$ in the expansion of the denominator equals $\binom{j+g_{\nu}-1}{g_{\nu}-1}$. One then considers all choices of $i,j$ such that $i(e_{\nu}+1)+j=m$. For a reference, see for instance \cite[Equation (4.0)]{Abramson}.
\end{proof}

\begin{df}\label{def:Upsilon} Using the notation from Proposition~\ref{prop:FVstable_OO}, put 
\[
\Upsilon=\{\prod_{s(\nu)\in (0,1)} \upsilon_\nu(\epsilon_1,\ldots,\epsilon_{g_\nu}): (\epsilon_1,\ldots,\epsilon_{g_\nu})\in \chi_\nu \text{ for all $\nu$ such that } s(\nu) \in (0,1) \}.
\]   
\end{df}

\section{Equivalences of categories} \label{sec:equiv}
In this section we will extend the results of Section~\ref{sec-abvars} to give a categorical equivalence in Theorem~\ref{thm-equivalence} between the isogeny class $\mathcal A_{\pi}$ and a category $\mathcal C_\pi$ of pairs of ideals, see Definition~\ref{def:Cpi}. The pairs of ideals will respectively give representatives of the Tate modules and the Dieudonn\'e module of the corresponding abelian variety. 

\begin{df}\label{def:Cpi}
    For each $\ell$, let $i_{\ell}$ be the injection $E \to E_{\ell}$. Let $\mathcal C_{\pi}$ be the category whose objects are pairs $(I,M)$, where $I$ is a fractional $R$-ideal in $E$, and $M$ is a $W_R\{F,V\}$-ideal such that $\Delta^{-1}(M)=i_p(I)R_p$. 
    The morphisms between objects $(I,M)$ and $(J,N)$ in $\mathcal C_{\pi}$ are the elements $\alpha \in E$ such that $\alpha I \subseteq J$ and $\Delta(i_p(\alpha)) M \subseteq N$.
\end{df}

\begin{thm} \label{thm-equivalence} There is an equivalence of categories $\Psi\colon \mathcal C_{\pi} \to \mathcal A_{\pi}$.
\end{thm}
\begin{proof}
Recall that we have fixed an abelian variety $X$ in $\mathcal A_{\pi}$.
We identify $E_\ell$ with $V_\ell(X)$ for each $\ell \neq p$.
Similarly, we identify $M(X)$ with a $W_R\{F,V\}$-ideal and $A$ with $M(X)\otimes_{\Z_p} \Q_p$.
For each object $(I,M)$ in $\mathcal C_{\pi}$, we fix positive integers $k$, $n_p$, and $n_{\ell}$ for all $\ell \neq p$, such that
\[ 
\ell^{n_{\ell}} T_{\ell}(X) \subseteq k \cdot i_{\ell}(I) \subseteq T_{\ell}(X) 
\quad\text{and}\quad 
p^{n_p} M(X) \subseteq k \cdot M  \subseteq M(X).
\]
For all $\ell \neq p$, let $K_{\ell}(I)$ be the subgroup of $X(\overline\F_p)[\ell^{n_{\ell}}] = T_{\ell}(X)/\ell^{n_{\ell}} T_{\ell}(X)$ equal to $k \cdot i_{\ell}(I)/\ell^{n_{\ell}} T_{\ell}(X)$. 
By the Dieudonn\'e-Cartier-Oda correspondence, the quotient $k\cdot M/p^{n_p}M(X)$ uniquely determines a $p$-power order subgroup scheme $K_{p}(M)$ of $X$ (see~\cite[Section 1]{Wat69} for details).
Let $K$ be the subgroup scheme of $X$ generated by all $K_{\ell}(I)$ together with $K_{p}(M)$. 
Define the abelian variety $X_{(I,M)}=X/K$ in $\mathcal A_{\pi}$.
By Tate's theorem, there exists an isogeny $\varphi\colon X_{(I,M)} \to X$ such that $\varphi_*(T_{\ell}(X_{(I,M)}))=k \cdot i_\ell(I) \subseteq T_{\ell}(X)$ and $\varphi_*(M(X_{(I,M)}))=k \cdot M \subseteq M(X)$.
Finally, define the quasi-isogeny $\varphi_{(I,M)}=\varphi/k\colon X_{(I,M)} \to X$ and put $\Psi((I,M))=X_{(I,M)}$. 

Let $\alpha\colon (I,M) \to (J,N)$ be a morphism between objects of $\mathcal C_{\pi}$. 
We will first only consider $\alpha \in E^\times$, which will correspond via $\Psi$ to isogenies in $\mathcal A_{\pi}$. The pair $(\alpha I,i_p(\alpha) M)$, which belongs to~$\mathcal C_\pi$ since $\alpha \in E^\times$, gives rise (by Tate's theorem, as above) to a quasi-isogeny $\varphi_{\alpha}\colon X_{(\alpha I,i_p(\alpha) M)} \to X$. 
Multiplication with $\alpha \in E^{\times}$ is then the same as an isomorphism $\epsilon_{\alpha}\colon X_{(I,M)} \to  X_{(\alpha I,i_p(\alpha) M)}$ such that $\varphi_{(I,M)}=\varphi_{\alpha} \circ \epsilon_{\alpha}$.
Again by Tate's theorem, the inclusion $(\alpha I,i_p(\alpha) M) \subseteq (J,N)$ gives rise to an isogeny $\epsilon\colon X_{(\alpha I,i_p(\alpha) M)} \to X_{(J,N)}$ such that $\varphi_{(I,M)}=\varphi_{(J,N)} \circ  \epsilon \circ \epsilon_{\alpha}$. 
We put $\Psi(\alpha\colon (I,M) \to (J,N))=\epsilon \circ \epsilon_{\alpha}$.

Since $\Psi$ respects composition and the identity, it is a functor.
We now show that $\Psi$ is essentially surjective. 
For any $Y\in \mathcal{A}_\pi$, choose an isogeny $\varphi:Y\to X$.
Then $\varphi_*(T_\ell(Y))$ is a fractional $R_\ell$-ideal in $E_\ell$, equal to $T_\ell(X)$ for all but finitely many primes $\ell \neq p$.
Similarly, $\varphi_*(M(Y))$ is a $W_R\{F,V\}$-ideal.
Let $M = \varphi_*(M(Y))$.
There exists a unique fractional $R$-ideal $I$ such that $i_\ell(I) = \varphi_*(T_\ell(Y))$ and $i_p(I) = \Delta^{-1}(M)$.
By construction, $(I,M)\in \mathcal{C}_\pi$ and $\Psi((I,M)) \simeq Y$.

Each isogeny $X_{(I,M)} \to X_{(J,N)}$ is equal to $\varphi_{(J,N)}^{-1} \circ \alpha \circ \varphi_{(I,M)}$ for some $\alpha \in E^{\times}$, viewed as a quasi-isogeny of $X$, and necessarily $\alpha I \subseteq J$ and $\Delta(i_p(\alpha)) M \subseteq N$. Moreover, if $\alpha \neq \beta \in E^{\times}$ we see directly that $\Psi(\alpha\colon (I,M) \to (J,N))\neq \Psi(\beta\colon (I,M) \to (J,N))$. 
This shows that $\Psi$ is both full and faithful when restricted to morphisms of $\mathcal C_\pi$ in $E^\times$ and isogenies in $\mathcal A_\pi$. 

Let us now consider any $\alpha \in E$. The isogeny decomposition $X \sim X_1 \times \ldots \times X_r$ gives a corresponding decomposition $E=E_1 \times \ldots \times E_r$. An element $\alpha=(\alpha_1,\ldots,\alpha_r) \in E$ is in $E^\times$ precisely if $\alpha_i \neq 0$ for all $i=1,\ldots,r$. Moreover, any homomorphism between simple factors $X_i$ and $X_j$ is either an isogeny (and then necessarily $i=j$), or the zero morphism. It is now immediate that $\Psi$ is both full and faithful on all morphisms. 
\end{proof}

\begin{df} For any overorder $S \supseteq R$ in $E$, let $\mathcal D_{\pi}(S)$ denote the category of fractional $S$-ideals with homomorphisms between two $S$-ideals $I,J$ being $\alpha \in E$ such that $\alpha I \subseteq J$. 
\end{df}

The following corollary reproves the second part of \cite[Corollary 4.4]{MarAbVar18}, which builds upon \cite{CentelegheStix15} and \cite{Del69}.

\begin{cor} Assume that (at least) one of the following conditions holds: 
\begin{enumerate}[(i)]
    \item The number $q=p$ is a prime;  
    \item The category $\mathcal A_\pi$ consists of ordinary abelian varieties, 
\end{enumerate} 
then there is an equivalence of categories between $\mathcal D_{\pi}(R)$ and $\mathcal A_{\pi}$, inducing a bijection between $\mathrm{ICM}(R)$ and $\mathcal A_{\pi}^{\mathrm{isom}}$. 
\end{cor}
\begin{proof} If either of the conditions holds, define a functor $\Xi\colon \mathcal D_{\pi}(R) \to \mathcal A_{\pi}$ by letting $\Xi(I)=\Psi(I,W \otimes I)$ and $\Xi(\alpha\colon I \to J)=\Psi(\alpha\colon I \to J)$.

If $q=p$ is a prime, then $W=\mathbb Z_p$ and $W_R\{F,V\}$-ideals are nothing but $R_p$-ideals. Theorem~\ref{thm-equivalence} then shows that $\Xi$ is an equivalence. 

Say now that $\mathcal A_\pi$ consists of ordinary abelian varieties. Proposition~\ref{prop-ordinary-ideals} shows that all $W_R\{F,V\}$-ideals are of the form $W \otimes I$, with $I$ a fractional $R$-ideal.  Theorem~\ref{thm-equivalence} then shows that also in this case $\Xi$ is an equivalence. 
\end{proof}

In the next result we restrict ourselves to abelian varieties such that the local-local part of their Dieudonn\'e module has maximal endomorphism ring. This gives an equivalence that is similar in spirit to, for instance, \cite[Theorems 5.1 and 5.3]{Wat69}, \cite[Theorems 1.1 and 4.5]{OswalShankar19} (see further in Example~\ref{ex:almost_ord}) and \cite[Theorem 1.3]{BhatnagarFu}.

\begin{cor} \label{cor-max} 
Let $S$ denote the smallest order in $\mathcal O_E$ containing $R$ and such that $S_p$ contains $\OO_{E'_p}$.
Let $r$ be the number of $\Delta'$-isomorphism classes of $W'_R\{F',V'\}$-ideals whose endomorphism ring is $\mathcal O_{E'_p}$.
Then there are full and faithful functors $\Psi_i\colon \mathcal D_{\pi}(S) \to \mathcal A_{\pi}$ for $i=1,..,r$ that induce a bijection between the isomorphism classes of abelian varieties $Z$ in $\mathcal{A}_{\pi}$ with $\mathcal O_{E'_p} \subseteq \End(Z)_p$ and the disjoint union of $r$ copies of $\mathrm{ICM}(S)$. 
\end{cor}
\begin{proof} 
Let $M'_1,\ldots,M'_r$ be representatives of each $\Delta'$-isomorphism class of $W'_R\{F',V'\}$-ideals whose endomorphism ring contains $\OO_{E'_p}$. For each $i=1,\ldots,r$ we then have $\Delta'^{-1}(M'_i)=\delta_i\OO_{E'_p}$, for some $\delta_i\in {E'_p}^\times$. Choose $F^{\{0\}}$ and $F^{\{1\}}$ as in Proposition~\ref{prop-ordinary-ideals} so that, for each $i = 1, \ldots, r$, 
\[
M_i=W_R^{\{0\}} \oplus \Delta'(\delta_i^{-1})M'_i \oplus W_R^{\{1\}}
\]
is a $W_R\{F,V\}=W^{\{0\}}_R\{F^{\{0\}},V^{\{0\}}\} \times W'_R\{F',V'\} \times W^{\{1\}}_R\{F^{\{1\}},V^{\{1\}}\}$-ideal. It follows that $\Delta^{-1}(M_i)=R_p^{\{0\}} \times \OO_{E'_p} \times R_p^{\{1\}}$ and so, for any fractional $S$-ideal $I$, it follows that 
$\Delta^{-1}(\Delta(i_p(I))M_i)=i_p(I) \, \Delta^{-1}(M_i)=i_p(I)(R_p^{\{0\}} \times \OO_{E'_p} \times R_p^{\{1\}})=i_p(I)S_p$.  
We conclude that $(I,\Delta(i_p(I))M_i)$ is an object in $\mathcal{C}_{\pi}$.

For each $i=1,\ldots,r$, define a functor $\Psi_i\colon \mathcal D_{\pi}(S) \to \mathcal A_{\pi}$ by putting $\Psi_i(I)=\Psi((I,\Delta(i_p(I))M_i))$ and $\Psi_i(\alpha\colon I \to J)=\Psi(\alpha\colon (I,\Delta(i_p(I))M_i) \to (J,\Delta(i_p(J))M_i))$. That these functors are full and faithful follows from the corresponding facts for $\Psi$. 

Finally, take any abelian variety $Z$ in $\mathcal{A}_{\pi}$ with $\OO_{E'_p} \subseteq \End(Z)_p$. Then there is an object $(I,M)$ in $\mathcal C_{\pi}$ such that $\Psi((I,M))$ is isomorphic to $Z$. 
Write $I'_p=(i_{p}(I)\cap E'_p)R'_p$.
Since $I'_p$ is a principal $\OO_{E'_p}$-ideal, there is a $\beta \in E'_p$ such that $\Delta'(\beta) M'=\Delta'(I'_p) M'_j$
for some $j \in \{1,\ldots,r\}$. Then 
\begin{equation} \label{eq:beta}
    \beta I'_p=\beta \Delta'^{-1}(M')=I'_p\Delta'^{-1}(M'_j)=I'_p\OO_{E'_p}=I'_p.
\end{equation}
By the strong approximation theorem we can find a new element $ \beta \in E$ such that \eqref{eq:beta} still holds, and moreover $\beta_\ell \in R^\times_\ell$, for any $\ell \neq p$, and $\beta_p^* \in (R_p^{*})^\times$ for $*\in\{\{0\},\{1\}\}$. 
Then put $\gamma=\beta \gamma'$ for any $\gamma' \in (I:I)^\times$.
It follows that projecting $\Delta(i_p(\gamma)) M$ to $A'$ gives $I'_p M'_j$,
and that $\gamma I=I$. By Proposition~\ref{prop-ordinary-ideals} we get that if $N_1$ and $N_2$ are any $W^{\{0\}}_R\{F^{\{0\}},V^{\{0\}}\}$-ideals such that $(\Delta^{\{0\}})^{-1}(N_1)=(\Delta^{\{0\}})^{-1}(N_2)$ then $N_1=N_2$. The corresponding statement also holds for $W^{\{1\}}_R\{F^{\{1\}},V^{\{1\}}\}$-ideals. It follows that $\Delta(\gamma) M=\Delta(i_p(I))M_j$. Hence, $\Psi_j(I)$ is isomorphic to $\Psi((I,M)) \simeq Z$ in $\mathcal A_{\pi}$, via $\gamma \in E$. 
\end{proof}

\begin{example}\label{ex:almost_ord} 
If $\mathcal A_\pi$ consists of almost-ordinary abelian varieties, then (using the duality of Subsection~\ref{sec:duality}) there is a unique place $\nu$ above~$p$ with slope $s(\nu) \in (0,1)$ and it will have slope $1/2$.
It follows from \cite[Proposition 2.1]{OswalShankar19} that the endomorphism ring of all $W'_R\{F',V'\}$-ideals contains $\OO_{E'_p}$.  
Corollary~\ref{cor-max} can then be compared with \cite[Theorem 4.5.(3)]{OswalShankar19}. Equation~\eqref{eq-numb} shows that if $\nu$ is the place above $p$ with $s(\nu)=1/2$, then the number~$r$ of $\Delta'$-isomorphism classes of $W'_R\{F',V'\}$-ideals whose endomorphism ring contains $\mathcal O_{E'_p}$ equals $r=2$ if $e_{\nu}=1$, and equals $r=1$ if $e_{\nu}=2$. Note that Corollary~\ref{cor-max} holds also for even $q$, while \cite[Theorem 4.5.(3)]{OswalShankar19} does not. 

In the case when $e_{\nu}=1$, we can use Theorem~\ref{thm-equivalence} to describe isogenies between two abelian varieties $X_1=\Psi_1(I)$ and $X_2=\Psi_2(J)$. Say that $M'_1$ is generated as a $\OO_{A_\nu}$-ideal by $(1,1) \in A_\nu=LE_\nu \oplus LE_\nu$, and $M'_2$ by $(1,t_\nu) \in A_\nu$, where $t_\nu$ is a uniformizer of $E_\nu \subseteq LE_\nu$ (compare with Proposition~\ref{prop:FVstable_OO}). 
We see that every $\alpha \in (J:I)\cap E^\times$ induces an isogeny $X_1\to X_2$, but $\beta \in (I:J)\cap E^\times$ induces an isogeny $X_2\to X_1$ if and only if $\val_\nu(\beta)>0$.
These types of isogenies are not described in \cite{OswalShankar19}. 

Note finally that if $q$ is not prime, then $R'_p$ is singular by Proposition~\ref{prop-01}. So, in this case $R$ cannot be an endomorphism ring of an almost-ordinary abelian variety. 
\end{example}

In the following two results we will draw some consequences from Theorem~\ref{thm-equivalence} on the endomorphism rings of abelian varieties in $\mathcal A_{\pi}$.

\begin{prop} \label{prop:extension_are_end} 
Say that there is a $W'_R\{F',V'\}$-ideal $M'$ with endomorphism ring $T' \subseteq E'_p$ such that $\Delta'^{-1}(M')$ is invertible in $T'$. Let $S$ be any overorder of $R$ such that $S_p'$ contains $T'$. Then there is an abelian variety in $\mathcal A_{\pi}$ with endomorphism ring $S$.
\end{prop}
\begin{proof} 
Choose $F^{\{0\}}$ and $F^{\{1\}}$ as in Proposition~\ref{prop-ordinary-ideals} and let $M$ be the $W_R\{F,V\}=W^{\{0\}}_R\{F^{\{0\}},V^{\{0\}}\} \times W'_R\{F',V'\} \times W^{\{1\}}_R\{F^{\{1\}},V^{\{1\}}\}$-ideal $(W_R^{\{0\}} \otimes R_p^{\{0\}}) \oplus M' \oplus (W_R^{\{1\}} \otimes R_p^{\{1\}})$. Fix the fractional $R$-ideal~$I$ such that $i_{\ell}(I)R_{\ell} = R_{\ell}$ for all $\ell \neq p$ and such that $i_p(I)R_p=\Delta^{-1}(M)$, whose existence is guaranteed by \cite[Theorem 9.4.9, Lemma 9.5.3]{JVQuat}. Since $\Delta^{-1}(\Delta(i_p(S))M)=i_p(S)\Delta^{-1}(M)=i_p(SI)R_p$ we have that $(SI,\Delta(i_p(S))M) \in \mathcal C_\pi$, 
and since $I$ is locally principal for all~$\ell$ (including $p$), we have that $(SI:SI)=(S:S)=S$.
The abelian variety $\Psi((SI,\Delta(i_p(S))M))$ then has endomorphism ring $S$. 
\end{proof}

\begin{cor}\label{cor:min_end_ell_01}
    Let $S$ be any overorder of $R$.
    There exists an abelian variety in the isogeny class whose endomorphism ring $T$ satisfies $S_\ell \simeq T_\ell$ for every $\ell \neq p$ and $S^{*}_p \simeq T^{*}_p$ for $*$ equal to $\{0\}$ or $\{1\}$.
\end{cor}
\begin{proof} 
    It follows from \cite[Theorem~5.1]{Wat69} that there is a $W'_R\{F',V'\}$-ideal $M'$ with maximal endomorphism ring. 
    The abelian variety $\Psi((SI,\Delta(i_p(S))M))$ constructed as in the proof of Proposition~\ref{prop:extension_are_end} then has an endomorphism ring with the wanted properties.
\end{proof}

\section{Computing fractional \texorpdfstring{$W_R\{F,V\}$-ideals}{WRFV2} }
\label{sec:aftergreenredblue} 
The goal of this section is to provide an effective method to compute representatives of the $\Delta$-isomorphism classes of $W_R\{F,V\}$-ideals.
We start by first observing that the content of Section~\ref{sec:tate} allows us to obtain the \'etale and multiplicative parts without extra effort.
Indeed, by the discussion after Notation~\ref{not:splitR}, the decomposition $R_p = R_p^{\{0\}} \times R_p'  \times R_p^{\{1\}}$ induces a decomposition of $\Delta$-isomorphism classes of $W\{F,V\}$-ideals, which yields a natural bijection 
\begin{equation}\label{eq:WRFVdec}
\begin{split}
\{ \Delta\text{-isom.~classes of }  & W_R\{F,V\}\text{-ideals} \} \longleftrightarrow \\
& \prod_{*\in \{ \{0\},(0,1),\{1\} \}} \{ \Delta^*\text{-isom.~classes of }W_R^*\{F^*,V^*\}\text{-ideals} \}.
\end{split}
\end{equation}
The next proposition states that for $* \in \{ \{0\},\{1\}\}$ the computation of the terms on the right hand side of~\eqref{eq:WRFVdec} reduces to computations of known objects discussed in Section~\ref{sec:tate}.

\begin{prop}\label{prop:Rp01}
    For $*\in \{\{0\},\{1\}\}$ we have a natural bijection
    \begin{equation}\label{eq:coridealsplit}    
     \{ \Delta^*\text{-isom.~classes of }W_R^*\{F^*,V^*\}\text{-ideals} \} \longleftrightarrow 
    \prod_{ \frp_\nu \in \mathcal{P}_{R_p}, s(\nu) \in * } \cW_{\frp_\nu}(R),
    \end{equation}
    where denotes $\frp_\nu$ the intersection with $R_p$ of the maximal ideal inducing $\nu$. 
\end{prop}
\begin{proof}
    By Corollary~\ref{cor-01}, the left-hand side of~\eqref{eq:coridealsplit} is in bijection with 
    $ \{ \Delta^*\text{-isom.~classes of }R_p^*\text{-ideals} \}$, which is just $\cW(R_p^*)$ by Remark~\ref{rmk:W_R_semilocal}.
    Since $R_p^*$ is the direct product of the rings $R_{\frp_\nu}$ for $\frp_\nu \in  \mathcal P_R$ such that $s(\nu) \in *$,
    the latter set is in bijection with the right hand side of~\eqref{eq:coridealsplit} by Proposition~\ref{prop:ell_to_prod_frl}.
\end{proof}

In view of Proposition \ref{prop:Rp01}, we restrict our attention to computing the $\Delta'$-isomorphism classes of $W_R'\{F',V'\}$-ideals. 
We can view $E'_p$ as a subset of $E_p$ and $A'$ as a subset of $A$, so that $\Delta\colon E_p \to A$ restricted to $E'_p$ is precisely $\Delta'$.
We say that two fractional $W'_R$-ideals $I$ and $J$ are $\Delta'$-isomorphic if there exists a non-zero divisor $\gamma \in R'_p$ such that $\Delta'(\gamma)I =J$.
Note that $\Delta'$-isomorphism implies isomorphism as $R'_p$-modules, but that conversely two fractional $W'_R$-ideals can be isomorphic as $R'_p$-modules while not being $\Delta'$-isomorphic. 
The computation of $\Delta'$-isomorphism classes is divided into four steps:
\begin{enumerate}
    \item Globalize the problem.
    \item Compute $W'_R$-isomorphism classes of fractional $W'_R$-ideals.
    \item Compute the partition of each $W'_R$-isomorphism class into $\Delta'$-isomorphism classes of fractional $W'_R$-ideals.
    \item Determine which $\Delta'$-isomorphism classes of fractional $W'_R$-ideals are stable under the action of $F$ and $V$, that is, which are $W'_R\{F',V'\}$-ideals.
\end{enumerate}
The four steps will be detailed in Subsections \ref{sec:global}, \ref{sec:Step1}, \ref{sec:Step2} and \ref{sec:Step3}, respectively.
While Step 1 and Step~2 are essentially reduction steps, Steps 3 and 4 are truly novel, to the best of our knowledge.
Details are given in the beginning of each subsection.

\subsection{Step 1: Global setting}
\label{sec:global}
In this preliminary first step we reduce the computation of the $\Delta'$-isomorphism classes into a problem which involves only global objects.
This process, which comes at the cost of heavier notation and some technical lemmas, will allow us to obtain exact algorithms that are not affected by the precision issues typical of $p$-adic computations.

Recall that $L$ is an unramified extension of degree $a$ of $\Q_p$.
Let $c(x)$ be a lift to $\Z[x]$ of any irreducible factor of degree $a$ of $x^{q-1}-1$ over $\F_p[x]$.
We have $L\simeq \Q_p[x]/c(x)$.
Let $\widetilde{L}=\Q[x]/c(x)$ and denote by $\zeta$ the class of $x$ in $\widetilde L$.
Note that $\widetilde{L} \otimes_\Q \Q_p \simeq L$.
Moreover $\widetilde L$ has a unique place above $p$, which is unramified.

Let $g_1(y),\ldots,g_r(y)$ be the irreducible factors over $\widetilde L[y]$ of the characteristic polynomial $h(y)$ of~$\pi$ over $\Q$.
Define 
\begin{equation}\label{eq:AWtilde}
\widetilde A = \prod_{i=1}^r \frac{\widetilde L[y]}{g_i(y)} \quad\text{and}\quad \widetilde W_R = \prod_{i=1}^r \frac{\OO_{\widetilde L}[y]}{g_i(y)}. 
\end{equation}
The association $\pi \mapsto (y+(g_i(y)))_{i=1, \ldots, r}$ induces an embedding $\widetilde\Delta \colon E \to \widetilde A$.

\begin{remark}\label{rmk:tilde}\ 
    \begin{enumerate}[(i)]
        \item We have isomorphisms $\widetilde A \otimes_\Q \Q_p \simeq A $ and $\widetilde W_R \otimes_\Z \Z_p \simeq W_R $ giving inclusions $\widetilde W_R \subseteq \widetilde A \subseteq A$.
        \item \label{rmk:tilde:ideals} For every fractional $W_R$-ideal $I \subseteq A$, there is a unique fractional $\widetilde W_R$-ideal $\widetilde I$ such that $\widetilde I \otimes_\Z \Z_p = I$ and $\widetilde I \otimes_\Z \Z_\ell = \OO_{\widetilde A} \otimes_\Z \Z_\ell$ for every prime $\ell\neq p$; see for example \cite[Theorem 9.4.9, Lemma 9.5.3]{JVQuat}, as in Proposition~\ref{prop:extension_are_end}.
        Also, if $I\subseteq J$ are fractional $W_R$-ideals then
        \[ \frac{\widetilde J}{\widetilde I} \simeq \frac{J}{I} \]
        is a finite $W_R$-module annihilated by a power of $p$.
        \item Let $\frp$ be a maximal ideal of $W_R$ and $\widetilde\frp$ be as in \ref{rmk:tilde:ideals}. Then $\widetilde W_{R,\widetilde \frP} \simeq W_{R,\frP}$.
        \item \label{rmk:tilde:sigma}
        Assume for a moment that $\widetilde L$ is a normal extension of $\Q$.
        Let $\widetilde \sigma_L$ be any generator of the decomposition group of the maximal ideal $p\OO_{\widetilde L}$ of $\OO_{\widetilde L}$.
        Then $\widetilde\sigma_L$ extends to the Frobenius $\sigma$ of $L$.
        Let $\pi_A$ denote the element $\widetilde \Delta(\pi)$.
        Since $h(x)$ is square-free, the elements $1_A, \pi_A, \ldots, \pi_A^{\deg(h)-1}$ form a $\widetilde L$-basis of $\widetilde A$.
        In other words, $\widetilde A$ can be described as
        \begin{equation}\label{eq:tildeA_powerbasis}
            \widetilde A = 1_A \cdot \widetilde L \oplus \pi_A \cdot \widetilde L \oplus \ldots \oplus \pi_A^{\deg(h)-1} \cdot \widetilde L.    
        \end{equation}
        Using the presentation given in Equation~\eqref{eq:tildeA_powerbasis}, define $\widetilde\sigma\colon \widetilde A \to \widetilde A$ by fixing $\pi_A$ and acting on the $\widetilde L$-coefficients as $\widetilde\sigma_L$.
        Observe that $\widetilde \sigma$ extends to the automorphism $\sigma$ of $A$.

        If $\widetilde L$ is not normal we cannot compute the decomposition group of $p\OO_{\widetilde L}$ or $\widetilde\sigma$.
        Nevertheless, we will only need the action of $\sigma$ on finite quotients of $\OO_{\widetilde A}$ (or of $\OO_{\widetilde A}^\times$). We will avoid precision errors by carefully approximating $\sigma$ on such quotients, see Remark~\ref{rmk:alg3_sigma_fin_quot}.
        \item \label{rmk:tilde:primes_above_nu}
        Let $\nu$ be a place of $E_p$.
        Then we have 
        \[ \OO_{A_\nu} = \underbrace{\OO_{LE_\nu} \times \ldots \times \OO_{LE_\nu}}_{g_\nu-\text{times}}, \]
        cf.~Equation~\eqref{eq:splittingA}. The unique maximal ideal of $\mathcal \OO_{E_\nu}$ is denoted by $\frp_{E_\nu}$. A maximal ideal of $\OO_{A_\nu}$ is equal to the unique maximal ideal $\frp_{LE_{\nu}}=\frp_{E_{\nu}} \OO_{LE_{\nu}}$ in exactly one of the $g_{\nu}$ factors, and to the ideal generated by $1$ in all other factors.
        In particular, there are exactly $g_\nu$ maximal ideals $\frP_{\nu,1},\ldots,\frP_{\nu,g_\nu}$ of $\OO_{A_\nu}$ extending $\nu$, all with the same ramification index. Let $t_1,\ldots,t_{g_\nu}$ be some corresponding uniformizers.
        If $t_\nu \in E_p$ is a uniformizer for $\nu$ then its image in $A_\nu$ via the embedding induced by $\Delta$ is $(t_\nu,\ldots,t_\nu)=(v_1t_1,\ldots,v_{g_\nu}t_{g_\nu})$ for units $v_i\in\OO_{LE_\nu}^\times$. 
        For each $i$, consider the maximal ideal $\widetilde{\frP}_{\nu,i}$ of~$\OO_{\widetilde{A}}$ built using~\ref{rmk:tilde:ideals}. 
        Any lift of a basis element of the $\OO_{\widetilde A}/\widetilde{\frP}_{\nu,i}$-vector space $\widetilde{\frP}_{\nu,i}/\widetilde{\frP}_{\nu,i}^2$ 
        will map to a uniformizer of $\OO_{\widetilde{A},\widetilde{\frP}_{\nu,i}}\simeq \OO_{A,\frP_{\nu,i}}$.
        Such lift can be taken to be a unit in the completions at the other maximal ideals above $\nu$.
        \item \label{rmk:tilde:prime} Let $S'$ be an order in $A'$, let $I'$ be a fractional $S'$-ideal, and let $\frp'$ be a maximal ideal of $S'$.
        Consider the order in $A=A^{\{0\}}\times A' \times A^{\{1\}}$ given by 
        $S=\OO_{A^{\{0\}}}\times S' \times \OO_{A^{\{1\}}}$
        and its maximal ideal 
        $\frp=\OO_{A^{\{0\}}}\times \frp' \times \OO_{A^{\{1\}}}$.
        Let $I= \OO_{A^{\{0\}}}\times I' \times \OO_{A^{\{1\}}}$.
        Define $\widetilde S$ as the order in $\widetilde A$ satisfying $\widetilde S \otimes_{\Z} \Z_p \simeq S$, and $\widetilde \frp$ as the prime of $\widetilde S$ such that  $\widetilde \frp \otimes_{\Z} \Z_p \simeq \frp$, and $\widetilde I$ as the fractional $\widetilde S$-ideal such that  $\widetilde I \otimes_{\Z} \Z_p \simeq I$, all using (ii).
        Then we have canonical isomorphisms 
        \[ \widetilde S_{\widetilde \frp} \simeq S_\frp \simeq S'_{\frp'} \quad\text{and}\quad \widetilde I_{\widetilde \frp} \simeq I_\frp \simeq I'_{\frp'}.\]
    \end{enumerate}
\end{remark}

\subsection{Step 2: \texorpdfstring{$W'_R$-isomorphism classes}{WpR} }
\label{sec:Step1}
In this step, we compute global representatives of the isomorphism classes of fractional $W'_R$-ideals.
Lemma~\ref{lem:W'R_prod_global} creates a bridge to the results contained in Section~\ref{sec:tate}, allowing us to use \cite[\texttt{ComputeW}]{Mar23_Loc} to complete the task.

\begin{lem}\label{lem:W'R_prod_global} Let $\mathcal S$ be the set of primes $\frP$ of $W'_R$.
    Then we have a natural bijection
    \[ \cW(W'_R) \longleftrightarrow \prod_{\frP \in \mathcal S} \cW_{\widetilde{\frP}} (\widetilde W_R).\]
\end{lem}
\begin{proof}
    Note that $W'_R$ is canonically isomorphic to the direct product of the completions $W_{R,\frP}$ where~$\frP$ ranges over the set $\mathcal{S}$.
    The statement follows from the canonical isomorphism $\widetilde W_{R,\widetilde \frP} \simeq W_{R,\frP}$, see Remark~\ref{rmk:tilde}.(iii), and from Lemma~\ref{rmk:WRlocalize}.\ref{rmk:WRlocalize:2}.
\end{proof}

Lemma~\ref{lem:W'R_prod_global}, when combined with Proposition \ref{prop:global_to_frl}, allow us to 
reduce the computation of the isomorphism classes of fractional $W'_R$-ideals to the computation of $\cW(\widetilde W_R + \widetilde \frp_i^k \OO_{\widetilde A})$, where $\widetilde \frp_i$ runs over the finite set of maximal ideals of $\widetilde W_R$ that lie below the maximal ideals of $\OO_{\widetilde A}$ extending the places of $E_p$ with slope in $(0,1)$.
For each $i$, a set of representatives of $\cW(\widetilde W_R + \widetilde \frp_i^k \OO_{\widetilde A})$ can be compute using for example \cite[\texttt{ComputeW}]{Mar23_Loc}.
One can glue these representatives using Lemma~\ref{lem:glue_local_parts}, so obtaining a set of fractional $\widetilde W_R$-ideals representing $\cW(W'_R)$.

\subsection{Step 3: \texorpdfstring{$\Delta'$-isomorphism classes}{Dp}}
\label{sec:Step2}
The output of the Step 2 consists of $W'_R$-linear isomorphism classes.
Our final goal of computing isomorphism classes of $W'_R\{F',V'\}$-ideals requires us to work modulo $\Delta'$-linear morphisms, that is, a finer equivalence relation.
This third step address this point.
Given a fractional $W'_R$-ideal $I$, we will denote by $[I]_{W'_R}$ its class in $\cW(W'_R)$ and by $[I]_{\Delta'}$ its $\Delta'$-isomorphism class.
Since every $\Delta'$-isomorphism is also a $W'_R$-linear isomorphism, we see that $[I]_{W'_R}$ splits into a disjoint union of $\Delta'$-isomorphism classes.

We start from the combinatorial classification of isomorphism classes of $W_R \{F,V\}$-ideals given by Waterhouse in \cite[Theorem~5.1]{Wat69}, where the author only considers ideals with maximal endomorphism ring.
This result is readily adapted to computing only the local-local part, see Proposition~\ref{prop:FVstable_OO}.
We denote by $e_I$ the extension map $[I']_{\Delta'} \mapsto  [I'\OO_{A'}]_{\Delta'}$ from the set of $\Delta'$-isomorphism classes of fractional $W'_R$-ideals that are $W'_R$-isomorphic to $I$, to the set of $\Delta'$-isomorphism classes of fractional $\OO_{A'}$-ideals.
In Proposition \ref{prop-fibers}, we describe each fiber of the extension map $e_I$ as an orbit space under the action of a quotient of the unit group $\OO_{A'}^\times$.
The main difficulty in this step is to compute global representatives of a certain finite quotient of the unit group $\OO_{A'}^\times$.
This requires us to carefully approximate the action of the Frobenius $\sigma$ of $L$ on such a quotient, see Lemma~\ref{lem:quot_stab_sigma} and Remark~\ref{rmk:units_global}.
Next, we observe in Proposition~\ref{prop:ext_of_FV_stable} that the extension of an $W'_R\{F',V'\}$-ideal is an $W'_R\{F',V'\}$-ideal.
Finally, Algorithm~\ref{alg:Deltaclasses_OO} wraps up the content of this subsection.
\begin{prop} \label{prop-fibers}
    Let $I$ be a fractional $W'_R$-ideal.
    Let $S=(I:I)$ and $J=I\OO_{A'}$.
    The association $\gamma \mapsto \gamma I$ for $\gamma \in \OO_{A'}^{\times}$ induces a free and transitive group action of $\OO_{A'}^\times/S^{\times}\Delta'(\mathcal O_{E'_p}^\times)$ on the fiber $e_I^{-1}([J]_{\Delta'})$.
\end{prop}
\begin{proof} 
    If $\gamma \in \OO^{\times}_{A'}$ then $\gamma I \OO_{A'}=J$.
    Now $\gamma I = I $ if and only if $\gamma\in S^{\times}$, and $\gamma I$ is $\Delta'$-isomorphic to $I$ if and only if $\gamma \in \Delta'({E'_p}^\times)$. 
    We claim that 
    \[\OO_{A'}^{\times} \cap S^{\times} \Delta'({E'_p}^\times) = S^{\times} \Delta'(\OO_{E'_p}^\times).\]
    Indeed, for any $y \in \OO_{A'}^{\times} \cap S^\times \Delta'({E'_p}^\times)$ there exist $s \in S^\times$ and $x \in \Delta'({E'_p}^\times)$ such that $y=sx \in \OO_{A'}^{\times}$.
    
    Recall that $A' = \prod_{\nu} A_\nu$ where $\nu$ runs over the places of $E'_p$ and that $A_\nu$ is a direct product of $g_\nu$ copies of $LE_\nu$. We denote here by $\val_\nu$ both the valuation on $E_\nu$ and its (unramified) extension to $LE_\nu$.
    Then $0=\val_{\nu}(s_{\nu,i}x_{\nu,i})=\val_{\nu}(x_{\nu,i})$ for each $\nu$ and each $0\leq i \leq g_\nu$, and hence $x_{\nu,i} \in \OO_{E'_\nu}^\times$. 
    Note that $x_{\nu,i}=x_{\nu,j}$ for all $i,j$ by assumption, so $x \in \Delta'(\OO_{E'_p}^\times)$.   
    This shows that the group action is well-defined and free. 

    Pick $[I_0]_{\Delta'}$ in $e_I^{-1}([J]_{\Delta'})$.
    By assumption there are $\delta \in A'^{\times}$ and $\beta\in\Delta({E'_p}^{\times})$ such that
    $I_0=\delta I$ and $I_0\OO_{A'} = \beta J$.
    Put $\gamma = \delta/\beta$.
    By construction, we have $\gamma I = \beta^{-1} I_0$ which shows that $[\gamma I]_{\Delta'} = [I_0]_{\Delta'}$.
    So, to conclude that the group action is also transitive, it suffices to show that $\gamma \in \OO_{A'}^\times$, as we now do.
    Recall that by hypothesis we have $J=I\OO_{A'}$. 
    Hence
    \[ \gamma I \OO_{A'} = \frac{\delta}{\beta}I\OO_{A'} = \frac{1}{\beta}I_0\OO_{A'} = J = I\OO_{A'}, \]
    which implies that $\gamma \in \OO_{A'}^{\times}$.
\end{proof}

In Algorithm~\ref{alg:Deltaclasses_OO} below, we will use that, given an overorder $S$ of $W'_R$, the quotient $\OO_{A'}^\times / S^{\times}\Delta(\OO_{E'_p}^\times)$ is the cokernel of the inclusion
\[ \frac{S^{\times}\Delta'(\OO_{E'_p}^\times)}{S^\times} \longrightarrow \frac{\OO_{A'}^\times}{S^\times},\]
as we now show.
\begin{lem}\label{lem:quot_stab_sigma}
    The subgroup $H=W_R^{\times}\Delta'(\OO_{E'_p}^\times)/W_R^\times$ is generated by the elements of $ \OO_{A'}^\times /W_R^\times $ fixed by the action induced by the automorphisms $\sigma \colon A_\nu\to A_\nu$ for $\nu$ ranging over the places of $E_p$ with $s(\nu)\in(0,1)$.
    
    Let $S$ be any overorder of $W_R$.
    The subgroup $S^{\times}\Delta'(\OO_{E'_p}^\times)/S^\times$ is generated by the image of $H$ via the natural projection $ \OO_{A'}^\times /W_R^\times \twoheadrightarrow \OO_{A'}^\times /S^\times $.
\end{lem}
\begin{proof}
    Recall from Equation~\eqref{eq:def_sigma} that $\sigma$ acts on $A_\nu = \oplus LE_{\nu}$ as a cyclic permutation on the $g_\nu$ copies of $LE_\nu$ followed by~$\tau_\nu$ on the last component.
   Therefore, the elements of $\OO_{A_\nu}^\times$ fixed by $\sigma$ are those whose components are fixed by $\tau_\nu$ and are all the same, hence the elements are in the image of $\Delta'$. 
   These form precisely $\Delta'(\OO_{E'_\nu}^\times)$.
    The first statement then follows since $\sigma$ acts also on $W_{R,\nu}$.
    The second statement is a consequence of the fact that we have $S^\times W_R^{\times}\Delta'(\OO_{E'_p}^\times) = S^{\times}\Delta'(\OO_{E'_p}^\times)$.
\end{proof}

\begin{remark}\label{rmk:units_global}
    In Step~\ref{alg:unitsquotient} of Algorithm~\ref{alg:Deltaclasses_OO}, we need to compute a set of representatives in $\OO_{\widetilde A}$ of $\OO_{A'}^\times / S^{\times}\Delta(\OO_{E'_p}^\times)$.
    Consider the finite rings
    \[\mathfrak{o}_A=\frac{\OO_{\widetilde A}}{\widetilde \frf + \prod_{i=1}^r\widetilde \frp_i^{k_i}\OO_{\widetilde A}} 
    \qquad\text{and}\qquad
    \mathfrak{s} = \dfrac{\widetilde S}{\widetilde \frf+\prod_{i=1}^r\widetilde \frp_i^{k_i}},\]
    where:
    $\widetilde{\frf}$ is the conductor of $\widetilde S$ in $\OO_{\widetilde A}$; 
    $\widetilde \frp_1,\ldots,\widetilde \frp_r$ are the maximal ideals of $\widetilde S$ corresponding to the finitely many maximal ideals of $S$;
    for $i=1\ldots,r$, set $k_i$ to be the smallest integer greater than or equal to $\val_p(|\widetilde S/\widetilde \frf|)/\val_p(|\widetilde S/\widetilde \frp_i|)$.
    Then we have a canonical isomorphism:
    \[ \frac{\OO_{A'}^\times}{S^\times} \simeq 
    \frac{\mathfrak{o}_A^\times}
    {
    \mathfrak{s}^\times}. \]
    Algorithms to concretely compute this quotient can be found in \cite{HessPauliPohst03} and \cite{klupau05}.
    Then, we reach our goal by using Lemma~\ref{lem:quot_stab_sigma} together with the fact that the action of $\sigma$ can be realized on the finite ring $\mathfrak{o}_A$, as we now explain.
    Let $\mathfrak{o}_L = \OO_{\widetilde L}/(p\OO_{\widetilde L})^m$, where $m$ is defined by $\vert\mathfrak{o}_A\vert=\vert \OO_{\widetilde L}/(p\OO_{\widetilde L})\vert^m$.
    Then, the finite ring~$\mathfrak{o}_A$ is an $\mathfrak{o}_L$-algebra.
    If $b_1,\ldots,b_{2g}$ is the image of a $\Z$-basis of $\OO_E$ in $\mathfrak{o}_A$ then we can give a $\sigma$-equivariant presentation of $\mathfrak{o}_A$ as an $\mathfrak{o}_L$-algebra by 
    \begin{align*}
        \mathfrak{o}_L\times \ldots \times \mathfrak{o}_L &\longrightarrow \mathfrak{o}_A \\
        (c_1,\ldots,c_{2g}) &\longmapsto \sum_{i=1}^{2g} c_ib_i.
    \end{align*}
    Hence, in order to compute the action of $\widetilde \sigma$ on $\mathfrak{o}_A$, it suffices to compute an approximation of the ($p$-adic) Frobenius automorphism of $L$ on $\mathfrak{o}_L$.
    This can be done as follows.
    Let $u$ be a lift in $\mathfrak{o}_L$ of a generator of $\left( \OO_{\widetilde L}/p \OO_{\widetilde L}\right)^\times$.
    Compute $u^{q^i}$ for $i>0$ until $u^{q^i} = u^{q^{i+1}}$ and let $z=u^{q^i}$.
    Then $z$ is the image of an inertial element of $\OO_L\simeq \Z_p[z]$ in $\mathfrak{o}_L$.
    Hence, by computing an explicit isomorphism $\Z_p[z]/(p\Z[z])^m \simeq \mathfrak{o}_L$, we can compute the action of $\sigma$ on $\mathfrak{o}_L$ by pushing forward the action of $\sigma $ on $\Z_p[z]$ which is given by $z\mapsto z^p$.
    
    Note that this construction does not require $\widetilde L$ to be normal; cf~Remark~\ref{rmk:tilde}.\ref{rmk:tilde:sigma}.
    We stress that the output of Step \eqref{alg:unitsquotient} in Algorithm~\ref{alg:Deltaclasses_OO} is independent of choice of the approximation of $\widetilde \sigma$ we computed, which is certainly not unique.  
\end{remark}

For the remainder of the article, for each isogeny class, we fix $F'$ of $W$-type (cf.~Definition~\ref{def:F-Waterhouse}).
Thanks to Proposition~\ref{prop:FVstable_OO}, we understand the $W'_R\{F',V'\}$-ideals with maximal endomorphism ring.
We now show in Proposition~\ref{prop:ext_of_FV_stable} that $e_I$ sends $W'_R\{F',V'\}$-ideals to $W'_R\{F',V'\}$-ideals.
This implies that, in order to compute all $W'_R\{F',V'\}$-ideals, it suffices to consider the fibers of $W'_R\{F',V'\}$-ideals with multiplicator ring $\OO_A'$.
\begin{prop}\label{prop:ext_of_FV_stable}
    Let $I$ be a fractional $W'_R$-ideal.
    If $I$ is a $W'_R\{F',V'\}$-ideal, then $I\OO_{A'}$ is also a $W'_R\{F',V'\}$-ideal. 
\end{prop}
\begin{proof} 
    Any element of $I\OO_{A'}$ can be written as a finite sum $z=\sum_i x_i a_i$ with $x_i \in I$ and $a_i \in \OO_{A'}$.
    Recall that $F'$ is of the form $z \mapsto \alpha z^\sigma$ on each component $A_\nu$.
    The action of $\sigma$ is multiplicative, so we can write
    \[ F'(z) = \sum_i F'(x_ia_i) = \sum_i F'(x_i)\sigma(a_i). \]
    Since $F'(x_i) \in I$ by assumption and $\sigma(a_i) \in \OO_{A'} ( \simeq W \otimes \OO_{E'_p})$ we see that $F'(z) \in I\OO_{A'}$. 
    Also, 
    \[ V'(z) = \sum_i V'(x_ia_i) = \sum_i V'(x_i)\sigma^{-1}(a_i). \]
    Now, $V'(x_i) \in I$ by assumption and $\sigma^{-1}(a_i) \in \OO_{A'}$.
    So $V'(z) \in I\OO_{A'}$.
    This shows that $I\OO_{A'}$ is a $W'_R\{F',V'\}$-ideal.
\end{proof}

We are now ready to combine all the results presented in this section into an algorithm to compute $\Delta'$-isomorphism classes.

\begin{alg}
\label{alg:Deltaclasses_OO}\hfill\\
    \textbf{Input}: A set of fractional $\widetilde W_R$-ideals $\widetilde I_1,\ldots,\widetilde I_n$ representing the $W'_R$-isomorphism classes of all fractional $W'_R$-ideals.\\
    \textbf{Output}: A set of fractional $\widetilde W_R$-ideals representing the $\Delta'$-isomorphism classes of all fractional $W'_R$-ideals whose extension to $\OO_{A'}$ is in $\Upsilon$ (see Definition~\ref{def:Upsilon}).\\
    \begin{enumerate}[label=({\arabic*}), ref={\arabic*},wide=0pt, leftmargin=\parindent]\vspace{-1em}
        \item 
        For each place $\nu$ of $E_p$ of slope in $(0,1)$, compute the maximal ideals $\widetilde{\frP}_{\nu,1},\ldots,\widetilde{\frP}_{\nu,g_\nu}$ of $\OO_{\widetilde A}$ extending~$\nu$; see Remark~\ref{rmk:tilde}.\ref{rmk:tilde:primes_above_nu}.
        For each index $k$, compute an element $t_{\nu,k}$ which is a uniformizer at $\mathfrak{P}_{\nu,k}$ and a unit at every other maximal ideal.
        \item Use Proposition~\ref{prop:FVstable_OO} to compute, up to $\Delta'$-isomorphism, all $W'\{F',V'\}$-ideals $J_1,\ldots,J_m$ having maximal multiplicator ring $\OO_{A'}$.
        Each $J_j$ is stored as a tuple of vectors of the form 
        $({\varepsilon_{j,\nu,1}},\ldots, {\varepsilon_{j,\nu,g_\nu}})_\nu$ (as in $\Upsilon$). 
        \item \label{alg:Deltaclasses_OO:foreachIi} For each $\widetilde I_i$, do the following (to ease the notation we partially suppress the dependency on $i$):
            \begin{enumerate}[label=({\ref{alg:Deltaclasses_OO:foreachIi}.\alph*}), ref={\ref{alg:Deltaclasses_OO:foreachIi}.\alph*}]
                \item Compute the extension $\widetilde J=\widetilde I_i\OO_{\widetilde A}$.
                \item Compute the factorization 
                \begin{equation} \label{eq:jdagger} \widetilde J = \widetilde{J}^\dagger\times\prod_{\nu}\prod_{k=1}^{g_\nu}\widetilde\frP^{\eta_{\nu,k}}_{\nu,k},  \end{equation}
                where $\nu$ runs over the places of $E_p$ of slope in $(0,1)$, and $\widetilde{J}^\dagger$ is the product of all the other maximal ideals.
                \item For each $j=1,\ldots,m$, let $\delta_j=\prod_{\nu}\prod_{k=1}^{g_\nu}t_{\nu,k}^{\eta_{\nu,k}-\varepsilon_{j,\nu,k}}$.
                \item  \label{alg:unitsquotient} Use Remark~\ref{rmk:units_global} to compute $\gamma_1,\ldots,\gamma_r$ in $\OO_{\widetilde A}$ representing the elements of $\OO_{A'}^\times / S^{\times}\Delta(\OO_{E'_p}^\times)$, where $S$ is the order in $A'$ generated by the multiplicator ring of $\widetilde I_i$.
                \item 
                Define the set of fractional $\widetilde W_R$-ideals $\mathcal{I}_i=\{ \delta_j^{-1}\gamma_l \widetilde I_i\ : 1\leq j \leq m , 1\leq l \leq r \}$.
            \end{enumerate}
        \item Return $\bigcup_{i=1}^n\mathcal{I}_i$.
    \end{enumerate}
\end{alg}
\begin{thm}
    Algorithm~\ref{alg:Deltaclasses_OO} is correct.
\end{thm}
\begin{proof}
    Let $L,L'\in \bigcup_{i=1}^n\mathcal{I}_i$ with $L\neq L'$.
    If $L\in \mathcal{I}_i$ and $L'\in\mathcal{I}_{i'}$ for indices $1\leq i < i' \leq n$ then $L$ and $L'$  are not $\Delta'$-isomorphic since they are not $W'_R$-isomorphic.
    If $L$ and $L'$ belong to the same $\mathcal{I}_i$ then they are not $\Delta'$-isomorphic by Proposition~\ref{prop-fibers}.
    Hence, all fractional $W'_R$-ideals in $\bigcup_{i=1}^n\mathcal{I}_i$ are pairwise non-$\Delta'$-isomorphic.
    
    To conclude we need to show that $\bigcup_{i=1}^n\mathcal{I}_i$ contains representatives of all $\Delta'$-isomorphism class of $W'_R$-ideals whose extension to $\OO_A'$ is a $W'_R\{F',V'\}$-ideal.
   Fix an $i$, and write $\widetilde{J}=\widetilde{J}^{\dagger} \times \widetilde{J}^{\dagger\dagger}$ as in Equation~\eqref{eq:jdagger}. Note that $\widetilde{J}^{\dagger\dagger}$ is isomorphic to $\widetilde{I}_i\OO_{\widetilde A}$ locally at every maximal ideal $\frP_{\nu,k}$.
    Due to the way the $t_{\nu,k}$ are constructed, we get that 
    $J_j$ and $\delta_j^{-1}\widetilde{J}^{\dagger\dagger}_i$ are equal locally at every maximal ideal $\frP_{\nu,k}$.
    Hence, the result follows from Propositions~\ref{prop:FVstable_OO}, \ref{prop-fibers} and \ref{prop:ext_of_FV_stable}.
\end{proof}

\subsection{Step 4: Stability under the action of \texorpdfstring{$F$ and $V$}{FV}}
\label{sec:Step3}
Algorithm~\ref{alg:Deltaclasses_OO} returns a list of fractional $\widetilde W_R$-ideals representing the $\Delta'$-isomorphism classes of fractional $W'_R$-ideals.
Consider the class $[I]_{\Delta'}$ represented by $\widetilde I$, which we can assume to be contained in $\OO_{\widetilde A}$.
This step determines whether $[I]_{\Delta'}$ consists of $W'_R\{F',V'\}$-ideals.
The main difficulty is that, in general, $F'$ and $V'$ are intrinsically $p$-adic, that is, cannot be realized on the $\Q$-algebra $\widetilde A$.
But Lemma~\ref{lem:step3_eq_cond} shows that we can instead check if $[I]_{\Delta'}$ consists of $W'_R\{F',V'\}$-ideals on a finite quotient $Q_{m_0}$, where $m_0$ is the precision parameter minimally chosen to guarantee the correctness of the output, as explained in Remark~\ref{rmk:alg3_sigma_fin_quot}.
Using Lemma~\ref{lem:Frob_fin_quot}, we construct presentations $F_{m_0}$ of $F'$ on $Q_{m_0}$.
In fact, we do it at a precision $m_0+1$, since $V_{m_0}$ cannot be immediately recovered from $F_{m_0}$.
This whole procedure is Algorithm~\ref{alg:FVclasses}, which can be considered as the core of the computational contribution of this paper.
For further analysis, it is desirable to record $F_{m_0}$ and~$V_{m_0}$ together with the output of the algorithm, see for example Remark~\ref{rmk:a-numbers}, which shows how to recover the $a$-numbers of the abelian varieties.

The algorithm determines whether a given $W'_R$-ideal is a $W'_R\{F',V'\}$-ideal by pushing it into a finite quotient $Q_{m_0}$ which depends on a precision parameter $m_0$. 
This parameter is chosen minimally so that we can realize the actions $F_{m_0}$ (resp.~$V_{m_0}$) of $F'$ (resp.~$V'$) on $Q_{m_0}$.
\begin{lem}\label{lem:step3_eq_cond}
    Fix a $W'_R\{F',V'\}$-ideal $J$ and a fractional $W'_R$-ideal $I$ satisfying $I \subseteq J$.
    Let $\widetilde J$, $\widetilde I$, $\widetilde{V'I}$ and $\widetilde{F'I}$ be defined as in Remark~\ref{rmk:tilde}.\ref{rmk:tilde:prime}.
    Let $N$ be the exponent of the finite quotient $\widetilde J/\widetilde I$.
    Let~$m$ be an integer such that  $m \geq \val_p(N)$.
    Denote by~$\frp_1,\ldots,\frp_n$ the maximal ideals of $W_R$ which lie below the maximal ideals of $\OO_A$ above the places of $E$ of slope in $(0,1)$.
    For each $i=1,\ldots,n$, let $m_i$ be a positive integer such that $\vert J/p^m J\vert \leq \vert W_R/\frp_i \vert^{m_i}$.
    Then the following statements are equivalent:
    \begin{enumerate}[(i)]
        \item \label{lem:step3_eq_cond:class} $[I]_{\Delta'}$ consists of $W'_R\{F',V'\}$-ideals;
        \item \label{lem:step3_eq_cond:prime} $I = I + F'I + V'I$;
        \item \label{lem:step3_eq_cond:tilde} $\widetilde I = \widetilde I + \widetilde{F'I} + \widetilde{V'I}$;
        \item \label{lem:step3_eq_cond:finite} The images of $\widetilde I$ and $\widetilde I + \widetilde{F'I}+\widetilde{V'I}$ in the finite quotient $\widetilde J/p^m\widetilde J$ are equal;
        \item \label{lem:step3_eq_cond:01} The images of $\widetilde I$ and $\widetilde I + \widetilde{F'I}+\widetilde{V'I}$ in the finite quotient $\widetilde J/\left(p^m\widetilde J+\prod_{i=1}^n\frp_i^{m_i}\widetilde J\right)$ are equal.
    \end{enumerate}
\end{lem}
\begin{proof}
    The equivalence of~\ref{lem:step3_eq_cond:class} and~\ref{lem:step3_eq_cond:prime} follows from Corollary~\ref{cor:FV_on_classes}.
    By Remark~\ref{rmk:tilde}.\ref{rmk:tilde:ideals}, we have a canonical isomorphism of finite quotients
    \[ \frac{\widetilde I + \widetilde{F'I} + \widetilde{V'I}}{\widetilde I}
    \simeq \frac{I + F'I + V'I}{I},\]
    readily giving the equivalence of~\ref{lem:step3_eq_cond:tilde} and \ref{lem:step3_eq_cond:prime}.
    The assumption on $m$ implies that we have inclusions
    \[ p^m  \widetilde J \subseteq \widetilde I \subseteq \widetilde I + \widetilde{F'I} + \widetilde{V'I} \subseteq \widetilde J. \]
    Hence, conditions \ref{lem:step3_eq_cond:tilde} and \ref{lem:step3_eq_cond:finite} are equivalent.
    Let $Q=\widetilde J/p^m \widetilde J$.
    The splitting of $W_R \simeq W_R^{\{0\}} \times W'_R \times W_R^{\{1\}}$ induces a decomposition $Q \simeq Q^{\{0\}} \oplus Q' \oplus Q^{\{1\}}$.
    Observe that $Q' \simeq \widetilde J/\left(p^m\widetilde J+\prod_{i=1}^n\frp_i^{m_i}\widetilde J\right)$.
    Hence \ref{lem:step3_eq_cond:finite} implies \ref{lem:step3_eq_cond:01}.
    The converse follows from Remark~\ref{rmk:tilde}.\ref{rmk:tilde:prime}, which states that the ${(0)}$-parts (resp.~${(1)}$-parts) of $\widetilde I$ and $\widetilde I +\widetilde{F'I} + \widetilde{V'I}$ coincide.
\end{proof}

\begin{lem}\label{lem:Frob_fin_quot} 
    Fix an integer $j \geq 1$ and a place $\nu$ of $E_p$. Let $t_{\nu}$ be a uniformizer of $E_\nu$.   
    Let $\gamma_0$ be an element of $\OO^\times_{LE_\nu}$ such that
    \[ 
        N_{LE_\nu/E_\nu}(\gamma_0) - \pi_\nu/t_{\nu}^{\val_{\nu}(\pi_\nu)} \in \frp_{E_{\nu}}^j. 
    \]
    Then there exists $\gamma_1 \in \frp_{LE_{\nu}}^j $ such that
    \[ N_{LE_\nu/E_\nu}(\gamma_0+\gamma_1) = \pi_\nu/t_{\nu}^{\val_{\nu}(\pi_\nu)} \]
\end{lem}
\begin{proof}
    For ease of notation, let $N=N_{LE_\nu/E_\nu}$. Since $\pi_\nu/t_{\nu}^{\val_{\nu}(\pi_\nu)}$ is a unit in $\OO_{E_\nu}$ we can write $\pi_\nu/t_{\nu}^{\val_\nu(\pi_\nu)}=\zeta(1+x)$ for some root of unity $\zeta$ and some $x \in \frp_{E_{\nu}}$. Then there is an $\epsilon_0 \in \frp_{E_{\nu}}^j$ such that $N(\gamma_0)=\zeta(1+x+\epsilon_0)$. 
    By \cite[Proposition 3, p.~82]{SerreLocalFields} the norm $N$ surjectively maps $(1+\frp_{LE_\nu}^j)$ onto $(1+\frp_{E_\nu}^j)$.
    Hence, there is an element $\delta_0 \in \frp_{LE_{\nu}}^j$ such that $N(1+\delta_0)=1+(-\epsilon_0)$. Then $N(\gamma_0(1+\delta_0))=\zeta(1+x+\epsilon_1)$ with $\epsilon_1 \in \frp_{E_{\nu}}^{j+1}$ and we can find $\delta_1 \in \frp_{LE_{\nu}}^{j+1}$ such that $N(1+\delta_1)=1+(-\epsilon_1)$. Then $N(\gamma_0(1+\delta_0)(1+\delta_1))=\zeta(1+x+\epsilon_2)$ with $\epsilon_2 \in \frp_{E_{\nu}}^{j+2}$. Continuing this process (since $\OO_{E_\nu}^\times$ and  $\OO_{LE_\nu}^\times$ are complete with respect to the topology induced by $1+\frp_{E_\nu}^n$ respectively $1+\frp_{LE_\nu}^n$), we find $\delta \in \frp_{LE_{\nu}}^{j}$ such that $N(\gamma_0(1+\delta))=\zeta(1+x)$. 
    Setting $\gamma_1=\gamma_0\delta$ concludes the proof.
\end{proof}

We are now ready to give the algorithm 
that computes the isomorphism classes of $W'_R\{F',V'\}$-ideals, for $F'$ of $W$-type.
\begin{alg}
\label{alg:FVclasses}\hfill\\
    \textbf{Input}: 
    A set $\mathcal{I}=\{ \widetilde{I}_1,\ldots,\widetilde{I}_r \}$ of fractional $\widetilde W_R$-ideals representing the $\Delta'$-isomorphism classes of all fractional $W'_R$-ideals whose extension to $\OO_{A'}$ is in $\Upsilon$ (see Definition~\ref{def:Upsilon}).\\
    \textbf{Output}: 
    A subset of $\mathcal{I}$, consisting of representatives of the $\Delta'$-isomorphism classes of $W'_R\{F',V'\}$-ideals.\\
    \begin{enumerate}[label=({\arabic*}), ref={\arabic*},wide=0pt, leftmargin=\parindent]\vspace{-1em}
        \item Pick a $W'_R\{F',V'\}$-ideal $J$ with maximal multiplicator ring $\OO_{A'}$ using Proposition~\ref{prop:FVstable_OO}, and compute~$\widetilde J$.
        \item \label{alg:FVclasses:scaleJ} If necessary, scale $\widetilde{J}$ by multiplying by an element in 
        $\widetilde\Delta(E)$ so that $\widetilde J \subseteq \OO_{\widetilde A}$.
        \item \label{alg:FVclasses:scale_allI} If necessary, scale each $\widetilde I_k \in \mathcal{I}$ by multiplying by an element in 
        $\widetilde\Delta(E)$ so that $\widetilde I_k \subseteq \widetilde J$.
        \item \label{alg:FVclasses:m0}
        Let $m_0 = \max_k\{ \val_p(N_k) \}$, where $N_k$ is the exponent of the finite quotient $\widetilde J/\widetilde I_k$.
        \item \label{alg:FVclasses:foreachnu} For each $\nu$ with slope $s(\nu)\in(0,1)$ do the following (to ease the notation we partially suppress the dependency on $\nu$):
            \begin{enumerate}[label=({\ref{alg:FVclasses:foreachnu}.\alph*}), ref={\ref{alg:FVclasses:foreachnu}.\alph*}]
                \item Let $u$ be an element of $E$ representing a uniformizer for $\nu$, and let $g=g_\nu$ and $\val = \val_\nu$.
                \item Compute $w = \pi/u^{\val(\pi)}$.
                \item Compute the maximal ideals $\widetilde\frP_1,\ldots,\widetilde\frP_g$ of $\OO_{\widetilde A}$ above $\nu$, all with the same ramification index $e_\nu$, see Remark~\ref{rmk:tilde}.\ref{rmk:tilde:primes_above_nu}.
                \item \label{alg:FVclasses:Q_nu} Let $Q=\prod_{i=1}^{g}\OO_{ \widetilde A}/\widetilde \frP_i^{e_\nu(m_0+1)}$.
                \item Compute the multiplicative subgroup $U = \prod_{i=1}^{g}\left( \OO_{\widetilde A}/\widetilde \frP_i^{e_\nu(m_0+1)}\right)^\times$ of $Q$, cf.~\cite{HessPauliPohst03}.
                \item \label{alg:FVclasses:sigma_Q} Compute $\sigma_Q$, the automorphism of $Q$ induced by $\widetilde\sigma\colon \widetilde A\to \widetilde A$, see Remark~\ref{rmk:alg3_sigma_fin_quot}.
                \item \label{alg:FVclasses:hom_onU} Compute the group homomorphism
                \[\begin{tikzcd}[row sep = 0.2em]
                    \varphi: \left( \OO_{\widetilde A}/\widetilde \frP_g^{e_\nu(m_0+1)}\right)^\times \arrow[r,rightarrow] &  U \arrow[r,rightarrow]  & U \\
                    \gamma \arrow[r,mapsto] & (1,\ldots,1,\gamma) & \\
                    & \beta \arrow[r,mapsto]  & \beta \beta^{\sigma_{Q}} \cdots \beta^{\sigma_{Q}^{a-1}},
                \end{tikzcd}\]
                where, as usual, $a=\log_p(q)$. 
                \item Let $w_{U}$ be the image of $\widetilde\Delta(w)$ in $U$.
                \item Let $\gamma_0$ be any preimage of $w_{U}$ via $\varphi$ 
                \item Let $u_0$ be the image of $\tilde \Delta(u^{\val(\pi)g/a})$ in $ \OO_{\widetilde A}/\widetilde \frP_g^{e_\nu(m_0+1)}$.
                \item \label{alg:FVclasses:alpha_Q} Let $\alpha_{Q_\nu} = (1,\ldots,1,\gamma_0)\cdot(1,\ldots,u_0) \in Q$.
            \end{enumerate}
        \item \label{alg:FVclasses:alpha} Use the Chinese Remainder Theorem to compute an element $\alpha'$ in $\widetilde A$ which maps to $\alpha_{Q,\nu}$ (defined in step~\eqref{alg:FVclasses:alpha_Q}) for each place $\nu$ of $E$ of slope in $(0,1)$.
        \item Compute the maximal ideals $\widetilde{\frp}_1,\ldots,\widetilde{\frp}_n$ of $\widetilde W_R$ which lie below the maximal ideals of $\OO_{\widetilde A}$ above the places $\nu$ of $E$ of slope in $(0,1)$.
        \item For each $i=1,\ldots,n$, compute a positive integer $m_i$ such that $\vert \widetilde J/p^{m_0+1} \widetilde J\vert \leq \vert \widetilde W_R/\widetilde{\frp}_i \vert^{m_i}$.
        \item \label{alg:FVclasses:Qm0} Let $Q_{m_0+1}=\widetilde J/\left(p^{m_0+1}\widetilde J+\prod_{i=1}^n\widetilde{\frp}_i^{m_i}\widetilde J\right)$ and $Q_{m_0}=\widetilde J/\left(p^{m_0}\widetilde J+\prod_{i=1}^n\widetilde{\frp}_i^{m_i}\widetilde J\right)$
        and compute the natural projection $\mathrm{pr}\colon Q_{m_0+1} \to Q_{m_0}$.
        \item Let $\alpha_{m_0+1}$ (resp.~$\alpha_{m_0}$) denote the multiplication-by-$\alpha'$ map on $Q_{m_0+1}$ (resp.~$Q_{m_0}$).
        \item \label{alg:FVclasses:sigma_red}
        Compute compatible approximations $\sigma_{m_0+1}$ and $\sigma_{m_0}$ of $\sigma$ on $Q_{m_0+1}$ and $Q_{m_0}$, respectively, keeping the compatibility with the previously computed $\sigma_Q$ (see Remark~\ref{rmk:alg3_sigma_fin_quot}).
        \item \label{alg:FVclasses:Fm01} Define $F_{m_0+1}\colon Q_{m_0+1} \to Q_{m_0+1}$ as $x \mapsto \alpha_{m_0+1}(x^{\sigma_{m_0+1}})$, and $F_{m_0}\colon Q_{m_0} \to Q_{m_0}$ as $x \mapsto \alpha_{m_0}(x^{\sigma_{m_0}})$.
        \item Compute the homomorphism $m_p\colon Q_{m_0+1} \to Q_{m_0+1}$ induced by the multiplication-by-$p$ map.
        \item \label{alg:FVclasses:foreachgen} For each generator as a finite group $\gamma$ of $Q_{m_0}$:
            \begin{enumerate}[label=({\ref{alg:FVclasses:foreachgen}.\alph*}), ref={\ref{alg:FVclasses:foreachgen}.\alph*}]
                \item Pick $x_\gamma \in Q_{m_0+1}$ such that $\mathrm{pr}(x_\gamma) = \gamma$.
                \item Pick $z_\gamma \in Q_{m_0+1}$ such that $F_{m_0+1}(z_\gamma) = m_p(x_\gamma)$.
            \end{enumerate}
        \item \label{alg:FVclasses:V} Compute $V_{m_0}\colon Q_{m_0} \to Q_{m_0}$ by setting $V_{m_0}(\gamma) = \mathrm{pr}(z_\gamma)$ for each $\gamma \in Q_{m_0}$.
        \item For each $\widetilde I_k \in \mathcal{I}$ compute its image $I_{k,m_0}$ in $Q_{m_0}$ under the projection map from $\widetilde{J}$.
        \item Return $\{ \widetilde{I}_k \in \mathcal{I} \text{ such that } I_{k,m_0} = I_{k,m_0} + F_{m_0}(I_{k,m_0}) + V_{m_0}(I_{k,m_0}) \}$. 
    \end{enumerate}
\end{alg}

Before proving that Algorithm~\ref{alg:FVclasses} is correct 
we highlight in Remark~\ref{rmk:alg3_sigma_fin_quot} some important subtleties in how we choose the precision and in Remark~\ref{rmk:a-numbers} how this choice allows us to compute the $a$-numbers of the abelian varieties.

\begin{remark}\label{rmk:alg3_sigma_fin_quot}
    In Algorithm~\ref{alg:FVclasses}, we need to compute the action of $\sigma_{Q}$ on $Q$ in Step~\eqref{alg:FVclasses:sigma_Q} (where $Q$ is the quotient depending on a place $\nu$ of $E$ defined in Step~\eqref{alg:FVclasses:Q_nu}), and of $\sigma_{m_0+1}$ and $\sigma_{m_0}$ on $Q_{m_0+1}$ and $Q_{m_0}$ in Step~\eqref{alg:FVclasses:sigma_red}, all induced by the action of $\sigma$.
    In Step~\eqref{alg:FVclasses:sigma_Q}, for each $\nu$, the quotient $Q$ is a factor of of the finite ring $\OO_{\widetilde A}/p^{m_0+1}\OO_{\widetilde A}$.
    As in Remark~\ref{rmk:units_global}, we compute a $\sigma$-equivariant presentation of 
    $\OO_{\widetilde A}/p^{m_0+1}\OO_{\widetilde A}$ as an $\OO_{\widetilde L}/(p\OO_{\widetilde L})^{m}$-module for $m$ defined by $\vert \OO_{\widetilde A}/p^{m_0+1}\OO_{\widetilde A} \vert = \vert \OO_{\widetilde L}/p\OO_{\widetilde L} \vert^{m}$.  
    In Step~\eqref{alg:FVclasses:sigma_red}, we proceed as follows.
    By Step~\ref{alg:FVclasses:scaleJ}, we have $\widetilde J \subseteq \OO_{\widetilde A}$. 
    Let $m' = m_0+1 + \val_p([\OO_{\widetilde A}:\widetilde{J}])$, so that
    \[ p^{m'}\OO_{\widetilde A} \subseteq p^{m_0+1}\widetilde{J} \subseteq \left(p^{m_0+1}\widetilde J+\prod_{i=1}^n\widetilde{\frp}_i^{m_i}\widetilde J\right) \subseteq \widetilde{J} \subseteq \OO_{\widetilde A}. \]
    Then a reduction $\sigma_{m'}$ of $\sigma$ on $\mathfrak{o}_A = \OO_{\widetilde A}/p^{m'}\OO_{\widetilde A}$, computed as in the preceding paragraph, will induce well-defined approximations $\sigma_{m_0}$ on $Q_{m_0}$ and $\sigma_{m_0+1}$ on $Q_{m_0+1}$. 
\end{remark}

\begin{remark} \label{rmk:a-numbers}
    The choice of $m_0$ made in Step \ref{alg:FVclasses:m0} in Algorithm~\ref{alg:FVclasses} allows us to compute the $a$-numbers of the Dieudonn{\'e} modules.
    Indeed, if~$\widetilde{I}_k$ is in the output set, then the $a$-number of the corresponding Dieudonn\'e module equals
    \[ \dim_{\F_q} \left( \frac{I_{k,m_0}}{F_{m_0}(I_{k,m_0})+V_{m_0}(I_{k,m_0})} \right).\]
    If one needs to know the action of $F'$ and $V'$ to higher precision, it suffices to increase the value of $m_0$ in Step \ref{alg:FVclasses:m0}. 
\end{remark}

\begin{thm}\label{thm:alg:FVclasses}
    Algorithm~\ref{alg:FVclasses} is correct.    
\end{thm}
\begin{proof}
    By Proposition~\ref{prop:ext_of_FV_stable}, the extension of a $W'_R\{F',V'\}$-ideal to $\OO_{A'}$ is a $W'_R\{F',V'\}$-ideal.
    Hence, considering only such ideals (that is, whose extension is in $\Upsilon$) in the input is not a limitation.

    The algorithm is an application of Lemma~\ref{lem:step3_eq_cond}: for each $I \in \mathcal{I}$, we want to check whether it is stable under the action of $F'$ and $V'$.
    The parameter $m_0$ is chosen in such a way that is it sufficient to look at the image of $I$ image in the finite quotient $Q_{m_0}$ defined in Step~\eqref{alg:FVclasses:Qm0}, see Remark~\ref{rmk:alg3_sigma_fin_quot}.
    So, we need to compute representations of $F'$ and $V'$ in $Q_{m_0}$.
    
    First consider $F'$. 
    In Step (\ref{alg:FVclasses:foreachnu}.i), for each $\nu$, we produce an element $\gamma_0$ in the multiplicative group of $\OO_{\widetilde A}/\widetilde \frP_g^{e_\nu(m_0+1)} \cong \mathcal O_{LE_\nu}/\frp_{LE_\nu}^{e_{\nu}(m_0+1)}$.
    This element $\gamma_0$ exists because the extension $LE_\nu/E_\nu$ is unramified, see \cite[p.~544]{Wat69} and \cite[Chapter V, \S 2, Corollary]{SerreLocalFields}.
    Under this isomorphism we get an element that with slight abuse of notation we also call $\gamma_0 \in \mathcal O_{LE_\nu}^\times$. 
    The property in (\ref{alg:FVclasses:foreachnu}.i) then translates to the equality $N_{LE_\nu/E_\nu}(\gamma_0) = \pi_\nu/t_{\nu}^{\val_{\nu}(\pi_\nu)}$ in $\mathcal O_{LE_\nu}/\frp_{LE_\nu}^{e_\nu(m_0+1)}$, for some uniformizer $t_\nu$ in $E_\nu$. 
    By Lemma~\ref{lem:Frob_fin_quot}, there exists an element 
    $\gamma_1 \in \frp_{LE_\nu}^{e_\nu(m_0+1)}$
    such that $N_{LE_\nu/E_\nu}(\gamma_0+\gamma_1) = \pi_\nu/t_{\nu}^{\val_{\nu}(\pi_\nu)}$. This means that if we define $F'_\nu=\alpha_\nu \circ \sigma$ with 
    \[\alpha_\nu=(1,\ldots,t_\nu^{\val_{\nu}(\pi_\nu)g_\nu/a}(\gamma_0 +\gamma_1)),
    \]
    then $F'_\nu$ has the Frobenius property and is of $W$-type, by Lemma~\ref{lem-F}. Putting $\alpha=(\alpha_\nu)_{\nu | p} \in A'$ we get $F'=\alpha \circ \sigma$, an additive map on $A'$ of $W$-type. Step \eqref{alg:FVclasses:alpha} produces an element $\alpha' \in \widetilde A$, from which we define the additive map $F_{m_0+1}\colon Q_{m_0+1}\to Q_{m_0+1}$ in Step~\eqref{alg:FVclasses:Fm01}. 
    Since all the approximations of $\sigma$ computed using Remark~\ref{rmk:alg3_sigma_fin_quot} are compatible with each other, we see that
    $F_{m_0+1}$ is the reduction of $F'$ restricted to $J'$ under the isomorphism $Q_{m_0+1} \cong J'/p^{m_0+1}J'$ and that, similarly, $F'$ also induces $F_{m_0}$ on $Q_{m_0}$.

    Now consider $V'$.
    We cannot directly compute an approximation of $V'$ on $Q_{m_0}$ using $F_{m_0}$, since this map might not be invertible.
    Instead, in Step~\ref{alg:FVclasses:V}, we compute $V_{m_0}$ from the representation $F_{m_0+1}$ of $F'$ on the larger quotient~$Q_{m_0+1}$, that is, with a higher precision.
    We now give the details.
    Let $\mathrm{pr}_{m_0}\colon J'\to Q_{m_0}$ and $\mathrm{pr}_{m_0+1}\colon J'\to Q_{m_0+1}$ be the natural projections.
    Observe that $\mathrm{pr}_{m_0} = \mathrm{pr} \circ \mathrm{pr}_{m_0+1}$.
    The following diagram shows the maps involved in this proof.

\begin{center}
\begin{tikzpicture}
    \node[] (J1) at (0,0) {$J'$};
    \node[] (J2) at (5,0) {$J'$};
    \node[] (J3) at (10,0) {$J'$};
    \node[] (Q1) at (1,-1.5) {$Q_{m_0+1}$};
    \node[] (Q2) at (6,-1.5) {$Q_{m_0+1}$};
    \node[] (Q3) at (9,-1.5) {$Q_{m_0+1}$};
    \node[] (QQ1) at (0,-3) {$Q_{m_0}$};
    \node[] (QQ2) at (5,-3) {$Q_{m_0}$};
    \node[] (QQ3) at (10,-3) {$Q_{m_0}$};
    \draw[->] (J1) -- node[anchor=west] {\small{$\mathrm{pr_{m_0+1}}$}} (Q1.north west);
    \draw[->] (J1) -- node[anchor=east] {\small{$\mathrm{pr_{m_0}}$}} (QQ1);
    \draw[->] (Q1.south west) -- node[anchor=west] {\small{$\mathrm{pr}$}} (QQ1);
    \draw[->] (J2) -- node[anchor=west] {\small{$\mathrm{pr_{m_0+1}}$}} (Q2.north west);
    \draw[->] (J2) -- node[anchor=east] {\small{$\mathrm{pr_{m_0}}$}} (QQ2);
    \draw[->] (Q2.south west) -- node[anchor=west] {\small{$\mathrm{pr}$}} (QQ2);
    \draw[->] (J3) -- node[anchor=east] {\small{$\mathrm{pr_{m_0+1}}$}} (Q3.north east);
    \draw[->] (J3) -- node[anchor=west] {\small{$\mathrm{pr_{m_0}}$}} (QQ3);
    \draw[->] (Q3.south east) -- node[anchor=east] {\small{$\mathrm{pr}$}} (QQ3);
    
    \draw[->] (J1) -- node[anchor=south] {\small{$V'$}} (J2);
    \draw[->] (J2) -- node[anchor=south] {\small{$F'$}} (J3);
    \path[->] (J1.north east) edge[bend left=15] node[anchor=north]{\small{$p$}} (J3.north west);
    
    \draw[->] (Q2) -- node[anchor=south] {\small{$F_{m_0+1}$}} (Q3);
    \draw[->] (QQ2) -- node[anchor=south] {\small{$F_{m_0}$}} (QQ3);
    \draw[->] (QQ1) -- node[anchor=south] {\small{$V_{m_0}$}} (QQ2);
\end{tikzpicture}
\end{center}

    Let~$x_{m_0}$ be an element of $Q_{m_0}$.
    Denote by $x$ a preimage of $x_{m_0}$ in $J'$ via $\mathrm{pr}_{m_0}$.
    Let $x_{m_0+1} = \mathrm{pr}_{m_0+1}(x)$ and $y =V'(x) \in J$.
    It remains to show that $\mathrm{pr}_{m_0}(y) = V_{m_0}(x_{m_0})$.
    By the construction of $V_{m_0}$, we have 
    \[ V_{m_0}(x_{m_0}) = \mathrm{pr}(z_{m_0+1}), \]
    where $z_{m_0+1}$ is an element of $Q_{m_0+1}$ such that 
    \[ p\cdot x_{m_0+1} = F_{m_0+1}(z_{m_0+1}). \]
    Let $z$ be a preimage of $z_{m_0+1}$ in $J'$ via $\mathrm{pr}_{m_0+1}$.
    Note that $F'(y) =p\cdot x$.
    Hence
    \[  \mathrm{pr}_{m_0+1}(F'(y)) = p\cdot x_{m_0+1}=F_{m_0+1}(z_{m_0+1})=\mathrm{pr}_{m_0+1}(F'(z)), \]
    where the last equality follows by the definition of $F_{m_0+1}$.
    Therefore 
    \[ F'(y)-F'(z)=F'(y-z) \in p^{m_0+1}J'. \]
    By applying $V'$, we then get
    \[  p(y-z) \in V'(p^{m_0+1}J) \subseteq p^{m_0+1}J'. \]
    Dividing by $p$ gives
    \[ y-z \in p^{m_0}J'. \]
    By applying $\mathrm{pr}_{m_0}$ we obtain
    \[ \mathrm{pr}_{m_0}(y) = \mathrm{pr}_{m_0}(z) = \mathrm{pr} \circ \mathrm{pr}_{m_0+1}(z) = \mathrm{pr}(z_{m_0+1}) = V_{m_0}(x_{m_0}), \]
    as required.
\end{proof}

\section{Computations and applications} \label{sec:examples}
\subsection{Overview of the algorithm}
As explained in Section~\ref{sec-abvars}, the computation of the isomorphism classes of abelian varieties in the isogeny class $\mathcal{A}_\pi$ is divided into several parts. 
We briefly summarize how each part of the algorithm works:
\begin{itemize}
    \item The prime-to-$p$ part $\prod_{\ell \neq p} \mathfrak{X}_{\pi,\ell}$: use the results in Section~\ref{sec:tate}, in particular Theorem~\ref{thm:W_R_X_ell}, to reduce the problem to a finite number of calls of \cite[\texttt{ComputeW}]{Mar23_Loc}.
    \item The $p$-part $\mathfrak{X}_{\pi,p}$, which splits into three subparts: the \'etale, the multiplicative  and the local-local parts, see Subsection~\ref{sec:connetale}.
    The computation of the first two is reduced to a finite number of calls of \cite[\texttt{ComputeW}]{Mar23_Loc} by Proposition~\ref{prop:Rp01} and the results of Section~\ref{sec:tate}.
    The computation of the local-local part is described in Section~\ref{sec:aftergreenredblue}, see in particular Algorithms~\ref{alg:Deltaclasses_OO} and \ref{alg:FVclasses}.
    \item Finally, the isomorphisms classes with given local parts form an orbit of the class group $\Cl(S)$ of the endomorphism ring $S$, which is uniquely determined by the local parts. 
    The computation of $\Cl(S)$ is reduced to the classical computation of the class group $\Cl(\OO_E)$ of the maximal order $\OO_E$ of $E$, see \cite{klupau05}.
\end{itemize}
We have implemented the whole procedure in Magma \cite{Magma}.
The package is available at \cite{IsomClAbVarFqCommEndAlg}.
The repository includes an appendix with the pseudo-code, together with some auxiliary technical details.
The examples in this section, which have been computed using our implementation, show some unexpected behavior of endomorphism rings, which we now describe in detail.

\subsection{Properties of the lattice of endomorphism rings}
Let $\mathcal{A}_\pi$ be an isogeny class of abelian varieties over $\F_q$ of dimension $g$ with commutative endomorphism algebra $E=\Q[\pi]$.
Put $R=\Z[\pi,q/\pi]$, as before.
Let $\mathcal{S}$ be the set of overorders of $R$ and let $\mathcal{E}$ be the subset of~$\mathcal{S}$ consisting of endomorphism rings, that is, orders $T$ such that there exists $X \in \mathcal{A}_\pi$ with $\End(X)=T$.
Let $\mathcal T$ be the set of all maximal ideals of $R$ except the one above $p$ of slope in $(0,1)$. In the rest of the section, we will associate to each order $S \in \mathcal{S}$ four nonnegative integers: $n(S)$, the number of isomorphism classes of abelian varieties in $\mathcal{A}_\pi$ with endomorphism ring $S$ (as defined above); $w(S)$,  the number of elements of $\cW_{\mathcal{T}}(S)$ with multiplicator ring $S$; $d(S)$, the number of $W'_R\{F',V'\}$-ideals with endomorphism ring $S$; and $h(S)$, the class number of $S$.
\begin{prop} \label{prop:nwdh} For all $S \in \mathcal{S}$,
    \[
    n(S)=w(S) d(S)h(S).
    \]
\end{prop}
\begin{proof} This follows from Theorem~\ref{thm-equivalence} combined with Propositions \ref{prop:S0_to_prod_ell}, \ref{prop:ell_to_prod_frl} and~\ref{prop:Rp01}.
\end{proof}

Hence, we have $S\in \mathcal{E}$, that is, $S$ is the endomorphism ring for some $X \in \mathcal{A}_\pi$, if and only if $n(S)>0$, which is also equivalent to $d(S)>0$. 

Consider the following three statements: 
\begin{enumerate}[(1)]
    \item\label{jumps} For every $S\in \mathcal{E}$ and $T\in \mathcal{S}$, if $S\subseteq T$ then $T \in \mathcal{E}$.
    \item\label{min_end} The order $S = \cap_{T \in \mathcal{E}} T$ is in $\mathcal{E}$.
    \item\label{num_iso} For every $S$ in $\mathcal{E}$, $n(\OO_E)$ divides $n(S)$.
\end{enumerate}
If $\mathcal{A}_\pi$ is ordinary, or almost ordinary, or if $q$ is prime, then statements \ref{jumps}, \ref{min_end} and \ref{num_iso} hold true; see~\cite[Corollary~4.4]{MarAbVar18}, \cite[Theorem~4.5]{OswalShankar19} (for odd characteristic) and Corollary~\ref{cor-max} together with Example~\ref{ex:almost_ord} (for all characteristics).\\

From now on, assume further that $\mathcal{A}_\pi$ has $p$-rank $<g$.
Let $\frp$ be the maximal ideal of $R=\Z[\pi,q/\pi]$ corresponding to the local-local part.
Consider the following statements:
\begin{enumerate}[(*)]
    \item\label{Endmax01} $\End(X)_\frp$ is maximal for every $X$ in $\mathcal{A}_\pi$.
\end{enumerate}
\begin{enumerate}[(a)]
    \item \label{Pabove} There exists a unique maximal ideal $\frP$ of $\OO_E$ above $\frp$ with slope $\leq 1/2$.
    \item \label{slope12} All maximal ideals of $\OO_E$ above $\frp$ have slope $1/2$.
    \item \label{minsl} For each maximal ideal $\frP$ of $\OO_E$ above $\frp$ we have that the slope equals $1/n_\frP$ or $1-1/n_{\frP}$, where~$n_\frP$ is the dimension of $E_\frP$ over $\Q_p$, that is, the product of the ramification index and the inertia degree.
\end{enumerate}

Note that $\mathcal{A}_\pi$ is almost ordinary if and only if all three conditions \ref{Pabove}, \ref{slope12} and \ref{minsl} hold true.

\begin{prop}\label{prop:examples_implications}
    Let $\mathcal{A}_\pi$ be an isogeny class of $g$-dimensional abelian varieties over $\F_q$ with commutative endomorphism algebra and $p$-rank $<g$.
    We have the following implications:
    \[ \ref{Pabove}+\ref{slope12}+\ref{minsl} \Longrightarrow \ref{Endmax01} \Longrightarrow \ref{jumps}+\ref{min_end}+\ref{num_iso}. \]
\end{prop}
\begin{proof}
    As pointed out above, \ref{Pabove}, \ref{slope12} and \ref{minsl} all hold if and only if $\mathcal{A}_\pi$ is almost ordinary.
    The first implication is shown in Example~\ref{ex:almost_ord}.

    Corollary~\ref{cor:min_end_ell_01} shows that \ref{Endmax01} implies \ref{jumps} and \ref{min_end}.
    We now show that \ref{Endmax01} also implies \ref{num_iso}, completing the proof of the second implication.
    By Proposition~\ref{prop:nwdh}, the number $n(S)$ is divisible by $d(S)h(S)$.
    Observe that \ref{Endmax01} implies that $d(S)=d(\OO_E)$ for each $S \in \mathcal{E}$, and that $h(\OO_E)$ divides $h(S)$ since the extension map $\Cl(S)\to \Cl(\OO_E)$ is surjective.
    Every fractional $\OO_E$-ideal is locally principal, hence $w(\OO_E)=1$.
    It follows that the number of abelian varieties with endomorphism ring $\OO_E$ is determined only by the class group of $\OO_E$ and by the local-local part.
    More precisely, we have $n(\OO_E) = d(\OO_E)h(\OO_E)$. 
    Combining all these statements we see that $n(\OO_E)$ divides $n(S)$ for each $S\in \mathcal{E}$, as required.
\end{proof}
In each of the following three examples (Examples~\ref{ex:simple1}, \ref{ex:nonsimple1} and ~\ref{ex:simple2}), we negate exactly one of the three statements \ref{Pabove}, \ref{slope12}, \ref{minsl}, and show that \ref{Endmax01} does not hold and that \ref{jumps}, \ref{min_end} or \ref{num_iso} fails.

In Example~\ref{ex:simple3}, we exhibit an isogeny class with $p$-rank $<g$ for which all of \ref{jumps}, \ref{min_end} and \ref{num_iso} hold true, but \ref{Endmax01} fails. 
This example also shows that for abelian varieties that are neither ordinary nor defined over a prime field, \ref{jumps} + \ref{min_end} + \ref{num_iso} is not equivalent to the isogeny class being almost ordinary.

In Example~\ref{ex:nonsimple2}, we exhibit an isogeny class with abelian varieties with the same endomorphism ring but whose $a$-numbers are different.

In all the examples below we have $w(S)=1$ for all $S \in \mathcal S$. We would have $w(S) > 1$, for some $S \in \mathcal S$, if and only if $R$ is not Bass at some maximal ideal of $\mathcal T$. For an example of an $R$ with this property, see \cite[Example 7.4]{MarAbVar18}. 

We conclude the section with some observations and a further example, arising from the computations we have performed.

\subsection{Examples}
In the following graphs, each vertex represents an overorder $S$ of $R$, labeled with the pair $(d(S),h(S))$, and each edge represents an inclusion, labeled with its index.
If $S$ is in $\mathcal{E}$, that is, $S$ is the endomorphism ring for some $X \in \mathcal{A}_\pi$, then we add in the subscript of the label, the $a$-numbers of the isomorphism classes of Dieudonn{\'e} modules. 
In all examples (except Example~\ref{ex:nonsimple2}) we write only one value, since all the Dieudonn{\'e} modules have the same $a$-number.

\begin{example}\label{ex:simple1}
    Consider the polynomial
    \[h=x^8 + x^7 + x^6 + 4x^5 - 4x^4 + 16x^3 + 16x^2 + 64x + 256.\]
    It determines an isogeny class of geometrically simple abelian fourfolds over $\F_{4}$ with commutative endomorphism algebra $E=\Q[\pi]=\Q[x]/h$.
    The isogeny class has LMFDB label \href{http://www.lmfdb.org/Variety/Abelian/Fq/4/4/b_b_e_ae}{4.4.b\_b\_e\_ae} and $p$-rank~$2$.
    The algebra $E$ has $3$ places above $p=2$ with slopes, ramification indices and inertia degrees equal to $(0, 1, 2), (1, 1, 2), (1/2, 2, 2)$, respectively.
    Hence, this isogeny class satisfies conditions \ref{Pabove} and \ref{slope12}, but not condition~\ref{minsl}.
    
    The unique maximal ideal of the order $R=\Z[\pi,16/\pi]$ in $E$ above $2$ is singular and in $\mathcal{P}_{R_2}^{(0,1)}$.
    One computes that $R$ has $13$ overorders $S$ and for each of these $w(S)=1$.
    It follows from Proposition~\ref{prop:nwdh} that $n(S)=d(S) h(S)$.
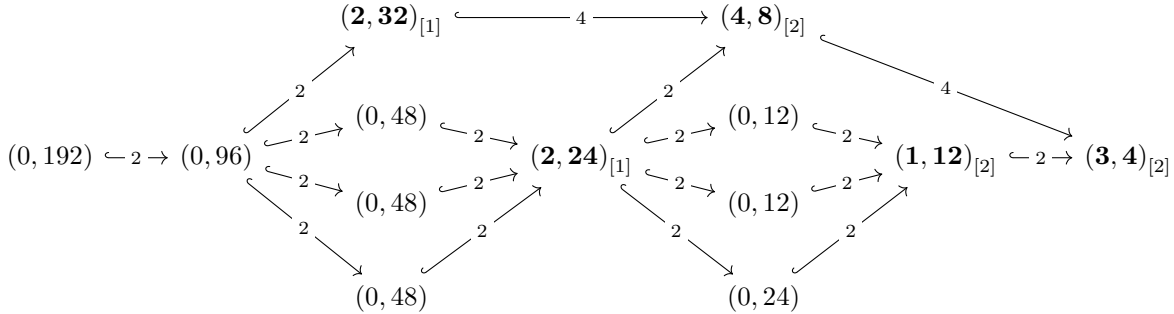
    \[
    \begin{tikzcd}
             &          & \mathbf{(2,32)}_{[1]} \arrow[rr,hook,"4" description] &        & \mathbf{(4,8)}_{[2]} \arrow[ddrr,hook,"4" description] & &\\
             &          & (0,48) \arrow[dr,hook,"2" description] &        & (0,12)\arrow[dr,hook,"2" description] & &\\[-2em]
    (0,192) \arrow[r,hook,"2" description] & (0,96) \arrow[uur,hook,"2" description] \arrow[ur,hook,"2" description] \arrow[ddr,hook,"2" description] \arrow[dr,hook,"2" description] &        & \mathbf{(2,24)}_{[1]} \arrow[uur,hook,"2" description] \arrow[ur,hook,"2" description] \arrow[ddr,hook,"2" description] \arrow[dr,hook,"2" description] &        & \mathbf{(1,12)}_{[2]} \arrow[r,hook,"2" description] & \mathbf{(3,4)}_{[2]} \\[-2em]
             &          & (0,48) \arrow[ur,hook,"2" description]&        & (0,12) \arrow[ur,hook,"2" description]& &\\
             &          & (0,48) \arrow[uur,hook,"2" description] &        & (0,24) \arrow[uur,hook,"2" description] & &
    \end{tikzcd}
    \]
    We see that statements \ref{jumps}, \ref{min_end} and \ref{num_iso} do not hold true for this isogeny class.
    Hence, also statement \ref{Endmax01} does not hold for this isogeny class: only the maximal order has maximal local-local part.
\end{example}

\begin{example}\label{ex:nonsimple1}
    Put $h_1=x^2 - 2x + 4$, $h_2=x^2 + 2x + 4$ and consider the polynomial 
    \[h=x^4 + 4x^2 + 16 = h_1 \cdot h_2.
    \]
    It determines an isogeny class $\mathcal{A}_\pi$ of abelian surfaces over $\F_{4}$ with commutative endomorphism algebra $E=\Q[\pi]=\Q[x]/h$. 
    Put $E_1=\Q[\pi_1]=\Q[x]/h_1$ and $E_2=\Q[\pi_2]=\Q[x]/h_2$.
    Any $X \in \mathcal{A}_\pi$ is isogenous to a product of supersingular elliptic curves $C_1 \times C_2$ with $C_i \in \mathcal{A}_{\pi_i}$. Note that the supersingular curves in $\mathcal{A}_{\pi_1}$ are quadratic twists of the ones in $\mathcal{A}_{\pi_2}$.
    The LMFDB labels of $\mathcal{A}_\pi$, $\mathcal{A}_{\pi_1}$, $\mathcal{A}_{\pi_2}$ are \href{http://www.lmfdb.org/Variety/Abelian/Fq/2/4/a_e}{2.4.a\_e}, \href{http://www.lmfdb.org/Variety/Abelian/Fq/1/4/ac}{1.4.ac}, \href{http://www.lmfdb.org/Variety/Abelian/Fq/1/4/c}{1.4.c}, respectively.
    Therefore, conditions \ref{slope12} and \ref{minsl} hold, but not condition~\ref{Pabove}.

    The unique maximal ideal of the order $R=\Z[\pi,16/\pi]$ in $E$ above $2$ is singular and in $\mathcal{P}_{R_2}^{(0,1)}$.
    One computes that $R$ has $13$ overorders $S$ and for each of these $w(S)=1$.
     It follows from Proposition~\ref{prop:nwdh} that $n(S)=d(S) h(S)$.
    
    An abelian variety $X$ is isomorphic to a product $C_1\times C_2$ with $C_i \in \mathcal{A}_{\pi_i}$ 
    if and only if the order $S$ is the direct product of $S_1\times S_2$ with $S_1$ an order in~$E_1$ and $S_2$ an order in~$E_2$. Such orders are marked with a $*$ in the label in the graph below.
    The number of isomorphism classes which are products of elliptic curves can be derived from \cite[Theorem~4.5]{Schoof87} (which corrects parts of the statement of  \cite[Theorem~4.5]{Wat69}).
    
    \[
    \begin{tikzcd}
             &          &  &        & & (0,1,*)\arrow[dddr,hook,"2" description] &\\[-2em]
             &          & (0,1)\arrow[ddr,hook,"2" description] &        & (0,1,*)\arrow[ur,hook,"2" description] \arrow[dr,hook,"2" description] & &\\[-2em]
             &          &  &        & & (0,1,*)\arrow[dr,hook,"2" description]  &\\[-2em]
             &          & (0,2)\arrow[r,hook,"2" description] & \mathbf{(2,1)}_{[1]}\arrow[uur,hook,"2" description] \arrow[r,hook,"2" description]\arrow[dr,hook,"2" description]  & \mathbf{(2,1)}_{[1]}\arrow[rr,hook,"4" description] & & \mathbf{(4,1,*)}_{[2]} \\
       (0,4) \arrow[r,hook,"2" description] & (0,2) \arrow[uuur,hook,"2" description]\arrow[ur,hook,"2" description] \arrow[r,hook,"2" description] \arrow[dr,hook,"2" description] & (0,1)\arrow[ur,hook,"2" description] &   &  \mathbf{(4,1)}_{[2]}\arrow[urr,hook,"4" description] &  & \\
             &          & \mathbf{(2,2)}_{[1]}\arrow[urr,hook,"4" description] &        &  & &
    \end{tikzcd}
    \]
    We see that statements \ref{jumps}, \ref{min_end} and \ref{num_iso} do not hold true for this isogeny class.
    Hence, also statement \ref{Endmax01} does not hold for this isogeny class: only the maximal order has maximal local-local part.
\end{example}

\begin{example}\label{ex:simple2}
    Consider the polynomial 
    \[h=x^8 - 6x^7 + 18x^6 - 36x^5 + 68x^4 - 144x^3 + 288x^2 - 384x + 256.\]
    It determines an isogeny class of geometrically simple abelian fourfolds over $\F_{16}$ with commutative endomorphism algebra $E=\Q[\pi]=\Q[x]/h$.
    The isogeny class has LMFDB label \href{http://www.lmfdb.org/Variety/Abelian/Fq/4/4/ag_s_abk_cq}{4.4.ag\_s\_abk\_cq} and $p$-rank~$0$, but it is not supersingular.
    The algebra $E$ has two places above $p=2$ with slopes, ramification indices and inertia degrees equal to $(1/4,2,2),(3/4,2,2)$, respectively.
    Hence, this isogeny class satisfies conditions \ref{Pabove} and \ref{minsl}, but not condition~\ref{slope12}.
    
    The order $R=\Z[\pi,4/\pi]$ has two singular maximal ideals: one above $11$, and one above $2$ which is in $\mathcal{P}_{R_2}^{(0,1)}$.
    The index of $R$ in $\OO_E$ is $704=11\cdot 64$.
    The computation of $\prod_{\ell\neq p}\mathfrak{X}_{\pi,\ell}$ returns
    $2$ classes, which can be represented by the maximal order~$\OO_E$ and the unique overorder $T$ of $R$ with index $[\OO_E:T]=11$. 
    One then computes that $R$ has $34$ overorders $S$ and for each of these $w(S)=1$.
    It follows from Proposition~\ref{prop:nwdh} that $n(S)=d(S) h(S)$.
    Since the graph of inclusions of all overorders is too unwieldy, in the following graph, we draw the lattice of inclusion of the overorders $S$ of $R$ which are actually endomorphism rings for some $X \in \mathcal{A}_\pi$, that is, for which $d(S) > 0$. 
    
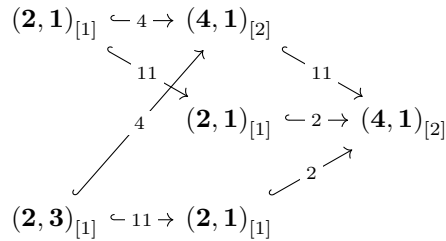
\[
    \begin{tikzcd}
    {\mathbf{(2,1)}_{[1]}} \arrow[r, "4" description, hook] \arrow[rd, "11" description, hook]  & {\mathbf{(4,1)}_{[2]}} \arrow[rd, "11" description, hook]      &                          \\
    & {\mathbf{(2,1)}_{[1]}} \arrow[r, "2" description, hook]  & {\mathbf{(4,1)}_{[2]}} \\
    {\mathbf{(2,3)}_{[1]}} \arrow[r, "11" description, hook] \arrow[ruu, "4" description, hook] & {\mathbf{(2,1)}_{[1]}} \arrow[ru, "2" description, hook] &                         
    \end{tikzcd}
    \]
    On the one hand, we see from the graph that \ref{min_end} and \ref{num_iso} do not hold true for this isogeny class.
    Hence, also statement \ref{Endmax01} does not hold for this isogeny class: only the maximal order has maximal local-local part.
    On the other hand, if one considers all overorder of $R$, the inclusion with index $4$ does not factor as the composition of two inclusions. 
    It follows that this isogeny class satisfies condition \ref{jumps}.
\end{example}

\begin{example}\label{ex:simple3}
    Consider the polynomial 
    \[h=x^6 - x^5 - 3x^4 + 45x^3 - 27x^2 - 81x + 729.\]
    It determines an isogeny class of geometrically simple abelian threefolds over $\F_{9}$ with commutative endomorphism algebra $E=\Q[\pi]=\Q[x]/h$.
    The isogeny class has LMFDB label \href{http://www.lmfdb.org/Variety/Abelian/Fq/3/9/ab_ad_bt}{3.9.ab\_ad\_bt} and $p$-rank~$1$.
    The algebra $E$ has three places above $p=3$ with slopes, ramification indices and inertia degrees equal to $(0, 1, 1)$, $(1, 1, 1)$ and $(1/2, 1, 4 )$, respectively.
    Hence, this isogeny class satisfies conditions \ref{Pabove} and \ref{slope12}, but not condition~\ref{minsl}.
    
    The unique singular maximal ideal of the order $R=\Z[\pi,9/\pi]$ is the maximal ideal $(3,\pi,9/\pi)$, which is in $\mathcal{P}_{R_3}^{(0,1)}$. 
    One computes that $R$ has $4$ overorders $S$ and for each of these $w(S)=1$.
    It follows from Proposition~\ref{prop:nwdh} that $n(S)=d(S) h(S)$.
    \[
    \begin{tikzcd}
    (0,60)\arrow[r, "3" description, hook]&(0,30)\arrow[r, "3" description, hook]&\mathbf{(2,20)}_{[1]}\arrow[r, "9" description, hook]&\mathbf{(2,2)}_{[2]}
    \end{tikzcd}
    \]
    Statements \ref{jumps}, \ref{min_end} and \ref{num_iso} hold true for this isogeny class, but statement \ref{Endmax01} does not hold, since only the maximal order has maximal local-local part.
\end{example}

\begin{example}\label{ex:nonsimple2}
    Consider the polynomial 
    \[h=x^6 + 11x^5 + 60x^4 + 208x^3 + 480x^2 + 704x + 512.\]
    It determines an isogeny class of abelian threefolds over $\F_{8}$ with commutative endomorphism algebra $E=\Q[\pi]=\Q[x]/h$.
    The isogeny class has LMFDB label \href{http://www.lmfdb.org/Variety/Abelian/Fq/3/8/l_ci_ia}{3.8.l\_ci\_ia} and $p$-rank $1$.
    Any $X \in \mathcal A_\pi$ is isogenous to a product of a supersingular elliptic curve and an almost-ordinary abelian surface.    
    
    The unique singular maximal ideal of the order order $R=\Z[\pi,8/\pi]$ is the maximal ideal $(2,\pi,8/\pi)$, which is in $\mathcal{P}_{R_2}^{(0,1)}$. 
    One computes that $R$ has $16$ overorders $S$ and for each of these $w(S)=1$.
    It follows from Proposition~\ref{prop:nwdh} that $n(S)=d(S) h(S)$.
    If $d(S)>0$ for an overorder $S$, then we add in the subscript of the label, the $a$-numbers of the isomorphism classes of Dieudonn{\'e} modules, using the exponent to denote the number of isomorphism classes of Dieudonn\'e modules with the indicated $a$-number. 
    The unique endomorphism ring with index $2$ in the maximal order is the endomorphism ring of $7$ isomorphism classes of abelian varieties, all with pairwise non-isomorphic Dieudonn\'e modules. Six of these isomorphism classes of abelian varieties have $a$-number $1$ while the last one has $a$-number $2$.
    The only overorder of $R$ which is a direct product of two orders is the maximal order. 
    We highlight this in the graph with a $*$ in its label.
    This means that the only abelian varieties that are isomorphic to a product of an elliptic curve and an abelian surface are the ones with maximal endomorphism ring.

    \begin{tikzcd}
        &&&& (0,2) \arrow[rd,"2" description] \arrow[rdd, "2" description]&&\\
        && (0,2) \arrow[rd, "2" description]& (0,1) \arrow[rdd, "2" description]& (0,1) \arrow[rdd, "2" description] \arrow[rd, "2" description]  & (0,2) \arrow[rd, "2" description] &\\
        (0,8) \arrow[r, "2" description] & (0,4) \arrow[ru, "2" description] \arrow[rd, "2" description] \arrow[r, "2" description] & (0,2) \arrow[ru, "2" description] \arrow[r, "2" description] \arrow[rd, "2" description] & (0,2) \arrow[rd, "2" description] \arrow[rdd, "2"] \arrow[ru, "2" description] \arrow[ruu,"2" description] && (0,1) \arrow[r,"2" description] & \mathbf{(1,1,*)}_{[2]} \\
        && (0,4) \arrow[ru,"2" description]& (0,1) \arrow[r,"2"']& \mathbf{(4,1)}_{[1^4]} \arrow[r,"2" description]& \mathbf{(7,1)}_{[1^6,2]} \arrow[ru,"2" description] &\\
        &&&& (0,2) \arrow[ru,"2" description] \arrow[ruuu,"2" description] &
    \end{tikzcd}
\end{example}

\begin{observation} \label{obs}
We have computed the isomorphism classes of abelian varieties for several thousands of isogeny classes (using lists of isogeny classes found in the LMFDB~\cite{lmfdb},  which was compiled following  \cite{DupuyKedlayaRoeVincent2020}) with commutative endomorphism algebra $E$ of dimension $g$ and $p$-rank $<g-1$ over non-prime finite fields.
We observed the following:
\begin{itemize}
    \item the order $R$ is never an endomorphism ring;
    \item there is always an endomorphism ring which is not maximal at the local-local part, that is, condition \ref{Endmax01} does not hold. Note that if this observation would always hold then the first implication in Proposition~\ref{prop:examples_implications} is actually an equivalence.
\end{itemize}
\end{observation}

\begin{remark} Over a prime field, every isogeny class contains an abelian variety with $R$ as endomorphism ring, see~\cite{CentelegheStix15}. Moreover, over any finite field, any isogeny class of ordinary abelian varieties contains an abelian variety with $R$ as endomorphism ring, see~\cite{Del69}. 

Contrary to this, we showed in Example~\ref{ex:almost_ord} that \emph{no} almost-ordinary abelian variety over a non-prime finite field has $R$ as endomorphism ring. 
\end{remark}

Recall that Proposition~\ref{prop-01} says that if $\mathcal{A}_\pi$ is an isogeny class over $\F_q$ with commutative endomorphism algebra which is non-ordinary and such that $q=p^a$ is not a prime, then the order $R=\Z[\pi,q/\pi]$ is not maximal at the maximal ideal $\frp=(p,\pi,q/\pi)$.
The converse does not hold, as we show in the following example.

\begin{example}\label{ex:end_max_p} 
    Let $\mathcal{A}_\pi$ be an isogeny class of simple abelian surfaces over $\F_2$ with $p$-rank $0$ determined by the characteristic polynomial $h(x) = x^8 - 2x^5 - 4x^3 + 16$; see the LMFDB-label \href{http://www.lmfdb.org/Variety/Abelian/Fq/4/2/a_a_ac_a}{4.2.a\_a\_ac\_a}.
    Let $E=\Q[\pi]=\Q[x]/h$ be the endomorphism algebra of the isogeny class.
    The order $R=\Z[\pi,2/\pi]$ has a unique maximal ideal above $p=2$, which is $\frp = (2,\pi,2/\pi)$. 
    One computes that the quotient $\frp/\frp^2$ has $8$ elements, while $R/\frp\simeq \F_2$.
    This shows that $\frp$ is singular, that is, the order $R_\frp$ is not maximal.
\end{example}

\section*{Acknowledgements}
The authors are grateful to Sergey Rybakov, Tommaso Giorgio Centeleghe and Jacob Stix for pointing out two mistakes in previous versions of the paper.
The authors are thankful to Caleb Springer for suggestions to improve the exposition in Section~\ref{sec:examples}.
The authors thank the anonymous referee for valuable feedback on how to improve the overall readability of the paper.

The second author was partially supported by the Dutch Research Council (NWO), through grants VI.Veni.192.038 and VI.Vidi.223.028. 

The third author was supported by NWO through grant VI.Veni.202.107, by Agence Nationale de la Recherche under the MELODIA project (grant number ANR-20-CE40-0013) and by Marie  Sk{\l}odowska-Curie Actions - Postdoctoral Fellowships 2023 (project 101149209 - AbVarFq).

\bibliographystyle{amsplain}
\bibliography{references}

\def\cprime{$'$}
\providecommand{\bysame}{\leavevmode\hbox to3em{\hrulefill}\thinspace}
\providecommand{\MR}{\relax\ifhmode\unskip\space\fi MR }
\providecommand{\MRhref}[2]{%
  \href{http://www.ams.org/mathscinet-getitem?mr=#1}{#2}
}
\providecommand{\href}[2]{#2}
\begin{thebibliography}{10}

\bibitem{Abramson}
Morton Abramson, \emph{Restricted combinations and compositions}, Fibonacci
  Quart. \textbf{14} (1976), no.~5, pp.~439--452.

\bibitem{BKM}
Jonas Bergstr\"{o}m, Valentijn Karemaker, and Stefano Marseglia,
  \emph{Polarizations of abelian varieties over finite fields via canonical
  liftings}, Int. Math. Res. Not. IMRN (2023), no.~4, 3194--3248. \MR{4565637}

\bibitem{BhatnagarFu}
Tejasi Bhatnagar and Yu~Fu, \emph{Abelian varieties with real multiplication:
  classification and isogeny classes over finite fields}, arXiv e-prints
  (2023), arXiv:2308.11105.

\bibitem{Magma}
Wieb Bosma, John Cannon, and Catherine Playoust, \emph{The {M}agma algebra
  system. {I}. {T}he user language}, J. Symbolic Comput. \textbf{24} (1997),
  no.~3-4, 235--265.

\bibitem{CentelegheStix15}
Tommaso Centeleghe and Jakob Stix, \emph{Categories of abelian varieties over
  finite fields, {I}: {A}belian varieties over {$\mathbb{F}_p$}}, Algebra
  Number Theory \textbf{9} (2015), no.~1, pp.~225--265.

\bibitem{CentelegheStix23}
\bysame, \emph{Categories of abelian varieties over finite fields {II}:
  {A}belian varieties over {$\mathbb F_q$} and {M}orita equivalence}, Israel J.
  Math. \textbf{257} (2023), no.~1, pp.~103--170.

\bibitem{chaiconradoort14}
Ching-Li Chai, Brian Conrad, and Frans Oort, \emph{Complex multiplication and
  lifting problems}, Mathematical Surveys and Monographs, vol. 195, American
  Mathematical Society, Providence, RI, 2014.

\bibitem{dadetz62}
Everett Dade, Olga Taussky, and Hans Zassenhaus, \emph{On the theory of orders,
  in particular on the semigroup of ideal classes and genera of an order in an
  algebraic number field}, Math. Ann. \textbf{148} (1962), pp.~31--64.

\bibitem{Del69}
Pierre Deligne, \emph{Vari\'et\'es ab\'eliennes ordinaires sur un corps fini},
  Invent. Math. \textbf{8} (1969), pp.~238--243.

\bibitem{DupuyKedlayaRoeVincent2020}
Taylor Dupuy, Kiran Kedlaya, David Roe, and Christelle Vincent, \emph{Isogeny
  classes of abelian varieties over finite fields in the {LMFDB}}, Arithmetic
  geometry, number theory, and computation, Simons Symp., Springer, Cham, 2021,
  pp.~375--448.

\bibitem{Greither}
Cornelius Greither, \emph{Cyclic {G}alois extensions of commutative rings},
  Lecture Notes in Mathematics, vol. 1534, Springer-Verlag, Berlin, 1992.

\bibitem{Guralnick84}
Robert Guralnick, \emph{The genus of a module}, J. Number Theory \textbf{18}
  (1984), no.~2, pp.~169--177.

\bibitem{Guralnick87}
Robert~M. Guralnick, \emph{The genus of a module. {II}. {R}o\u{\i}ter's
  theorem, power cancellation and extension of scalars}, J. Number Theory
  \textbf{26} (1987), no.~2, pp.~149--165.

\bibitem{HessPauliPohst03}
Florian Hess, Sebastian Pauli, and Michael~E. Pohst, \emph{Computing the
  multiplicative group of residue class rings}, Math. Comp. \textbf{72} (2003),
  no.~243, pp.~1531--1548.

\bibitem{Honda68}
Taira Honda, \emph{Isogeny classes of abelian varieties over finite fields}, J.
  Math. Soc. Japan \textbf{20} (1968), 83--95.

\bibitem{JKPRSBT18}
Bruce Jordan, Allan Keeton, Bjorn Poonen, Eric Rains, Nicholas Shepherd-Barron,
  and John Tate, \emph{Abelian varieties isogenous to a power of an elliptic
  curve}, Compos. Math. \textbf{154} (2018), no.~5, pp.~934--959.

\bibitem{Kani11}
Ernst Kani, \emph{Products of {CM} elliptic curves}, Collect. Math. \textbf{62}
  (2011), no.~3, pp.~297--339.

\bibitem{KNRR}
Markus Kirschmer, Fabien Narbonne, Christophe Ritzenthaler, and Damien Robert,
  \emph{Spanning the isogeny class of a power of an elliptic curve}, Math.
  Comp. \textbf{91} (2021), no.~333, pp.~401--449.

\bibitem{klupau05}
J{\"u}rgen Kl{\"u}ners and Sebastian Pauli, \emph{Computing residue class rings
  and {P}icard groups of orders}, J. Algebra \textbf{292} (2005), no.~1,
  47--64.

\bibitem{LevyWiegand85}
Lawrence Levy and Roger Wiegand, \emph{Dedekind-like behavior of rings with
  {$2$}-generated ideals}, J. Pure Appl. Algebra \textbf{37} (1985), no.~1,
  pp.~41--58.

\bibitem{lmfdb}
The {LMFDB Collaboration}, \emph{The l-functions and modular forms database},
  \url{http://www.lmfdb.org}, 2013, [Online; accessed 12 September 2024].

\bibitem{IsomClAbVarFqCommEndAlg}
Stefano Marseglia, \url{https://github.com/stmar89/IsomClAbVarFqCommEndAlg},
  The examples in this paper were computed using the code at commit
  c25be473adfeb1dba9932d47961e54649889fa78.

\bibitem{MarBassPow}
\bysame, \emph{Computing abelian varieties over finite fields isogenous to a
  power}, Res. number theory \textbf{5} (2019), no.~4, 35.

\bibitem{MarICM18}
\bysame, \emph{Computing the ideal class monoid of an order}, J. Lond. Math.
  Soc. (2) \textbf{101} (2020), no.~3, pp.~984--1007.

\bibitem{MarAbVar18}
\bysame, \emph{Computing square-free polarized abelian varieties over finite
  fields}, Math. Comp. \textbf{90} (2021), no.~328, pp.~953--971.

\bibitem{MarsegliaType22}
\bysame, \emph{Cohen-{M}acaulay type of orders, generators and ideal classes},
  J. Algebra \textbf{658} (2024), pp.~247--276.

\bibitem{Mar23_Loc}
\bysame, \emph{Local isomorphism classes of fractional ideals of orders in
  \'etale algebras}, J. Algebra \textbf{673} (2025), pp.~77--102.

\bibitem{MarModules25}
\bysame, \emph{Modules over orders, conjugacy classes of integral matrices, and
  abelian varieties over finite fields}, Res. Number Theory \textbf{11} (2025),
  no.~1, 16 (English), Id/No 27.

\bibitem{OswalShankar19}
Abhishek Oswal and Ananth Shankar, \emph{Almost ordinary abelian varieties over
  finite fields}, Journal of the London Mathematical Society \textbf{101}
  (2020), no.~3, pp.~923--937.

\bibitem{Reiner03}
Irving Reiner, \emph{Maximal orders}, London Mathematical Society Monographs.
  New Series, vol.~28, The Clarendon Press, Oxford University Press, Oxford,
  2003.

\bibitem{Roggenkamp70II}
Klaus Roggenkamp, \emph{Lattices over orders. {II}}, Lecture Notes in
  Mathematics, vol. 142, Springer-Verlag, Berlin-New York, 1970.

\bibitem{Roggenkamp70I}
Klaus Roggenkamp and Verena Huber-Dyson, \emph{Lattices over orders. {I}},
  Lecture Notes in Mathematics, vol. 115, Springer-Verlag, Berlin-New York,
  1970.

\bibitem{Schoof87}
Ren{\'e} Schoof, \emph{Nonsingular plane cubic curves over finite fields}, J.
  Combin. Theory Ser. A \textbf{46} (1987), no.~2, pp.~183--211.

\bibitem{SerreLocalFields}
Jean-Pierre Serre, \emph{Local fields}, Graduate Texts in Mathematics, vol.~67,
  Springer-Verlag, New York-Berlin, 1979, Translated from the French by Marvin
  Jay Greenberg.

\bibitem{psh08}
Peter Stevenhagen, \emph{The arithmetic of number rings}, Algorithmic number
  theory: lattices, number fields, curves and cryptography, Math. Sci. Res.
  Inst. Publ., vol.~44, Cambridge Univ. Press, Cambridge, 2008, pp.~209--266.

\bibitem{Tate66}
John Tate, \emph{Endomorphisms of abelian varieties over finite fields},
  Invent. Math. \textbf{2} (1966), pp.~134--144.

\bibitem{JVQuat}
John Voight, \emph{Quaternion algebras}, Graduate Texts in Mathematics, vol.
  288, Springer, Cham, 2021.

\bibitem{Wat69}
William Waterhouse, \emph{Abelian varieties over finite fields}, Ann. Sci.
  \'Ecole Norm. Sup. (4) \textbf{2} (1969), pp.~521--560.

\bibitem{xue2017counting}
Jiang~Wei Xue and Chia-Fu Yu, \emph{On counting certain abelian varieties over
  finite fields}, Acta Math. Sin. (Engl. Ser.) \textbf{37} (2021), no.~1,
  pp.~205--228.

\bibitem{chiafu2010}
Chia-Fu Yu, \emph{Simple mass formulas on {S}himura varieties of {PEL}-type},
  Forum Math. \textbf{22} (2010), no.~3, pp.~565--582.

\bibitem{chiafu2012}
\bysame, \emph{Superspecial abelian varieties over finite prime fields}, J.
  Pure Appl. Algebra \textbf{216} (2012), no.~6, pp.~1418--1427.

\end{thebibliography}

\end{document}